\documentclass[12pt]{amsart}
\usepackage[margin=0.5in]{geometry}
\newcommand{\R}{\mathbb{R}}
\newcommand{\Z}{\mathbb{Z}}
\newcommand{\N}{\mathbb{N}}

\newcommand{\C}{\mathbb{C}}

\renewcommand{\epsilon}{\varepsilon}

\usepackage[english]{babel}
\usepackage{hyperref}

\usepackage[alphabetic]{amsrefs}
\usepackage{amsmath,amssymb,amsfonts,amsthm,enumerate}
\usepackage{url}
\usepackage{graphicx,epstopdf,color}
\usepackage{mathtools}

\numberwithin{equation}{section}

\newtheorem{theorem}{Theorem}[section]
\newtheorem{lemma}[theorem]{Lemma}

\newtheorem{example}[theorem]{Example}

\renewcommand{\leq}{\leqslant}
\renewcommand{\le}{\leqslant}

\renewcommand{\ge}{\geqslant}
\allowdisplaybreaks

\title[A two-dimensional Frenkel-Kontorova model]{A quantitative rigidity result\\
for a two-dimensional\\
Frenkel-Kontorova model}

\author{Serena Dipierro}
\address{Serena Dipierro: Department of Mathematics and Statistics, The University of Western Australia, 35 Stirling Highway, Crawley, Perth, WA 6009, Australia}
\email{serena.dipierro@uwa.edu.au}

\author{Giorgio Poggesi}
\address{Giorgio Poggesi: Department of Mathematics and Statistics, The University of Western Australia, 35 Stirling Highway, Crawley, Perth, WA 6009, Australia}
\email{giorgio.poggesi@uwa.edu.au}

\author{Enrico Valdinoci}
\address{Enrico Valdinoci: Department of Mathematics and Statistics, The University of Western Australia, 35 Stirling Highway, Crawley, Perth, WA 6009, Australia}
\email{enrico.valdinoci@uwa.edu.au}

\begin{document}

\begin{abstract}
We consider
a Frenkel-Kontorova system of harmonic oscillators in a two-dimensional Euclidean lattice
and we obtain a quantitative estimate on the angular function of the equilibria.
The proof relies on a PDE method related to a classical conjecture by E. De Giorgi,
also in view of an elegant technique based on complex variables that was introduced by A. Farina.

In the discrete setting, a careful analysis of the reminders is needed to exploit
this type of methodologies inspired by continuum models.\end{abstract}


\maketitle

\section{Introduction and statement of the main result}

In~\cite{MR0001169}, Yakov Frenkel and Tatiana Kontorova introduced a simple, but very
effective, model
to describe the atom dislocation dynamics of a crystal lattice.
The model takes into account
a pattern of particles with harmonic nearest neighbor interactions
and subject to a substrate potential
(in its simplest form, the potential
is a periodic trigonometric function, but more general forcing terms
can be also taken into account).

The simplest expression of the model
by Frenkel and Kontorova
consists
of a harmonic chain of atoms of unit mass in a sinusoidal potential.
The atoms are supposed to be at some (small) distance~$h$
the ones from the others;
hence, for simplicity, in dimension~$1$,
we can consider the location
of the atoms at rest to be described by the lattice~$h\Z$.
The displacement~$u_i(t)$, for each~$i\in h\Z$ and~$t\in\R$,
describes the evolution of such a
harmonic oscillator subject to
nearest neighbor interactions (with Hooke constant~$d>0$)
and the sinusoidal potential according
to the equation
\begin{equation}\label{FKPAR}
{\displaystyle {\ddot u_i}+\sin u_i-\frac{d}{h^2}\,\big(
u_{i+1}+u_{i-1}-2u_{i}\big)=0}.
\end{equation}
Equilibrium configurations,
i.e., stationary solutions
of~\eqref{FKPAR},
are therefore obtained from the equation
\begin{equation}\label{FKPAR-2}
{\displaystyle \sin u_i-\frac{d}{h^2}\,\big(
u_{i+1}+u_{i-1}-2u_{i}\big)=0}.
\end{equation}
Natural generalizations of~\eqref{FKPAR-2}
occur by considering the rest positions of the atoms
in a plane (see Figure~\ref{FIG-FK})
and more general potentials than the sinusoid, possibly depending
also on the position,
in which case~\eqref{FKPAR-2} is replaced by the more general form
\begin{equation}\label{DGEQ0} \sum_{j=1}^2 \frac{ u_{i+he_j} + u_{i-he_j} -2u_i}{h^2}=f(i,u_i),\qquad{\mbox{for all }}i\in h\Z^2,\end{equation}
being~$e_1:=(1,0)$ and~$e_2:=(0,1)$.
The detailed and simple mechanical interpretation
of~\eqref{DGEQ0} is
given by a system of particles constrained to move
in the space along a vertical track, see Figure~\ref{FIG-FK}.
More precisely, one assumes that
the tracks are equally distributed in a square pattern,
namely the intersection of the tracks and the ground
level corresponds to the lattice~$h\Z^2$.
One also assumes that the closest particles
are connected by elastic springs, say with Hooke constant
equal to~$d/h^2$. Moreover, the particles
are subject to a potential which depends
on the height of the particle and on the position
of the vertical track along which the particle moves.
In this way, if~$u_i$ is the height
of the particle located on the vertical
trail placed at~$i\in h\Z^2$, we denote by~$V(i, u_i)$
the corresponding external potential.

\begin{figure}[h!]
\includegraphics[scale=0.45]{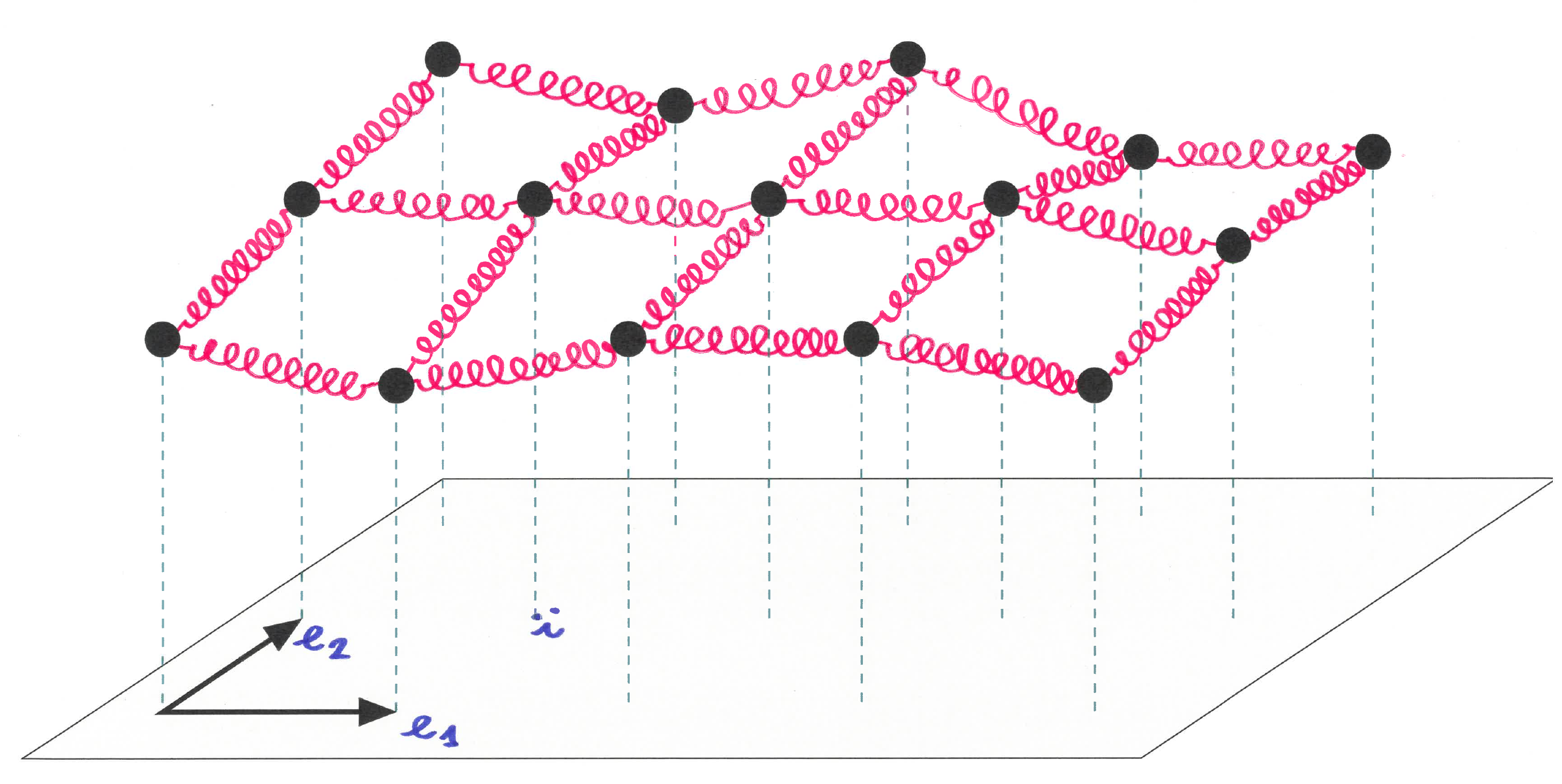}
\caption{Mechanical interpretation of~\eqref{DGEQ0}.}\label{FIG-FK}
\end{figure}

The total energy of this system is given by the formal series
\begin{equation}\label{LO:20ef} {\mathcal{E}}(u):=\frac{d}{2h^2}\sum_{{i\in h\Z^2}\atop{j\in\{1,2\}}} |v_{i+he_j}-v_i|^2
+\sum_{i\in h\Z^2} V(i, u_i),\end{equation}
where~$v_i:=(i,u_i)$ and~$u=\{u_i\}_{i\in h\Z^2}$.

Observing that
$$ |v_{i+he_j}-v_i|^2=
(u_{i+he_j}-u_i)^2+|(i+he_j)-i|^2=
(u_{i+he_j}-u_i)^2+h^2,$$
and the latter term is independent on the configuration,
the equilibria of~\eqref{LO:20ef}
coincide with those of
\begin{equation}\label{LO:20ef1} {\mathcal{F}}(u):=\frac{d}{2h^2}\sum_{{i\in h\Z^2}\atop{j\in\{1,2\}}} (u_{i+he_j}-u_i)^2
+\sum_{i\in h\Z^2} V(i, u_i).\end{equation}
The equilibria of~\eqref{LO:20ef1}
are found by considering the critical points for compact perturbations,
leading to the equation
$$ 0=\langle D{\mathcal{F}}(u), \,\varphi\rangle=\frac{d}{h^2}\sum_{{i\in h\Z^2}\atop{j\in\{1,2\}}} (u_{i+he_j}-u_i)(\varphi_{i+he_j}-\varphi_i)
+\sum_{i\in h\Z^2} \partial_{u_i} V(i, u_i)\varphi_i,$$
for every~$\varphi=\{\varphi_i\}_{i\in h\Z^2}$
such that~$\{\varphi_i\ne0\}$ is a finite set.

Writing
\begin{eqnarray*}
&&\sum_{{i\in h\Z^2}\atop{j\in\{1,2\}}} (u_{i+he_j}-u_i)(\varphi_{i+he_j}-\varphi_i)
\\&=&
\sum_{{i\in h\Z^2}\atop{j\in\{1,2\}}} (u_{i+he_j}-u_i)\varphi_{i+he_j}
-
\sum_{{i\in h\Z^2}\atop{j\in\{1,2\}}} (u_{i+he_j}-u_i)\varphi_i\\
&=&
\sum_{{i\in h\Z^2}\atop{j\in\{1,2\}}} (u_{i}-u_{i-he_j})\varphi_{i}
-
\sum_{{i\in h\Z^2}\atop{j\in\{1,2\}}} (u_{i+he_j}-u_i)\varphi_i\\
&=&
\sum_{{i\in h\Z^2}\atop{j\in\{1,2\}}} (2u_{i}-u_{i+he_j}-u_{i+he_j})\varphi_i
,\end{eqnarray*}
we see that the equilibrium configurations satisfy
\begin{equation}\label{LDTG}\frac{d}{h^2}\,\sum_{j=1}^2(2u_{i}-u_{i+he_j}-u_{i+he_j})
+
\partial_{u_i} V(i, u_i)=0,\end{equation}
which is precisely of the form given in~\eqref{DGEQ0}.\medskip

In this paper, we will provide some ``approximate
symmetry'' results for solutions of~\eqref{DGEQ0}
under suitable structural assumptions. Roughly speaking,
we will consider the angular function of the solution
(i.e., the phase of the discrete increment of the solution)
and control its weighted $\ell^2$-norm
by the small parameter~$h$ (and suitable structural constants).
The weight function of such $\ell^2$-norm
will also have a concrete meaning in the model,
being the square of the discrete increment of the solution
(the precise estimate will be formally stated in~\eqref{FORM}
below). In the formal limit~$h\searrow0$,
estimates of this kind would entail a one-dimensional
symmetry for the solution, yielding that the two-dimensional
equilibrium configuration can be in fact represented
by a one-dimensional function in some direction
and the level sets of the solution are all straight lines.
In this spirit, our result can be seen as a quantitative
estimate on ``how far the equilibrium
configuration is from being one-dimensional''
in the discrete setting
and gives an optimal estimate on the perturbative effect
played by the spatial parameter~$h$.
\medskip

We also point out that,
in terms of atom dislocation theory, the
Frenkel-Kontorova model can be considered as an ``atomistic''
description which can be rigorously related to the ``microscopic''
description of hybrid type given by the
Peierls-Nabarro model, see~\cite{MR2852206}.
See also~\cite{MR2035039} for a throughout presentation of
the Frenkel-Kontorova model and for the detailed discussions
of several applications to fields different than the theory of crystal
dislocation (including, among the others,
absorption, crowdions, magnetically ordered structures, Josephson junctions,
hydrogen-bonded and DNA chains).
In the dynamical systems setting, the equilibria
of the Frenkel-Kontorova model with sinusoidal layer potential
give rise to the
Chirikov-Taylor map, and the
continuum-limit is the sine-Gordon equation.\medskip

Given its constructive importance in the theory of crystal dislocation,
its flexibility in a number of different applications, and its strong link
to problems in dynamical systems and differential equations,
the Frenkel-Kontorova model has been widely studied in the literature
under different perspectives, and it has become a classical topic
in several branches of statistical mechanics and in the analysis of harmonic oscillators
on lattices, see e.g.~\cite{MR719055, MR719634, MR766107, MR1620543, MR2356117, MR3038682, MR3663614, MR3725364, MR3912645, MR4015338, MR4028786, MR4102235}.\medskip

The point of view that we take in this article aims at
describing the monotone solutions of a two-dimensional
Frenkel-Kontorova model. The results obtained will be valid
for every type of layer potential and their proof will rely on a number of analytical
methods inspired by a classical conjecture by Ennio De Giorgi, see~\cite{MR533166}
(see also~\cite{MR1775735, MR1637919, MR1655510, MR2480601, MR2757359, MR2728579, MR3488250}
for several positive results in the direction of such conjecture,
\cite{MR2473304}
for a counterexample in high dimension,
and~\cite{MR2528756}
for a survey on this topic).\medskip

Given~$h\in(0,1]$ and~$i\in h\Z^2$, for any~$u:h\Z^2\to\R$, we set
\begin{equation}\label{ELLE} {\mathcal{L}} u_i:=
\sum_{j=1}^2 \frac{ u_{i+he_j} + u_{i-he_j} -2u_i}{h^2}.\end{equation}
We observe that,
as~$h\searrow0$, the operator~$\mathcal{L}$ recovers the usual
Laplace operator.

Our main objective here is to consider solutions~$u:h\Z^2\to\R$
of the equation (as in~\eqref{DGEQ0})
\begin{equation}\label{DGEQ}
{\mathcal{L}} u_i=f(i, u_i)\qquad{\mbox{ for all }}i\in h\Z^2.
\end{equation}
We assume that $f: h \Z^2 \times \R \to \R$ satisfies the following condition: there exist a function $L_f^+: h \Z^2 \times \R \to \R $ and a finite positive constant $\kappa_0^+$ such that
\begin{equation}\label{hp:GENERALsemilinearitaconspazio}
\sum_{j=1}^{2} \frac{ \big| f(i+he_j, u_{i+he_j})-f(i, u_i)- L_f^+ (i, u_i)(u_{i+he_j}-u_i)\big| }{h} \le \kappa_0^+ \,  h , \quad \text{for all } \, i \in h\Z^2 . 
\end{equation}

A significant particular case is given by a function $f: h \Z^2 \times \R \to \R$ which is of class $C^1$ in the real variable, that is, 
\begin{equation}\label{hp:regularitysemilinearity}
f(i, \cdot) \in C^1 (\R) , \quad \text{for all } \, i \in h \Z^2 ,
\end{equation}  
and satisfies \eqref{hp:GENERALsemilinearitaconspazio} with $L_f^+ (i,u_i) = f'(i,u_i)$, where $f'$ denotes the derivative of $f$ with respect to the real variable. In this case, \eqref{hp:GENERALsemilinearitaconspazio} reads as
\begin{equation}\label{hp:semilinearitaconspazio}
\sum_{j=1}^{2} \frac{ \big| f(i+he_j, u_{i+he_j})-f(i, u_i)- f' (i, u_i)(u_{i+he_j}-u_i)\big| }{h} \le \kappa_0^+ \,  h , \quad \text{for all } \, i \in h\Z^2 .
\end{equation}
We observe that assumption~\eqref{hp:semilinearitaconspazio}
states, in a quantitative way, that the dependence on the
site of the nonlinearity~$f$ is ``negligible'' (namely,
the nonlinearity ``mostly'' depends on the state parameter~$u_i$,
rather than on~$i$), see also page~\pageref{LI0} for additional comments on this
assumption in comparison with the continuous framework.

In this setting, we introduce the following notation:
for every~$j\in\{1,2\}$ and any~$i \in h\Z^2$, we let
\begin{equation}\label{DER} {\mathcal{D}}_j^+ u_i:=\frac{ u_{i+he_j} -u_i}{h}
\qquad{\mbox{and}}\qquad {\mathcal{D}}_j^- u_i:=
\frac{ u_i-u_{i-he_j}}{h}.
\end{equation}
The operators in~\eqref{DER} can be seen as discrete increments
that converge to the standard derivative
as~$h\searrow0$.

We will suppose that the solution~$u$ of~\eqref{DGEQ}
satisfies some structural assumptions that we now describe in
detail. Our main assumption is that
\begin{equation}\label{hp:graddiversodazero}
\sum_{j=1}^2 |u_{i+h e_j} - u_i|^2 > 0 , \qquad{\mbox{for all }}i\in h\Z^2.
\end{equation}
Condition \eqref{hp:graddiversodazero} requires that the ``squared norm of the gradient'' $ \sum_{{1\le j\le 2}} | {\mathcal{D}}_j^+ u_i|^2$ does not vanish, for all $i \in h \Z^2$.

In addition, we assume that
\begin{equation}\label{BOULI}
\kappa_1^+:=\sup_{{i\in h\Z^2}\atop{1\le j\le2}}|{\mathcal{D}}_j^+ u_i|<+\infty.
\end{equation}
Roughly speaking, one can consider~\eqref{BOULI} as
a ``Lipschitz'' assumption on the solution~$u$.

Following~\cite{MR2014827}, it is convenient
to use a complex variable notation, identifying~$\R^2$ with~$\C$.
To this end, we set~${\mathcal{I}}:=\sqrt{-1}$ and
\begin{equation} \label{1.7BIS} {U_i^+}:=
{\mathcal{D}}_1^+ u_i+{\mathcal{I}}\,
{\mathcal{D}}_2^+ u_i=\frac{
\big(u_{i+he_1} -u_i\big)+{\mathcal{I}}\,
\big(u_{i+he_2} -u_i\big)}h.
\end{equation}

Since by \eqref{hp:graddiversodazero} we have that $|{U_i^+}| > 0$ for any $i \in h\Z^2$, then $ {U_i^+} / |{U_i^+}|$ is a function from $h \Z^2$ to the unit sphere $\mathcal{S}^1$ of $\R^2$.
Thus, using a polar representation, there exists
\begin{equation}\label{PIKS}
\vartheta^+_i \in \left( - \pi , \pi \right]
\end{equation}
such that
$${U_i^+}=\rho_i^+\, e^{{\mathcal{I}}\vartheta^+_i} ,$$
where 
$$
\rho_i^+ := |{U_i^+}|.
$$
In light of~\eqref{BOULI}, we also know that
\begin{equation}\label{D23ecs824}
\sup_{i\in h\Z^2} \rho_i^+=\sup_{i\in h\Z^2}|{U_i^+}|=
\sup_{i\in h\Z^2}\big( \sqrt{ |{\mathcal{D}}_1^+ u_i|^2 +|{\mathcal{D}}_2^+ u_i|^2 }\big)\le
\sqrt{2} \sup_{{i\in h\Z^2}\atop{1\le j\le2}} |{\mathcal{D}}_j^+ u_i|= \sqrt{2} \, \kappa_1^+ .
\end{equation}

We will now state some regularity assumptions on~$\vartheta^+$
and~$\rho$. 
First of all, we take some integrability hypotheses, supposing that
\begin{equation}\label{K2}
\kappa_2^+:=
\sum_{{1\le j\le 2}\atop{i\in h\Z^2}}
(\rho_i^+)^2 \,\Big(
|{\mathcal{D}}_j^+\vartheta^+_i|\,|{\mathcal{D}}_j^+ ({\mathcal{D}}_j^+ \vartheta^+)_{i-he_j}|+
|{\mathcal{D}}_j^-\vartheta^+_i|\,|{\mathcal{D}}_j^- ({\mathcal{D}}_j^- \vartheta^+)_{i+he_j}|
\Big)<+\infty.
\end{equation}
We observe that~\eqref{K2} is satisfied provided that the angular function~$\vartheta^+$
``does not oscillate'' too much at infinity. We now take additional assumptions
in this spirit by supposing that~$\vartheta^+$ is suitably close to a limit angle
at infinity. To this end, for all~$j\in\{1,2\}$ and all~$i\in h\Z^2$,
we 
introduce the notation
\begin{equation}\label{0ok8uh7gfr8} {\mathcal{L}}_j {u}_i:=
\frac{ u_{i+he_j} + u_{i-he_j} -2u_i}{h^2}.\end{equation}
Notice that
\begin{equation}\label{2.2BIS}
{\mathcal{L}}:={\mathcal{L}}_1+{\mathcal{L}}_2.\end{equation}
We assume that there exists~$\vartheta^+_\infty \in \left( -\pi,\pi \right] $
such that the following assumptions hold true:
\begin{equation}\label{K33}\begin{split}&\kappa_3^+:=
\sum_{{1\le j\le2}\atop{i\in h\Z^2}}
(\rho_i^+)^2\,
\Big( |{\mathcal{D}}_j^+ \vartheta^+_i|^3+|{\mathcal{D}}_j^- \vartheta^+_i|^3\Big)\,|\vartheta^+_i-\vartheta^+_\infty|<+\infty,
\\&
\kappa_4^+:=\sum_{{1\le j\le2}\atop{i\in h\Z^2}}\rho_i^+\,
\big( |{\mathcal{D}}_j^+\rho_i^+|\;|{\mathcal{D}}_j^+ \vartheta^+_i|^2
+
|{\mathcal{D}}_j^-\rho_i^+|\;|{\mathcal{D}}_j^- \vartheta^+_i|^2
\big)\,|\vartheta^+_i-\vartheta^+_\infty| <+\infty
\\&
\kappa_5^+:=
\kappa_0^+ \, \sum_{i\in h\Z^2}\rho_i^+ \,|\vartheta^+_i-\vartheta^+_\infty|<+\infty,
\\&
\kappa_6^+:=
\sum_{{1\le j\le2}\atop{i\in h\Z^2}}\Big(
|{\mathcal{D}}_j^+(\rho_i^+)^2|\,|{\mathcal{D}}_j^+({\mathcal{D}}_j^+\vartheta^+)_i|+
|{\mathcal{D}}_j^-(\rho_i^+)^2|\,|{\mathcal{D}}_j^-({\mathcal{D}}_j^-\vartheta^+)_i|+
h(\rho_i^+)^2\,|{\mathcal{L}}^2_j\vartheta^+_{i}|\Big)\,|\vartheta^+_i-\vartheta^+_\infty| < +\infty ,
\\&
\kappa_7^+:=
\sum_{{1\le j\le2}\atop{i\in h\Z^2}}\Big(
|{\mathcal{D}}_j^+\rho_i^+|^2\,
|{\mathcal{D}}_j^+\vartheta^+_i|
+|{\mathcal{D}}_j^-\rho_i^+|^2\,|{\mathcal{D}}_j^-\vartheta^+_i|\Big)\,|\vartheta^+_i-\vartheta^+_\infty| < +\infty .
\end{split}
\end{equation}

In this setting, our main result\footnote{We observe that the quantities~$\kappa^\pm_m$,
with~$m\in\{0,\dots,7\}$ will be taken to be finite, otherwise the estimates of the main results
would be trivial, possibly with a right-hand side equal to infinity; in any case,
these quantities are not assumed to be bounded uniformly in~$h$.

However, when these quantities happen to be bounded uniformly in~$h$, the estimate in~\eqref{FORM}
becomes particularly significant, since it bounds the mismatch between the discrete and continuous models
linearly in the mesh parameter~$h$.} here is as follows:

\begin{theorem}\label{DGDG}
Let $f: h \Z^2 \times \R \to \R$ satisfy  \eqref{hp:GENERALsemilinearitaconspazio}.
Let~$u:h\Z^2\to\R$ be a solution of~\eqref{DGEQ},
satisfying~\eqref{hp:graddiversodazero},
\eqref{BOULI}, \eqref{K2},
and~\eqref{K33}.

Then,
\begin{equation}\label{FORM}
\sum_{{1\le j\le 2}\atop{i\in h\Z^2}}
(\rho_i^+)^2 \;\Big(
|{\mathcal{D}}_j^+\vartheta^+_i|^2+|{\mathcal{D}}_j^-\vartheta^+_i|^2
\Big)\le C h ,
\end{equation}
where~$C>0$ is given by
\begin{equation}\label{eq:constant C}
C:= 4 \left( \kappa_2^++ 2 e^{2\pi}\kappa_3^++ 2 e^{2\pi}\kappa_4^++2 \kappa_5^+ + \kappa_6^+ + \kappa_7^+ \right) .
\end{equation}
\end{theorem}

As a variant of Theorem~\ref{DGDG}, one can also consider~${U_i^-}:=
{\mathcal{D}}_1^- u_i+{\mathcal{I}}\,
{\mathcal{D}}_2^- u_i$ and~$\rho_i^-:= |{U_i^-}|$.
Thus, under the assumption that
\begin{equation}\label{hp:graddiversodazero-II}
\sum_{j=1}^2 |u_{i-h e_j} - u_i|^2 > 0 , \qquad{\mbox{for all }}i\in h\Z^2,
\end{equation}
we have that~$\rho_i^-\ne0$ for every~$i\in h\Z^2$, whence we can define~$\vartheta^-_i \in \left( - \pi , \pi \right]$
such that~${U_i^-}=\rho_i^-\, e^{{\mathcal{I}}\vartheta^-_i}$.
To obtain a full counterpart of Theorem~\ref{DGDG} that takes into account ``both positive and negative
increments'' it is convenient to complement
\eqref{hp:GENERALsemilinearitaconspazio} with the assumption that there exist a function $L_f^-: h \Z^2 \times \R \to \R $ and a constant $\kappa_0^- > 0$ such that
\begin{equation}\label{hp:GENERAL CON MENO semilinearitaconspazio}
\sum_{j=1}^{2} \frac{ \big| f(i, u_i) - f(i-he_j, u_{i-he_j}) - L_f^- (i, u_i)(u_i - u_{i-he_j})\big| }{h} \le \kappa_0^- \,  h , \quad \text{for all } \, i \in h\Z^2 . 
\end{equation}
Analogously, assumptions~\eqref{BOULI}, \eqref{K2}, and~\eqref{K33} will be combined with the following conditions:
\begin{equation}\label{COND-KA-II}
\begin{split}&
\kappa_1^-:=\sup_{{i\in h\Z^2}\atop{1\le j\le2}}|{\mathcal{D}}_j^- u_i|<+\infty,\\&
\kappa_2^-:=
\sum_{{1\le j\le 2}\atop{i\in h\Z^2}}
(\rho_i^-)^2 \,\Big(
|{\mathcal{D}}_j^+\vartheta^-_i|\,|{\mathcal{D}}_j^+ ({\mathcal{D}}_j^+ \vartheta^-)_{i-he_j}|+
|{\mathcal{D}}_j^-\vartheta^-_i|\,|{\mathcal{D}}_j^- ({\mathcal{D}}_j^- \vartheta^-)_{i+he_j}|
\Big)<+\infty,\\
&\kappa_3^-:=
\sum_{{1\le j\le2}\atop{i\in h\Z^2}}
(\rho_i^-)^2\,
\Big( |{\mathcal{D}}_j^+ \vartheta^-_i|^3+|{\mathcal{D}}_j^- \vartheta^-_i|^3\Big)\,|\vartheta^-_i-\vartheta^-_\infty|<+\infty,
\\&
\kappa_4^-:=\sum_{{1\le j\le2}\atop{i\in h\Z^2}}\rho_i^-\,
\big( |{\mathcal{D}}_j^+\rho_i^-|\;|{\mathcal{D}}_j^+ \vartheta^-_i|^2
+
|{\mathcal{D}}_j^-\rho_i^-|\;|{\mathcal{D}}_j^- \vartheta^-_i|^2
\big)\,|\vartheta^-_i-\vartheta^-_\infty| <+\infty
\\&
\kappa_5^-:=
\kappa_0^- \, \sum_{i\in h\Z^2}\rho_i^- \,|\vartheta^-_i-\vartheta^-_\infty|<+\infty,
\\&
\kappa_6^-:=
\sum_{{1\le j\le2}\atop{i\in h\Z^2}}\Big(
|{\mathcal{D}}_j^+(\rho_i^-)^2|\,|{\mathcal{D}}_j^+({\mathcal{D}}_j^+\vartheta^-)_i|+
|{\mathcal{D}}_j^-(\rho_i^-)^2|\,|{\mathcal{D}}_j^-({\mathcal{D}}_j^-\vartheta^-)_i|+
h(\rho_i^-)^2\,|{\mathcal{L}}^2_j\vartheta^-_{i}|\Big)\,|\vartheta^-_i-\vartheta^-_\infty| < +\infty ,
\\&
\kappa_7^-:=
\sum_{{1\le j\le2}\atop{i\in h\Z^2}}\Big(
|{\mathcal{D}}_j^+\rho_i^-|^2\,
|{\mathcal{D}}_j^+\vartheta^-_i|
+|{\mathcal{D}}_j^-\rho_i^-|^2\,|{\mathcal{D}}_j^-\vartheta^-_i|\Big)\,|\vartheta^-_i-\vartheta^-_\infty| < +\infty .
\end{split}\end{equation}
for some~$\vartheta^-_\infty \in \left( -\pi,\pi \right] $.
Then, we set~$\kappa_m:=\kappa_m^++\kappa_m^-$ for all~$m\in\{0,\dots,7\}$ and we have the following rigidity result:

\begin{theorem}\label{DGDG-GE}
Let $f: h \Z^2 \times \R \to \R$ satisfy \eqref{hp:GENERALsemilinearitaconspazio} and \eqref{hp:GENERAL CON MENO semilinearitaconspazio}.

Let~$u:h\Z^2\to\R$ be a solution of~\eqref{DGEQ},
satisfying~\eqref{hp:graddiversodazero},
\eqref{BOULI}, \eqref{K2}, \eqref{K33},
\eqref{hp:graddiversodazero-II}, and~\eqref{COND-KA-II}.

Then,
\begin{equation}\label{TUTTA}
\sum_{{1\le j\le 2}\atop{i\in h\Z^2}}\Big[
(\rho_i^+)^2 \;\Big(
|{\mathcal{D}}_j^+\vartheta^+_i|^2+|{\mathcal{D}}_j^-\vartheta^+_i|^2
\Big)
+(\rho_i^-)^2 \;\Big(
|{\mathcal{D}}_j^+\vartheta^-_i|^2+|{\mathcal{D}}_j^-\vartheta^-_i|^2
\Big)\Big]\le C h ,
\end{equation}
where~$C>0$ is given by
\begin{equation*}
C:= 4 \left( \kappa_2 + 2 e^{2\pi}\kappa_3+ 2 e^{2\pi}\kappa_4+2 \kappa_5 + \kappa_6 + \kappa_7 \right) .
\end{equation*}
\end{theorem}

We observe that estimates~\eqref{FORM} and~\eqref{TUTTA} provide quantitative rigidity results. Indeed,
in the formal limit as~$h\searrow0$, if $C h$ tends to $0$ then the quantities	
\begin{equation}\label{VAN} \sum_{{1\le j\le 2}\atop{i\in h\Z^2}}
(\rho_i^\pm)^2 \;\Big(
|{\mathcal{D}}_j^+\vartheta^\pm_i|^2+|{\mathcal{D}}_j^-\vartheta^\pm_i|^2
\Big)\end{equation}
become infinitesimal. In particular, the vanishing of~\eqref{VAN}
would correspond to a constant direction of the gradient (in all regions where the gradient itself does
not vanish). See Lemma~\ref{KK-1D} for a precise formulation of
the one-dimensional symmetry property related
to these conditions.

Theorem~\ref{DGDG-GE} is a perfect counterpart of Theorem~\ref{DGDG}, therefore
in this paper we will mostly focus on the first of these results.

We observe that, for a given~$h\in(0,1]$ (i.e., even without taking limits),
a simple byproduct of~\eqref{FORM} and of the fact, due to~\eqref{hp:graddiversodazero}, that~$(\rho_i^+)^2 = h^{-2} \sum\limits_{j=1}^2 |u_{i+h e_j} - u_i|^2 >0$ is that, for all~$i\in h\Z^2$ and all~$j\in\{1,2\}$,
$$ |\vartheta^+_{i+he_j}-\vartheta^+_i|^2=h^2|{\mathcal{D}}_j^+\vartheta^+_i|^2\le\frac{Ch^3}{(\rho_i^+)^2} =
\frac{C h^5}{\sum\limits_{j=1}^2 |u_{i+h e_j} - u_i|^2},
$$
which gives an explicit bound on the discrete variation of the angular function
in terms of the size of the lattice and the assumption in~\eqref{hp:graddiversodazero}.

In this spirit, we mention that assumptions~\eqref{K2} and~\eqref{K33}
are, a-posteriori, consistent with the (approximate) constancy of the angular function,
in the sense that if~$\vartheta^+$ is constant, then conditions~\eqref{K2} and~\eqref{K33}
are obviously fulfilled.

We point out that the summability conditions
in~\eqref{K2} and~\eqref{K33}
are specific for the discrete case and do not have a clear
counterpart in the continuous case. As a matter of fact,
roughly speaking, at a formal level, the quantities introduced
in~\eqref{K2} and~\eqref{K33} are multiplied by~$h$
in~\eqref{FORM} and therefore this product
formally disappears in the continuous limit.

On the other hand,
one could also consider a continuous analogue of the discrete
conditions in~\eqref{K2} and~\eqref{K33}
simply by replacing increments by derivatives
and sums with integrals: this formal passage to the limit would
correspond to several integrability conditions which, as far as we
are aware of, do not appear in the literature related to symmetry
properties of semilinear elliptic partial differential equations.
However all these integrability conditions would be obviously
satisfied by one-dimensional solutions (since the corresponding
phase of the gradient would be constant in space), therefore,
a posteriori, these conditions
do not trivialize the space of solutions in the continuous setting.
\medskip

We observe that Theorems \ref{DGDG} and~\ref{DGDG-GE} possess a neat mechanical interpretation
according to Figure~\ref{FIG-FK}. Indeed,
recalling~\eqref{LDTG},
potentials of particular interests are the ones only depending
on the height (say, of the form~$V(i, u_i)=\hat{V} (u_i)$),
and for instance the gravitational potential is of this form.

And it is of course of particular interest to understand
equilibrium configurations when the tracks become denser and denser (that is for smaller and smaller~$h$).

Theorems \ref{DGDG} and~\ref{DGDG-GE}
address quantitatively this question,
by establishing that for ``very dense'' tracks
and potentials depending ``almost only on the height''
then equilibria are ``almost flat'' configurations,
in the sense that their increments have ``almost constant''
direction in the plane -- the precise quantification
of this rough statement being given by the bound in~\eqref{FORM}.
\medskip

We also provide an observation of geometric flavor, stating that
when the left-hand side of~\eqref{TUTTA} vanishes identically the function~$u:h\Z^2\to\R$ is
one-dimensional, in the sense that it can be reconstructed by a one-dimensional function~$\widetilde{u}:h\Z\to\R$
(and this also highlights the fact that Theorem~\ref{DGDG-GE}
can be seen as a one-dimensional, quantitative, stability result).

\begin{lemma}\label{KK-1D}
Let~$u:h\Z^2\to\R$ be such that
\begin{equation}\label{der-2se}
\mathcal{D}^+_2 u_i\ne0\qquad{\mbox{and}}\qquad\mathcal{D}^-_2 u_i\ne0
\end{equation}
for all~$i\in h\Z^2$,
and assume that
\begin{equation}\label{noinrthoch45}
\sum_{{1\le j\le 2}\atop{i\in h\Z^2}}\Big[
(\rho_i^+)^2 \;\Big(
|{\mathcal{D}}_j^+\vartheta^+_i|^2+|{\mathcal{D}}_j^-\vartheta^+_i|^2
\Big)
+(\rho_i^-)^2 \;\Big(
|{\mathcal{D}}_j^+\vartheta^-_i|^2+|{\mathcal{D}}_j^-\vartheta^-_i|^2
\Big)\Big]
=0.\end{equation}
Then, there exist~$\widetilde{u}:h\Z\to\R$, $c^+$, $c^-\in\R$, such that 
\begin{equation}\label{k8i-know}
u_{(hk,hm)}=\sum_{j=0}^{|k|} {\binom {|k|}{j}} (c^{\sigma_k})^j\,(1-c^{\sigma_k})^{|k|-j}\,\widetilde{u}_{h(m+\sigma_k j)},
\end{equation}
for every~$k$, $m\in\Z$, where
$$ \sigma_k:=\begin{cases}
+ & {\mbox{ if }}k>0,\\
-& {\mbox{ if }}k<0,\\
0& {\mbox{ if }}k=0,
\end{cases}$$
and we understand $c^{\sigma_0}=c^0=1$.

In addition,
\begin{equation}\label{Agfoi-doiut}
\frac{u_{i\pm he_1}-u_i}{u_{i\pm he_2}-u_i}=c^\pm
\qquad{\mbox{for all }}i\in h\Z^2.\end{equation}
\end{lemma}

We stress that~\eqref{k8i-know} states that the knowledge of the
one-dimensional function~$\widetilde{u}$ is sufficient for the complete knowledge of
the two-dimensional function~$u$. Moreover, the identity
in~\eqref{Agfoi-doiut} can be seen as the discrete counterparts
of continuous identities of the type~$\frac{\partial_1u}{\partial_2u}={\rm const}$
which characterizes smooth functions in~$\R^2$
whose level sets are parallel straight lines.

We also remark that~\eqref{k8i-know} can be seen, formally,
as a discrete analogue of the continuous identity
\begin{equation}\label{k8i-know-CONTIN}
u(x_1,x_2)=\widetilde{u}(cx_1+x_2),
\end{equation}
for some constant~$c$, which expresses the fact the
the function~$u:\R^2\to\R$ is actually
depending on one Euclidean variable in the direction~$(c,1)$:
see Appendix~\ref{APPEDK894}
for further comments on this formal relation.

There is however an interesting conceptual difference
between the discrete relation in~\eqref{k8i-know}
and its continuous counterpart in~\eqref{k8i-know-CONTIN}.
Indeed, in the continuous
case, the value of~$u$ at a given point, say~$(x_1,x_2)$,
is reconstructed by the knowledge of the value of
its one-dimensional representation~$\widetilde{u}$ at precisely
one specific point (namely, in light of~\eqref{k8i-know-CONTIN},
at the point~$cx_1+x_2$). Instead, in the discrete case, the value
of~$u$ at a given site, for instance~$(hk,hm)$ with~$k$, $m\in\Z$
and~$k>0$,
is reconstructed via~\eqref{k8i-know} by the values of
its one-dimensional representation~$\widetilde{u}$ at
several sites, namely~$(0,hm)$,
$(h,hm)$, $\dots$, $(hk,hm)$ (though of course this nonlocal
effect disappears in the continuous limit of infinitesimal~$h$).
\medskip

We emphasize that the symmetry results in the continuous case are necessarily obtained for potentials of the form $f(u)$ that only depend on $u$.
Here, the quantitative nature of our result allows to consider potentials of the form $f(i, u_i)$ that also depend on the position $i \in h \Z^2$.
Notice that, assumption \eqref{hp:semilinearitaconspazio} quantitatively controls the dependence of $f$ on $i$: in the formal limit as $h \searrow 0$, \eqref{hp:semilinearitaconspazio} informs us \label{LI0}
that such dependence tends to disappear, in accordance with the symmetry results known in the continuous case.

We also notice that assumptions \eqref{hp:regularitysemilinearity} and \eqref{hp:semilinearitaconspazio} (and hence in particular \eqref{hp:GENERALsemilinearitaconspazio}) are always satisfied by any semilinearity of the form $f(i,u_i)= \hat{f}(u_i)$ (only depending on $u$), provided that 
\begin{equation}\label{eq:assumptionsufchenondipendedaspazio}
\hat{f} \in C^2(\R) \quad \text{ and } \quad \|\hat{f}''\|_{L^\infty(\R)}<+\infty .
\end{equation}
Indeed, in this case \eqref{hp:regularitysemilinearity} trivially holds true. Moreover, we have
\begin{equation*}
\begin{split}
& \sum_{j=1}^{2} \frac{ \big| f(i+he_j, u_{i+he_j})-f(i, u_i)- f' (i, u_i)(u_{i+he_j}-u_i)\big| }{h}
\\
= \, & \sum_{j=1}^{2} \frac{  \big| \hat{f}( u_{i+he_j})- \hat{f}( u_i) - \hat{f}'( u_i)(u_{i+he_j}-u_i)\big|}h \\=\,&\frac1h\sum_{j=1}^2\left|
\int_{u_i}^{u_{i + h e_j}} \big( \hat{f}'(t)- \hat{f}'(u_i)\big)\,dt
\right|\\ 
\le\,&
\frac{\| \hat{f}''\|_{L^\infty(\R)}}{2h}\sum_{j=1}^2
\,|u_{i+he_j}-u_i|^2\\=\,&\frac{\| \hat{f}''\|_{L^\infty(\R)}\,h}2
\sum_{j=1}^2|{\mathcal{D}}_j^+ u_i
|^2 ,
\end{split}
\end{equation*}
and hence \eqref{hp:semilinearitaconspazio} holds true with $\kappa_0^+ = (\kappa_1^+)^2 \,  \| \hat{f}''\|_{L^\infty(\R)}$.
\medskip

We stress that, in our setting, an exact symmetry result 
analogous to that in the continuous case 
cannot be obtained,
as the examples in Section \ref{sec:examples} show.
In particular, Examples \ref{example ARCTAN} and \ref{example EXPONENTIAL} show that such an exact symmetry result cannot hold true in the discrete case, even if we restrict our analysis to the case of source terms $f(i,u_i)= \hat{f}(u_i)$ that do not depend on the position and satisfy~\eqref{hp:GENERALsemilinearitaconspazio}.

All the examples in Section \ref{sec:examples}
also show that the rate of convergence of the estimate \eqref{FORM} in the formal limit $h \searrow 0$ is optimal in the sense that, in this case, right-hand side and left-hand side are of the same order of $h$. 
\medskip

Here, we will focus on the proof of Theorem~\ref{DGDG}
(the proof of Theorem~\ref{DGDG-GE} would then follow by a spatial symmetry argument).
The approach to prove Theorem~\ref{DGDG}
in this paper relies on the complex variable method introduced in~\cite{MR2014827}
to deal with the original De Giorgi's problem in~\cite{MR533166}
(in our setting, the discrete structure of the lattices
requires a careful estimate on the discrepancies between differentiable functions
and finite increments and concrete bounds on the approximations
performed).\medskip

The strategy of the proof of Theorem~\ref{DGDG}
is based on a useful identity for the increments of the solution
(roughly speaking, this method would be the discrete counterpart of
the study of a ``linearized equation'' in the continuum models setting). This step is accomplished in Lemma \ref{lem:linearized equation}.
Then, the desired result is obtained by an application of a new discrete quantitative
Liouville-type theorem (see the proof of Theorem~\ref{DGDG}).

The complex variable formalism introduced in~\cite{MR2014827}
reveals interesting cancellations when looking at the imaginary parts of
this type of equations: this is an interesting fact which makes
it possible to exploit this method
to all layer potentials, without structural restrictions, since
the potential plays no role in the imaginary part of the limit equation
in the continuum model case, and in the discrete case it only plays
a role in the estimates of the remainders (therefore, in our framework, assumption \eqref{hp:GENERALsemilinearitaconspazio} on the source term~$f$ is required to get
the error estimates, but no condition on the shape of~$f$ is necessary to obtain Theorem~\ref{DGDG}).
\medskip

The technical details of the proof of 
Theorem~\ref{DGDG} are
presented in Section~\ref{0.90m3S}
and exploit also auxiliary computations of elementary flavor
that are collected in Section~\ref{9i8e2udwh83yfgv}.
Section~\ref{0.90m3S} also contains the proof of
Lemma~\ref{KK-1D}.
Then,
in Section \ref{sec:examples} we provide some examples showing the optimality of our estimates.
\medskip

In future projects, we aim at developing the method of this paper
to obtain other forms of approximate counterparts of De Giorgi's conjecture in
the discrete setting especially by addressing approximation results of geometrical nature
and possibly detecting suitable hypotheses which make level sets of discrete solutions
appropriately close to a line, or to a portion of a line.
Besides the methods in~\cite{MR1655510}, for this it will be convenient to revisit
the geometric Poincar\'e formula in~\cite{MR1650327}, as utilized in~\cite{MR2483642},
since this type of inequalities provide natural bounds for the total curvature of the level
sets of the solutions. This technique could also lead to the study of long-range
interaction models, possibly involving infinitely many site interactions with suitable decay at
infinity, by taking advantage of integro-differential versions of the geometric Poincar\'e formula,
as done in~\cite{MR3469920} for the continuous case.\medskip

We conclude this introduction by pointing out a substantial difference between the discrete and the continuous settings. In the continuous case, if $u$ is a $C^2(\R^2)$ function with $| \nabla u|>0$, then $ \nabla u / | \nabla u| \in C^1( \R^2, \mathcal{S}^1)$. In that setting, in place of $\vartheta^+$ we have $\vartheta$, which is a $C^1(\R^2)$ function satisfying $\nabla u = | \nabla u| e^{i \vartheta} $. Notice that $\vartheta$ may be
unbounded in the continuous case\footnote{We remark that in
our notation~$\vartheta$ is a real number, rather than an element of the circle. In this way,
the polar representation of~$ \nabla u / | \nabla u| $
is a continuous function of~$\vartheta$. The boundedness or unboundedness of~$\vartheta$ has therefore
to be understood in this notation and, in particular, the unboundedness of~$\vartheta$ would
correspond to winding around the circle ``infinitely many times''.

A counterexample to the boundedness of~$\vartheta$ in the continuous
case is provided by the function~$u(x_1, x_2)=\sin x_1 -\cos x_2$, $(x_1,x_2) \in \R^2$, which satisfies~$\Delta u=-u$
and~$\nabla u(x_1 , x_2 )=(\cos x_1 , \sin x_2)$ (hence~$\nabla u(t,t)$ corresponds to~$e^{it}$ 
in complex variable notation and thus~$\vartheta(t,t)=t$, which is unbounded).}.
This necessarily leads to require further assumptions on $u$ such as monotonicity in a given direction, or more generally, stability (see \cite{MR2528756}). We recall that a sufficient condition to perform Farina's proof is given by the boundedness of $\vartheta^+$ (see \cite{MR2014827} and \cite{MR2528756}), which is surely verified for instance if $u$ is increasing in a given direction (say $e_2$). Indeed, the monotonicity of $u$ in the $e_2$ direction guarantees that $\nabla u$ does not ``turn backwards'', that is $\vartheta^+ \in \left[ 0, \pi \right] $, and so in particular that $\vartheta^+$ is bounded.

The discrete setting
provides here an interesting difference with respect to the continuous case.
Indeed, being
$\vartheta^+_i$ a function defined on $h\Z^2$, no continuity notions come into play, 
and one is free to choose all the angles in~$(-\pi,\pi]$ (that is, such a normalization
does not conflict with any continuity assumption in the discrete case).
For this reason, in our setting, if~\eqref{hp:graddiversodazero} holds true,
then one can always define the polar angle,
and renormalize it to fulfill~\eqref{PIKS}. In particular, in our framework,
no monotonicity or stability assumption is required, in contrast with the
models in the continuum.

We also stress that assumption \eqref{hp:graddiversodazero} is obviously satisfied if we assume
\begin{equation}
\label{MONO}
u_{i+he_2}>u_i\qquad{\mbox{for all }}i\in h\Z^2 ,
\end{equation}
that is a discrete ``monotonicity assumption''
in the vertical direction.

\section{Toolbox}\label{9i8e2udwh83yfgv}

This section contains some ancillary observations, to be used in the
proof of Theorem~\ref{DGDG} which is contained in Section~\ref{0.90m3S}.
We start by computing the operator $ {\mathcal{L}} $ on the product of two functions.
For this, if~$f$, $g:h\Z^2\to\R$, we write that~$\psi:=fg$ to mean
that, for any~$i\in h\Z^2$, we have~$\psi_i:=f_i g_i$.
We have the following product rule for the increment
quotients introduced in~\eqref{DER}:

\begin{lemma}
Let $f$, $g:h\Z^2\to\R$. Then, for all~$j\in\{1,2\}$,
\begin{equation}\label{PROFGRAD-1}
{\mathcal{D}}_j^+(fg)_i=
\frac{f_{i+he_j}+f_i}2\;{\mathcal{D}}_j^+g_i+
\frac{g_{i+he_j}+g_i}2\;{\mathcal{D}}_j^+f_i
\end{equation}
and
\begin{equation}\label{PROFGRAD-2}
{\mathcal{D}}_j^-(fg)_i=
\frac{f_{i-he_j}+f_i}2\;{\mathcal{D}}_j^-g_i+
\frac{g_{i-he_j}+g_i}2\;{\mathcal{D}}_j^-f_i
.\end{equation}
\end{lemma}

\begin{proof} We focus on the proof of~\eqref{PROFGRAD-2},
since the proof of~\eqref{PROFGRAD-1} is similar.
To this end, we point out that
\begin{eqnarray*}
{\mathcal{D}}_j^-(fg)_i&=&\frac{ (fg)_i-(fg)_{i-he_j}}h\\
&=&\frac{ f_i (g_i-g_{i-he_j})+g_{i-he_j}(f_i-f_{i-he_j})}h\\
&=& f_i\,{\mathcal{D}}_j^-g_i+g_{i-he_j}\,{\mathcal{D}}_j^-f_i.
\end{eqnarray*}
Also, exchanging the roles of~$f$ and~$g$,
\begin{eqnarray*}
{\mathcal{D}}_j^-(fg)_i&=& g_i\,{\mathcal{D}}_j^-f_i+f_{i-he_j}\,{\mathcal{D}}_j^-g_i,
\end{eqnarray*}
and therefore
\begin{eqnarray*}
2{\mathcal{D}}_j^-(fg)_i&=& 
\Big(f_i\,{\mathcal{D}}_j^-g_i+g_{i-he_j}\,{\mathcal{D}}_j^-f_i\Big)
+\Big(g_i\,{\mathcal{D}}_j^-f_i+f_{i-he_j}\,{\mathcal{D}}_j^-g_i\Big)\\
&=& (g_{i-he_j}+g_i)\,{\mathcal{D}}_j^-f_i
+(f_{i-he_j}+f_i)\,{\mathcal{D}}_j^-g_i,
\end{eqnarray*}
from which~\eqref{PROFGRAD-2} plainly follows.
\end{proof}

It is also interesting to observe that the increments defined in~\eqref{DER} produce by iteration (a suitable translation and
projection of)
the operator in~\eqref{ELLE}, namely, 
recalling the notation in~\eqref{0ok8uh7gfr8},
the following result holds true:

\begin{lemma}\label{SECLA}
Let $f:h\Z^2\to\R$. Then,
\begin{equation}\label{ITER-01}
 {\mathcal{D}}_j^+ ({\mathcal{D}}_j^+ f)_i=
{\mathcal{L}}_j f_{i+he_j}
\end{equation}
and
\begin{equation}\label{ITER-02}
{\mathcal{D}}_j^- ({\mathcal{D}}_j^- f)_i=
{\mathcal{L}}_jf_{i-he_j}.
\end{equation}
\end{lemma}

\begin{proof} We see that, for every~$j\in\{1,2\}$,
\begin{eqnarray*}
{\mathcal{D}}_j^- ({\mathcal{D}}_j^- f)_i&=&\frac{
{\mathcal{D}}_j^- f_i-{\mathcal{D}}_j^- f_{i-he_j}}h\\&=&\frac{
(f_i-f_{i-he_j})-(f_{i-he_j}-f_{i-2he_j})}{h^2}\\
&=& \frac{f_{i-2he_j}+f_{i}-2f_{i-he_j}}{h^2}.
\end{eqnarray*}
This proves~\eqref{ITER-02}, and the proof of~\eqref{ITER-01}
is similar.
\end{proof}

We remark that it is not always convenient to sum~\eqref{ITER-01} and~\eqref{ITER-02}
over~$j\in\{1,2\}$, since the right-hand side depends on~$j$ in such a way that
the operator~${\mathcal{L}}$ does not appear straight away after such a summation.

We also present the following useful computation:

\begin{lemma} \label{LEPROD-ekf}
Let $f$, $g:h\Z^2\to\R$. Then,
\begin{equation}\label{PROD1}
{\mathcal{L}}(fg)_i= {\mathcal{L}} f_i\,g_i+
{\mathcal{L}} g_i\,f_i+
\sum_{j=1}^2
\big( {\mathcal{D}}_j^+f_i\,{\mathcal{D}}_j^+g_i+
{\mathcal{D}}_j^-f_i\,{\mathcal{D}}_j^-g_i \big)
.
\end{equation}
\end{lemma}

\begin{proof} We observe that
\begin{eqnarray*}&&
{\mathcal{L}} (fg)_i-{\mathcal{L}} f_i\,g_i\\&=&\frac1{h^2}
\sum_{j=1}^2 \big( f_{i+he_j}g_{i+he_j} + f_{i-he_j}g_{i-he_j} -2f_ig_i
- f_{i+he_j}g_i - f_{i-he_j}g_i +2f_ig_i\big)\\&=&\frac1{h^2}
\sum_{j=1}^2 \big( f_{i+he_j}(g_{i+he_j}-g_i) + f_{i-he_j}(g_{i-he_j}-g_i)
\big)\\
&=&\frac1{h^2}\sum_{j=1}^2 \big( (f_{i+he_j}-f_i)(g_{i+he_j}-g_i) 
+f_i(g_{i+he_j}-g_i)
+ (f_{i-he_j}-f_i)(g_{i-he_j}-g_i)+f_i(g_{i-he_j}-g_i)
\big)\\&=&
{\mathcal{L}} g_i\,f_i+\frac1{h^2}
\sum_{j=1}^2 \big( (f_{i+he_j}-f_i)(g_{i+he_j}-g_i) 
+ (f_i-f_{i-he_j})(g_i-g_{i-he_j})
\big).
\end{eqnarray*}
This, together with the notation in~\eqref{DER}, proves~\eqref{PROD1}.
\end{proof}

Now we give the following ``summation by parts'' formula:

\begin{lemma}
Let~$f$, $g:h\Z^2\to\R$. Assume that
\begin{equation}\label{BYPa0}
\#\{i\in h\Z^2{\mbox{ s.t. }}g_i\ne0\}<+\infty.\end{equation}
Then,
\begin{equation}\label{BYPa1}
\sum_{i\in h\Z^2} {\mathcal{D}}_j^+ f_i\, g_i=-\sum_{i\in h\Z^2}f_i\,{\mathcal{D}}_j^-g_i
\end{equation}
and
\begin{equation}\label{BYPa2}
\sum_{i\in h\Z^2} {\mathcal{D}}_j^- f_i\, g_i=-\sum_{i\in h\Z^2}f_i\,{\mathcal{D}}_j^+g_i.
\end{equation}
\end{lemma}

\begin{proof} We give the proof of~\eqref{BYPa2} in detail, the proof of~\eqref{BYPa1}
being similar. We note that the series in~\eqref{BYPa2} are finite sums, thanks to~\eqref{BYPa0}. As a result,
\begin{eqnarray*}
&&h\sum_{i\in h\Z^2} {\mathcal{D}}_j^- f_i\, g_i=
\sum_{i\in h\Z^2} (f_i-f_{i-he_j})\, g_i
=\sum_{i\in h\Z^2} f_ig_i-\sum_{k\in h\Z^2} f_{k} g_{k+he_j}\\&&\qquad
=-\sum_{i\in h\Z^2} f_i(g_{i+he_j}-g_i)
=-h\sum_{i\in h\Z^2}f_i\,{\mathcal{D}}_j^+g_i,
\end{eqnarray*}
as desired.
\end{proof}

\section{Proof of Theorem~\ref{DGDG}}\label{0.90m3S}

This section contains the proof of
Theorem~\ref{DGDG}. Some of the arguments are inspired
by the complex variable formulation introduced in~\cite{MR2014827}:
in our framework, the core of the proof is to exploit the ``continuum models''
techniques arising in partial differential equation in the ``discrete''
setting provided by the operator in~\eqref{ELLE}, with a careful estimates
of the reminders.

The following lemma provides an identity for the increments of the solution of \eqref{DGEQ} together with a quantitative estimate of the reminder term.

\begin{lemma}\label{lem:linearized equation}
It holds that
\begin{equation}\label{KS:0olsdPSPD}
\sum_{j=1}^2\Big({\mathcal{D}}_j^+(\rho^2 {\mathcal{D}}_j^+\vartheta^+)_i+
{\mathcal{D}}_j^-(\rho^2 {\mathcal{D}}_j^-\vartheta^+)_i\Big)
= \epsilon_i^{\star} ,
\end{equation}
with
\begin{equation}\label{eq:stimaerrore eps12 in enunciato lemma}
\begin{split}
|\epsilon_i^{\star}| & \le  
2e^{2\pi}h\, \sum_{j=1}^2
(\rho_i^+)^2\,
\Big( |{\mathcal{D}}_j^+ \vartheta^+_i|^3+|{\mathcal{D}}_j^- \vartheta^+_i|^3\Big)
\\& \quad +
2e^{2\pi}h\,\sum_{j=1}^2\rho_i^+
\big( |{\mathcal{D}}_j^+\rho_i^+|\;|{\mathcal{D}}_j^+ \vartheta^+_i|^2
+
|{\mathcal{D}}_j^-\rho_i^+|\;|{\mathcal{D}}_j^- \vartheta^+_i|^2
\big)+
2 \kappa_0^+ h \rho_i^+
\\& \quad +
h \sum_{j=1}^2\Big(
|{\mathcal{D}}_j^+(\rho_i^+)^2| \,|{\mathcal{D}}_j^+({\mathcal{D}}_j^+\vartheta^+)_i|+
|{\mathcal{D}}_j^-(\rho_i^+)^2| \,|{\mathcal{D}}_j^-({\mathcal{D}}_j^-\vartheta^+)_i|+
 h \,  (\rho_i^+)^2 \,  | {\mathcal{L}^2}_j\vartheta^+_{i} |\Big)
 \\&\quad +
h \sum_{j=1}^2\Big(
|{\mathcal{D}}_j^+\rho_i^+|^2\,
|{\mathcal{D}}_j^+\vartheta^+_i|
+ |{\mathcal{D}}_j^-\rho_i^+|^2\,|{\mathcal{D}}_j^-\vartheta^+_i|\Big) ,
\end{split}
\end{equation}
where $\kappa_0^+$ is that defined in \eqref{hp:GENERALsemilinearitaconspazio}.
\end{lemma}

We remark that Lemma~\ref{lem:linearized equation}
is an approximate counterpart in the discrete setting of Lemma~2.2
in~\cite{MR2014827}. In particular, formula~(2.9)
in~\cite{MR2014827} is the continuous counterpart of~\eqref{KS:0olsdPSPD}
here. Notice that in the continuous case the remainder~$\epsilon_i^{\star}$
is replaced simply by zero.

\begin{proof}[Proof of Lemma~\ref{lem:linearized equation}]
For~$k\in\{1,2\}$, we use~\eqref{DGEQ} to write that
\begin{eqnarray*}
f( i+h e_k, u_{i+he_k}) - f(i, u_i)&=&
{\mathcal{L}} u_{i+he_k}-{\mathcal{L}} u_i\\&=&\frac1{h^2}
\sum_{j=1}^2 \big( u_{i+he_k+he_j} + u_{i+he_k-he_j} -2u_{i+he_k}
- u_{i+he_j} - u_{i-he_j} +2u_i\big)\\&=&\frac1{h^2}
\sum_{j=1}^2 \Big( \big(u_{i+he_k+he_j}
- u_{i+he_j}\big) + \big(u_{i+he_k-he_j} - u_{i-he_j}\big)-2\big(
u_{i+he_k} -u_i\big)\Big)
\\&=&\frac1{h}
\sum_{j=1}^2 \Big(
{\mathcal{D}}_k^+ u_{i+he_j} 
+{\mathcal{D}}_k^+ u_{i-he_j}
-2{\mathcal{D}}_k^+ u_{i}\Big)
\\&=& h\,{\mathcal{L}}({\mathcal{D}}_k^+ u)_{i},
\end{eqnarray*}
and consequently, recalling the definition
of~${U_i^+}$ in~\eqref{1.7BIS},
$$ \frac{\big(f( i+h e_1, u_{i+he_1})-f( i, u_i)\big)+{\mathcal{I}}
\big(f( i+he_2, u_{i+he_2})-f(i, u_i)\big)}h=
{\mathcal{L}}({\mathcal{D}}^+_1u)_i+
{\mathcal{I}}{\mathcal{L}}({\mathcal{D}}^+_2u)_i
=
{\mathcal{L}}{U_i^+}.$$
Therefore, setting\footnote{In the following pages, we will have to estimate several
remainders that will be denoted by~$\epsilon_i^{(1)}, \cdots, \epsilon_i^{(14)}$. Each of
these remainders does not possess a particular meaning in itself and requires a specific
estimate in order to be controlled by quantities depending on the mesh parameter~$h$.}
\begin{multline*}
\epsilon_i^{(1)}:=
\\
\frac{
\big(f( i+he_1, u_{i+he_1})-f( i, u_i)\big)+{\mathcal{I}}
\big(f( i+he_2, u_{i+he_2})-f(i, u_i)\big)- L_f^+ (i, u_i)\big(
(u_{i+he_1}-u_i)+{\mathcal{I}}(u_{i+he_2}-u_i)\big)}h ,
\end{multline*}
we find that
\begin{equation}\label{in56p4rf}
{\mathcal{L}}{U_i^+}= L_f^+ (i, u_i)\,{U_i^+}+\epsilon_i^{(1)} ,
\end{equation}
and, by \eqref{hp:GENERALsemilinearitaconspazio},
\begin{equation}\label{eq:nuovastimadieps1}
|\epsilon_i^{(1)}| \le \kappa_0^+ h.
\end{equation}
We also note that
\begin{equation}\label{eq:staraggiunta1}
{\mathcal{D}}_j^+ e^{{\mathcal{I}}\vartheta^+_i}= \frac{
e^{{\mathcal{I}}\vartheta^+_{i+he_j}}-e^{{\mathcal{I}}\vartheta^+_i}}h
= \frac{{\mathcal{I}}e^{{\mathcal{I}}\vartheta^+_i}(\vartheta^+_{i+he_j}-\vartheta^+_{i})}h+
\epsilon^{(2,+,j)}_i ,
\end{equation}
where
$$ \epsilon^{(2,+,j)}_i:=\frac{e^{{\mathcal{I}}\vartheta^+_{i+he_j}}-e^{{\mathcal{I}}\vartheta^+_i}-
{\mathcal{I}}e^{{\mathcal{I}}\vartheta^+_i}(\vartheta^+_{i+he_j}-\vartheta^+_{i})}h.$$
Similarly,
\begin{eqnarray}\label{eq:staraggiunta2}
{\mathcal{D}}_j^- e^{{\mathcal{I}}\vartheta^+_i}&=&\frac{
{\mathcal{I}}e^{{\mathcal{I}}\vartheta^+_i}(\vartheta^+_{i}-\vartheta^+_{i-he_j})}{h}+
\epsilon^{(2,-,j)}_i
,\end{eqnarray}
where
$$ \epsilon^{(2,-,j)}_i:=\frac{e^{{\mathcal{I}}\vartheta^+_i}-
e^{{\mathcal{I}}\vartheta^+_{i-he_j}}
-
{\mathcal{I}}e^{{\mathcal{I}}\vartheta^+_i}(\vartheta^+_{i}-\vartheta^+_{i-he_j})}h.$$
We remark that,
for every~$t\in(-2\pi,2\pi)$,
$$ \left|e^{{\mathcal{I}} t}-1-{\mathcal{I}}t\right|=\left|\sum_{\ell=2}^{+\infty}\frac{({\mathcal{I}}t)^\ell}{\ell!}\right|\le
\sum_{\ell=2}^{+\infty}\frac{|t|^\ell}{\ell!}=t^2
\sum_{m=0}^{+\infty}\frac{|t|^{m}}{(m+2)!}
\le t^2
\sum_{m=0}^{+\infty}\frac{(2\pi)^{m}}{m!}=e^{2\pi}t^2,
$$
and therefore, by~\eqref{PIKS},
\begin{equation}\label{2.4}
\begin{split}
|\epsilon^{(2,+,j)}_i|\,&=\left|\frac{e^{{\mathcal{I}}\vartheta^+_i}}h\,
\Big(e^{{\mathcal{I}}(\vartheta^+_{i+he_j}-\vartheta^+_i)}-1-
{\mathcal{I}}(\vartheta^+_{i+he_j}-\vartheta^+_{i})
\Big)\right|
\\&=\frac{1}{h}\,
\Big|e^{{\mathcal{I}}(\vartheta^+_{i+he_j}-\vartheta^+_i)}-1-
{\mathcal{I}}(\vartheta^+_{i+he_j}-\vartheta^+_{i})
\Big|\\&\le\frac{e^{2\pi}\,|\vartheta^+_{i+he_j}-\vartheta^+_{i}|^2}{h}
=e^{2\pi}h\, |{\mathcal{D}}_j^+ \vartheta^+_i|^2
\end{split}\end{equation}
and analogously
\begin{equation}\label{2.5}
|\epsilon^{(2,-,j)}_i|\le e^{2\pi}h\, |{\mathcal{D}}_j^- \vartheta^+_i|^2.
\end{equation}
Furthermore,
\begin{equation}\label{INStere74}
\begin{split}
{\mathcal{L}}e^{{\mathcal{I}}\vartheta^+_i}\,=&\frac1{h^2}
\sum_{j=1}^2\big(
e^{{\mathcal{I}}\vartheta^+_{i+he_j}}+e^{{\mathcal{I}}\vartheta^+_{i-he_j}}
-2e^{{\mathcal{I}}\vartheta^+_i}\big)
\\=&
\frac{e^{{\mathcal{I}}\vartheta^+_{i}}}{h^2}
\sum_{j=1}^2\big(
e^{{\mathcal{I}}(\vartheta^+_{i+he_j}-\vartheta^+_i)}+e^{{\mathcal{I}}(\vartheta^+_{i-he_j}-\vartheta^+_i)}
-2\big)
\\=&
\frac{{\mathcal{I}}\,e^{{\mathcal{I}}\vartheta^+_{i}}}{h^2}
\sum_{j=1}^2\big(
\vartheta^+_{i+he_j}+\vartheta^+_{i-he_j}-2\vartheta^+_i\big)
-
\frac{e^{{\mathcal{I}}\vartheta^+_{i}}}{2}
\sum_{j=1}^2\big(|{\mathcal{D}}_j^+\vartheta^+_i|^2+|{\mathcal{D}}_j^-\vartheta^+_i|^2\big)
+\epsilon^{(2)}_i\\=&
{\mathcal{I}}\,e^{{\mathcal{I}}\vartheta^+_{i}}{\mathcal{L}}\vartheta^+_i
-
\frac{e^{{\mathcal{I}}\vartheta^+_{i}}}{2}
\sum_{j=1}^2\big(|{\mathcal{D}}_j^+\vartheta^+_i|^2+|{\mathcal{D}}_j^-\vartheta^+_i|^2\big)
+\epsilon^{(2)}_i,
\end{split}\end{equation}
where
\begin{eqnarray*}
&&\epsilon^{(2)}_i:=\frac{e^{{\mathcal{I}}\vartheta^+_{i}}}{h^2}
\sum_{j=1}^2\Bigg(
\big(
e^{{\mathcal{I}}(\vartheta^+_{i+he_j}-\vartheta^+_i)}+e^{{\mathcal{I}}(\vartheta^+_{i-he_j}-\vartheta^+_i)}
-2\big)-{\mathcal{I}}\big(
\vartheta^+_{i+he_j}+\vartheta^+_{i-he_j}-2\vartheta^+_i\big)\\&&\qquad\qquad\qquad+\frac{h^2}2\big(|{\mathcal{D}}_j^+\vartheta^+_i|^2+|{\mathcal{D}}_j^-\vartheta^+_i|^2\big)\Bigg).\end{eqnarray*}
Since, for every~$t\in(-2\pi,2\pi)$,
$$ \left|e^{{\mathcal{I}} t}-1-{\mathcal{I}}t+\frac{t^2}{2}\right|=\left|\sum_{\ell=3}^{+\infty}\frac{({\mathcal{I}}t)^\ell}{\ell!}\right|\le
\sum_{\ell=3}^{+\infty}\frac{|t|^\ell}{\ell!}=|t|^3
\sum_{m=0}^{+\infty}\frac{|t|^{m}}{(m+3)!}
\le |t|^3
\sum_{m=0}^{+\infty}\frac{(2\pi)^{m}}{m!}=e^{2\pi}|t|^3,
$$
we deduce from~\eqref{PIKS} that
\begin{equation}\label{2.7}
\begin{split}
|\epsilon^{(2)}_i|\,&=\frac{1}{h^2}
\Bigg|
\sum_{j=1}^2\Bigg(
\big(
e^{{\mathcal{I}}(\vartheta^+_{i+he_j}-\vartheta^+_i)}+e^{{\mathcal{I}}(\vartheta^+_{i-he_j}-\vartheta^+_i)}
-2\big)-{\mathcal{I}}\big(
(\vartheta^+_{i+he_j}-\vartheta^+_i)+(\vartheta^+_{i-he_j}-\vartheta^+_i)\big)\\&\qquad\quad+\frac12\big(|\vartheta^+_{i+he_j}-\vartheta^+_i|^2+|\vartheta^+_{i-he_j}-\vartheta^+_i|^2\big)\Bigg)
\Bigg|\\&\le
\frac{1}{h^2}
\sum_{j=1}^2\Bigg(
\Bigg|
e^{{\mathcal{I}}(\vartheta^+_{i+he_j}-\vartheta^+_i)}-1
-{\mathcal{I}}
(\vartheta^+_{i+he_j}-\vartheta^+_i)
+\frac{|\vartheta^+_{i+he_j}-\vartheta^+_i|^2}2
\Bigg|\\&\qquad\quad+\Bigg|
e^{{\mathcal{I}}(\vartheta^+_{i-he_j}-\vartheta^+_i)}-1
-{\mathcal{I}}
(\vartheta^+_{i-he_j}-\vartheta^+_i)
+\frac{|\vartheta^+_{i-he_j}-\vartheta^+_i|^2}2
\Bigg|\Bigg)\\&\le\frac{e^{2\pi}}{h^2}\sum_{j=1}^2\Big(|\vartheta^+_{i+he_j}-\vartheta^+_i|^3+
|\vartheta^+_{i-he_j}-\vartheta^+_i|^3\Big)\\&=
e^{2\pi}h\, \sum_{j=1}^2\Big( |{\mathcal{D}}_j^+ \vartheta^+_i|^3+|{\mathcal{D}}_j^- \vartheta^+_i|^3\Big)
.
\end{split}
\end{equation}
Then, from~\eqref{PROD1}, \eqref{eq:staraggiunta1}, \eqref{eq:staraggiunta2}, and~\eqref{INStere74},
\begin{eqnarray*}{\mathcal{L}}{U_i^+}&=&
{\mathcal{L}}(\rho
e^{{\mathcal{I}}\vartheta^+})_i\\&=&
{\mathcal{L}} \rho_i^+\,e^{{\mathcal{I}}\vartheta^+_i}+
{\mathcal{L}} e^{{\mathcal{I}}\vartheta^+_i}\,\rho_i^+ +
\sum_{j=1}^2
\big( {\mathcal{D}}_j^+\rho_i^+\,{\mathcal{D}}_j^+e^{{\mathcal{I}}\vartheta^+_i}+
{\mathcal{D}}_j^-\rho_i^+\,{\mathcal{D}}_j^-e^{{\mathcal{I}}\vartheta^+_i} \big)\\&=&
{\mathcal{L}} \rho_i^+\,e^{{\mathcal{I}}\vartheta^+_i}+
\rho_i^+\left( {\mathcal{I}}e^{{\mathcal{I}}\vartheta^+_i}\,
{\mathcal{L}}\vartheta^+_i
-\frac{e^{{\mathcal{I}}\vartheta^+_{i}}}{2}
\sum_{j=1}^2\big(|{\mathcal{D}}_j^+\vartheta^+_i|^2+
|{\mathcal{D}}_j^-\vartheta^+_i|^2\big)+\epsilon^{(2)}_i\right)\\&&\quad+
\sum_{j=1}^2
\Bigg( {\mathcal{D}}_j^+\rho_i^+\left(
{\mathcal{I}}e^{{\mathcal{I}}\vartheta^+_i}\;\frac{\vartheta^+_{i+he_j}-\vartheta^+_{i}}h+
\epsilon^{(2,+,j)}_i
\right)+
{\mathcal{D}}_j^-\rho_i^+\left(
{\mathcal{I}}e^{{\mathcal{I}}\vartheta^+_i}\;\frac{\vartheta^+_{i}-\vartheta^+_{i-he_j}}h+
\epsilon^{(2,-,j)}_i\right)\Bigg)\\&=&
{\mathcal{L}} \rho_i^+\,e^{{\mathcal{I}}\vartheta^+_i}+
{\mathcal{I}}\rho_i^+\, e^{{\mathcal{I}}\vartheta^+_i}\,{\mathcal{L}}\vartheta^+_i-\frac{\rho_i^+ \,e^{{\mathcal{I}}\vartheta^+_{i}}}{2}
\sum_{j=1}^2\big(|{\mathcal{D}}_j^+\vartheta^+_i|^2+|{\mathcal{D}}_j^-\vartheta^+_i|^2\big)\\&&\quad+
\sum_{j=1}^2
\left( {\mathcal{D}}_j^+\rho_i^+\;
{\mathcal{I}}e^{{\mathcal{I}}\vartheta^+_i}\;\frac{\vartheta^+_{i+he_j}-\vartheta^+_{i}}h
+
{\mathcal{D}}_j^-\rho_i^+\;
{\mathcal{I}}e^{{\mathcal{I}}\vartheta^+_i}\;\frac{\vartheta^+_{i}-\vartheta^+_{i-he_j}}h\right)+\epsilon^{(3)}_i\\
&=&
e^{{\mathcal{I}}\vartheta^+_i}\Bigg[
{\mathcal{L}} \rho_i^++{\mathcal{I}}
\rho_i^+  \,{\mathcal{L}}\vartheta^+_i-\frac{\rho_i^+ }{2}
\sum_{j=1}^2\big(|{\mathcal{D}}_j^+\vartheta^+_i|^2+|{\mathcal{D}}_j^-\vartheta^+_i|^2\big)\\&&\quad+
{\mathcal{I}}
\sum_{j=1}^2
\Big( {\mathcal{D}}_j^+\rho_i^+\;
{\mathcal{D}}^+_j \vartheta^+_{i}
+
{\mathcal{D}}_j^-\rho_i^+\;{\mathcal{D}}^-_j\vartheta^+_{i}\Big)\Bigg]
+\epsilon^{(3)}_i\\
&=&
e^{{\mathcal{I}}\vartheta^+_i}\Bigg[
{\mathcal{L}} \rho_i^++{\mathcal{I}}
\rho_i^+  \,{\mathcal{L}}\vartheta^+_i
-\frac{\rho_i^+ }{2}
\sum_{j=1}^2\big(|{\mathcal{D}}_j^+\vartheta^+_i|^2+|{\mathcal{D}}_j^-\vartheta^+_i|^2\big)\\&&\quad+ 
{\mathcal{I}}
\sum_{j=1}^2
\Big( {\mathcal{D}}_j^+\rho_i^+\;
{\mathcal{D}}_j^+\vartheta^+_{i}
+
{\mathcal{D}}_j^-\rho_i^+\;{\mathcal{D}}_j^-\vartheta^+_{i}\Big)+\epsilon^{(4)}_i\Bigg].
\end{eqnarray*}
where
$$\epsilon^{(3)}_i:=\rho_i^+ \epsilon^{(2)}_i+
\sum_{j=1}^2
\big( {\mathcal{D}}_j^+\rho_i^+\;
\epsilon^{(2,+,j)}_i
+
{\mathcal{D}}_j^-\rho_i^+\;
\epsilon^{(2,-,j)}_i\big)
$$
and
$$\epsilon^{(4)}_i:=
e^{-{\mathcal{I}}\vartheta^+_i}\,\epsilon^{(3)}_i.$$
Hence, recalling~\eqref{in56p4rf} and letting
$$ \epsilon^{(5)}_i:=
e^{-{\mathcal{I}}\vartheta^+_i}\,\epsilon^{(1)}_i,$$
we conclude that
\begin{equation}\label{BEF}
\begin{split}&
e^{{\mathcal{I}}\vartheta^+_i}\Big( L_f^+ (i, u_i)\,\rho_i^++
\epsilon_i^{(5)}\Big)\\=\;&
L_f^+ (i, u_i)\,{U_i^+}+\epsilon_i^{(1)}
\\ =\;&{\mathcal{L}}{U_i^+}\\
=\;&
e^{{\mathcal{I}}\vartheta^+_i}\Bigg[
{\mathcal{L}} \rho_i^++{\mathcal{I}}
\rho_i^+  \,{\mathcal{L}}\vartheta^+_i-\frac{\rho_i^+ }{2}
\sum_{j=1}^2\big(|{\mathcal{D}}_j^+\vartheta^+_i|^2+|{\mathcal{D}}_j^-\vartheta^+_i|^2\big)+{\mathcal{I}}
\sum_{j=1}^2
\Big( {\mathcal{D}}_j^+\rho_i^+\;
{\mathcal{D}}_j^+\vartheta^+_{i}
+
{\mathcal{D}}_j^-\rho_i^+\;{\mathcal{D}}_j^-\vartheta^+_{i}\Big)+\epsilon^{(4)}_i\Bigg].
\end{split}\end{equation}
We point out that
\begin{equation}\label{K29j29}
\begin{split}
|\epsilon^{(4)}_i|+|\epsilon^{(5)}_i|\,&=|\epsilon^{(3)}_i|+|
\epsilon^{(1)}_i|
\\&
\le|
\rho_i^+ |\,|\epsilon^{(2)}_i|+
\sum_{j=1}^2
\big( |{\mathcal{D}}_j^+\rho_i^+|\;
|\epsilon^{(2,+,j)}_i|
+
|{\mathcal{D}}_j^-\rho_i^+|\;|
\epsilon^{(2,-,j)}_i|\big)+|\epsilon^{(1)}_i|\\&\le
e^{2\pi}h\, \sum_{j=1}^2
|\rho_i^+ |\,
\Big( |{\mathcal{D}}_j^+ \vartheta^+_i|^3+|{\mathcal{D}}_j^- \vartheta^+_i|^3\Big)
+
e^{2\pi}h\,\sum_{j=1}^2
\big( |{\mathcal{D}}_j^+\rho_i^+|\;|{\mathcal{D}}_j^+ \vartheta^+_i|^2
+
|{\mathcal{D}}_j^-\rho_i^+|\;|{\mathcal{D}}_j^- \vartheta^+_i|^2
\big)
+ \kappa_0^+ h
%
%
%
%
\end{split}\end{equation}
where we have also exploited~\eqref{eq:nuovastimadieps1},
\eqref{2.4}, \eqref{2.5} and~\eqref{2.7}.

After simplifying the term~$e^{{\mathcal{I}}\vartheta^+_i}$
in~\eqref{BEF}, and setting
$$ \epsilon_i^{(6)}:=\epsilon_i^{(5)}-\epsilon_i^{(4)},$$
we discover that
\begin{equation}\label{0-02erkgarPRE} L_f^+ (i, u_i)\,\rho_i^++
\epsilon_i^{(6)}\,=\,
{\mathcal{L}} \rho_i^++{\mathcal{I}}
\rho_i^+  \,{\mathcal{L}}\vartheta^+_i-\frac{\rho_i^+ }{2}\sum_{j=1}^2\big(|{\mathcal{D}}_j^+\vartheta^+_i|^2+|{\mathcal{D}}_j^-\vartheta^+_i|^2\big)+{\mathcal{I}}
\sum_{j=1}^2
\Big( {\mathcal{D}}_j^+\rho_i^+\;
{\mathcal{D}}_j^+\vartheta^+_{i}
+
{\mathcal{D}}_j^-\rho_i^+\;{\mathcal{D}}_j^-\vartheta^+_{i}\Big).\end{equation}
Therefore, denoting by~$\epsilon_i^{(7)}$ the imaginary part of~$\epsilon_i^{(6)}$, 
by taking the imaginary part of equation~\eqref{0-02erkgarPRE}
we find that
\begin{equation}\label{0-02erkgar}
\epsilon_i^{(7)}\,=\,\rho_i^+\,
{\mathcal{L}}\vartheta^+_i+\sum_{j=1}^2
\Big( {\mathcal{D}}_j^+\rho_i^+\;
{\mathcal{D}}_j^+\vartheta^+_{i}
+
{\mathcal{D}}_j^-\rho_i^+\;{\mathcal{D}}_j^-\vartheta^+_{i}\Big).
\end{equation}
In addition, recalling~\eqref{K29j29},
\begin{equation}\label{spearr0}\begin{split}&
|\epsilon_i^{(7)}|\le|\epsilon_i^{(6)}|=|\epsilon_i^{(4)}-\epsilon_i^{(5)}|\le |\epsilon_i^{(4)}|+|\epsilon_i^{(5)}|\\
&\qquad\leq 
e^{2\pi}h\, \sum_{j=1}^2
|\rho_i^+ |\,
\Big( |{\mathcal{D}}_j^+ \vartheta^+_i|^3+|{\mathcal{D}}_j^- \vartheta^+_i|^3\Big)
+
e^{2\pi}h\,\sum_{j=1}^2
\big( |{\mathcal{D}}_j^+\rho_i^+|\;|{\mathcal{D}}_j^+ \vartheta^+_i|^2
+
|{\mathcal{D}}_j^-\rho_i^+|\;|{\mathcal{D}}_j^- \vartheta^+_i|^2
\big)
+ \kappa_0^+ h.\end{split}
\end{equation}
Now, using~\eqref{PROFGRAD-2}, we see that
\begin{equation}\begin{split}\label{TH6} {\mathcal{D}}_j^-(\rho^2 {\mathcal{D}}_j^-\vartheta^+)_i\,&=
\frac{\rho^2_{i-he_j}+(\rho_i^+)^2}2\;{\mathcal{D}}_j^-({\mathcal{D}}_j^-\vartheta^+)_i+
\frac{{\mathcal{D}}_j^-\vartheta^+_{i-he_j}+{\mathcal{D}}_j^-\vartheta^+_i}2
\;{\mathcal{D}}_j^-(\rho_i^+)^2
\\&=
(\rho_i^+)^2\;{\mathcal{D}}_j^-({\mathcal{D}}_j^-\vartheta^+)_i+
{{\mathcal{D}}_j^-\vartheta^+_i}
\;{\mathcal{D}}_j^-(\rho_i^+)^2+\epsilon_i^{(8,-,j)}
,\end{split}\end{equation}
where
\begin{eqnarray*}\epsilon_i^{(8,-,j)}&:=&
\frac{\rho^2_{i-he_j}-(\rho_i^+)^2}2\;{\mathcal{D}}_j^-({\mathcal{D}}_j^-\vartheta^+)_i+
\frac{{\mathcal{D}}_j^-\vartheta^+_{i-he_j}-{\mathcal{D}}_j^-\vartheta^+_i}2
\;{\mathcal{D}}_j^-(\rho_i^+)^2
\\ &=&\Big((\rho^2_{i-he_j}-(\rho_i^+)^2)+h\,
{\mathcal{D}}_j^-(\rho_i^+)^2\Big)\,\frac{{\mathcal{D}}_j^-({\mathcal{D}}_j^-\vartheta^+)_i}2
\\&=&(\rho^2_{i-he_j}-(\rho_i^+)^2)\,{\mathcal{D}}_j^-({\mathcal{D}}_j^-\vartheta^+)_i.
\end{eqnarray*}
{F}rom~\eqref{ITER-02} and~\eqref{TH6} we deduce that
\begin{equation}\label{ILAN-3} \sum_{j=1}^2{\mathcal{D}}_j^-(\rho^2 {\mathcal{D}}_j^-\vartheta^+)_i=\sum_{j=1}^2
(\rho_i^+)^2\;{\mathcal{L}}_j\vartheta^+_{i-he_j}+\sum_{j=1}^2
{\mathcal{D}}_j^-\vartheta^+_i
\;{\mathcal{D}}_j^-(\rho_i^+)^2+\epsilon_i^{(8,-)},\end{equation}
with
$$ \epsilon_i^{(8,-)}:=\sum_{j=1}^2\epsilon_i^{(8,-,j)}.$$
Similarly, setting
$$\epsilon_i^{(8,+,j)}:=
\frac{\rho^2_{i+he_j}-(\rho_i^+)^2}2\;{\mathcal{D}}_j^+({\mathcal{D}}_j^+\vartheta^+)_i+
\frac{{\mathcal{D}}_j^+\vartheta^+_{i+he_j}-{\mathcal{D}}_j^+\vartheta^+_i}2
\;{\mathcal{D}}_j^+(\rho_i^+)^2=
(\rho^2_{i+he_j}-(\rho_i^+)^2){\mathcal{D}}_j^+({\mathcal{D}}_j^+\vartheta^+)_i
$$
and
$$ \epsilon_i^{(8,+)}:=\sum_{j=1}^2\epsilon_i^{(8,+,j)},$$
we see that
\begin{equation}\label{ILAN-4} \sum_{j=1}^2{\mathcal{D}}_j^+(\rho^2 {\mathcal{D}}_j^+\vartheta^+)_i=\sum_{j=1}^2
(\rho_i^+)^2\;{\mathcal{L}}_j\vartheta^+_{i+he_j}+\sum_{j=1}^2
{\mathcal{D}}_j^+\vartheta^+_i
\;{\mathcal{D}}_j^+(\rho_i^+)^2+\epsilon_i^{(8,+)}.\end{equation}
Accordingly, if we define
$$ \epsilon_i^{(9)}:=
\epsilon_i^{(8,+)}+\epsilon_i^{(8,-)}+\sum_{j=1}^2
(\rho_i^+)^2\;\big({\mathcal{L}}_j\vartheta^+_{i+he_j}
+{\mathcal{L}}_j\vartheta^+_{i-he_j}-2{\mathcal{L}}_j\vartheta^+_{i}\big),$$
we have that
\begin{equation}\label{EPS9}
\begin{split}
|\epsilon_i^{(9)}|\,&= \left|
\sum_{j=1}^2\epsilon_i^{(8,+,j)}+\sum_{j=1}^2\epsilon_i^{(8,-,j)}+\sum_{j=1}^2
(\rho_i^+)^2\;\big({\mathcal{L}}_j\vartheta^+_{i+he_j}
+{\mathcal{L}}_j\vartheta^+_{i-he_j}-2{\mathcal{L}}_j\vartheta^+_{i}\big)\right|\\
&=
\Bigg|
\sum_{j=1}^2\Big(
(\rho^2_{i+he_j}-(\rho_i^+)^2)\,{\mathcal{D}}_j^+({\mathcal{D}}_j^+\vartheta^+)_i+
(\rho^2_{i-he_j}-(\rho_i^+)^2)\,{\mathcal{D}}_j^-({\mathcal{D}}_j^-\vartheta^+)_i\Big)\\&\qquad\qquad+\sum_{j=1}^2
(\rho_i^+)^2\;\big({\mathcal{L}}_j\vartheta^+_{i+he_j}
+{\mathcal{L}}_j\vartheta^+_{i-he_j}-2{\mathcal{L}}_j\vartheta^+_{i}\big)
\Bigg|
\\
&\le
\sum_{j=1}^2\Big(
|\rho^2_{i+he_j}-(\rho_i^+)^2|\,|{\mathcal{D}}_j^+({\mathcal{D}}_j^+\vartheta^+)_i|+
|\rho^2_{i-he_j}-(\rho_i^+)^2|\,|{\mathcal{D}}_j^-({\mathcal{D}}_j^-\vartheta^+)_i|\\&\qquad\qquad+
(\rho_i^+)^2\,|{\mathcal{L}}_j\vartheta^+_{i+he_j}
+{\mathcal{L}}_j\vartheta^+_{i-he_j}-2{\mathcal{L}}_j\vartheta^+_{i}|\Big).
\end{split}
\end{equation}
Also,
from~\eqref{ILAN-3} and~\eqref{ILAN-4}, we conclude that
\begin{equation}\label{NMDAME0} \begin{split}&
\sum_{j=1}^2\Big({\mathcal{D}}_j^+(\rho^2 {\mathcal{D}}_j^+\vartheta^+)_i+
{\mathcal{D}}_j^-(\rho^2 {\mathcal{D}}_j^-\vartheta^+)_i\Big)\,\\ =\,&
2(\rho_i^+)^2\;\sum_{j=1}^2{\mathcal{L}}_j\vartheta^+_{i}+\sum_{j=1}^2\Big(
{\mathcal{D}}_j^+\vartheta^+_i
\;{\mathcal{D}}_j^+(\rho_i^+)^2
+{\mathcal{D}}_j^-\vartheta^+_i
\;{\mathcal{D}}_j^-(\rho_i^+)^2
\Big)
+\epsilon_i^{(9)}
\\=\,&
2(\rho_i^+)^2\;{\mathcal{L}}\vartheta^+_{i}+\sum_{j=1}^2\Big(
{\mathcal{D}}_j^+\vartheta^+_i
\;{\mathcal{D}}_j^+(\rho_i^+)^2
+{\mathcal{D}}_j^-\vartheta^+_i
\;{\mathcal{D}}_j^-(\rho_i^+)^2
\Big)
+\epsilon_i^{(9)},\end{split}\end{equation}
where we used~\eqref{2.2BIS}.

Now, in light of~\eqref{PROFGRAD-2} we note that
\begin{equation}\label{NMDAME1} {\mathcal{D}}_j^-(\rho_i^+)^2=
(\rho^+_{i-he_j}+\rho_i^+)\;{\mathcal{D}}_j^-\rho_i^+=2\rho_i^+\;{\mathcal{D}}_j^-\rho_i^++\epsilon_i^{(10,-,j)},
\end{equation}
where
\begin{equation}\label{3.22PRE} \epsilon_i^{(10,-,j)}:=(\rho^+_{i-he_j}-\rho_i^+)\;{\mathcal{D}}_j^-\rho_i^+.\end{equation}
In the same way, we see that
\begin{equation}\label{NMDAME2} {\mathcal{D}}_j^+(\rho_i^+)^2=2\rho_i^+\;{\mathcal{D}}_j^+\rho_i^++\epsilon_i^{(10,+,j)},
\end{equation}
where
\begin{equation}\label{3.22BIS} \epsilon_i^{(10,+,j)}:=(\rho^+_{i+he_j}-\rho_i^+)\;{\mathcal{D}}_j^+\rho_i^+.\end{equation}
By inserting~\eqref{NMDAME1} and~\eqref{NMDAME2} into~\eqref{NMDAME0}, we thereby
find that
\begin{eqnarray*}&&
\sum_{j=1}^2\Big({\mathcal{D}}_j^+(\rho^2 {\mathcal{D}}_j^+\vartheta^+)_i+
{\mathcal{D}}_j^-(\rho^2 {\mathcal{D}}_j^-\vartheta^+)_i\Big)\\&=&
2(\rho_i^+)^2\;{\mathcal{L}}\vartheta^+_{i}+\sum_{j=1}^2\Big(
{\mathcal{D}}_j^+\vartheta^+_i
\;\big(2\rho_i^+\;{\mathcal{D}}_j^+\rho_i^++\epsilon_i^{(10,+,j)}\big)
+{\mathcal{D}}_j^-\vartheta^+_i
\;\big(2\rho_i^+\;{\mathcal{D}}_j^-\rho_i^++\epsilon_i^{(10,-,j)}\big)
\Big)
+\epsilon_i^{(9)}
\\&=&2\rho_i^+
\Bigg[
\rho_i^+\;{\mathcal{L}}\vartheta^+_{i}+\sum_{j=1}^2\Big(
{\mathcal{D}}_j^+\vartheta^+_i
\; {\mathcal{D}}_j^+\rho_i^+
+{\mathcal{D}}_j^-\vartheta^+_i
\;{\mathcal{D}}_j^-\rho_i^+
\Big)
\Bigg]+\epsilon_i^{(11)}
\end{eqnarray*}
where
\begin{equation}\label{spearr} \epsilon_i^{(11)}:=\epsilon_i^{(9)}+
\sum_{j=1}^2\Big(
{\mathcal{D}}_j^+\vartheta^+_i\;\epsilon_i^{(10,+,j)}+
{\mathcal{D}}_j^-\vartheta^+_i\;\epsilon_i^{(10,-,j)}
\Big).\end{equation}
Gathering this information with~\eqref{0-02erkgar},
we conclude that \eqref{KS:0olsdPSPD} holds true with
$$
\epsilon_i^{\star} := 2\rho_i^+\,\epsilon_i^{(7)}+\epsilon_i^{(11)}.
$$

We remark that
\begin{equation*}
\begin{split}
|\epsilon_i^{\star}|\,&\le2\rho_i^+\,|\epsilon_i^{(7)}|+|\epsilon_i^{(11)}|
\\&\le
2e^{2\pi}h\, \sum_{j=1}^2
(\rho_i^+)^2\,
\Big( |{\mathcal{D}}_j^+ \vartheta^+_i|^3+|{\mathcal{D}}_j^- \vartheta^+_i|^3\Big)
\\&\qquad\qquad+
2e^{2\pi}h\,\sum_{j=1}^2\rho_i^+
\big( |{\mathcal{D}}_j^+\rho_i^+|\;|{\mathcal{D}}_j^+ \vartheta^+_i|^2
+
|{\mathcal{D}}_j^-\rho_i^+|\;|{\mathcal{D}}_j^- \vartheta^+_i|^2
\big)+
2 \kappa_0^+ h \rho_i^+\\&\qquad\qquad
+
|\epsilon_i^{(9)}|+
\sum_{j=1}^2\Big|
{\mathcal{D}}_j^+\vartheta^+_i\;\epsilon_i^{(10,+,j)}+
{\mathcal{D}}_j^-\vartheta^+_i\;\epsilon_i^{(10,-,j)}
\Big|
\\ &\le
2e^{2\pi}h\, \sum_{j=1}^2
(\rho_i^+)^2\,
\Big( |{\mathcal{D}}_j^+ \vartheta^+_i|^3+|{\mathcal{D}}_j^- \vartheta^+_i|^3\Big)\\&\qquad\qquad+
2e^{2\pi}h\,\sum_{j=1}^2\rho_i^+
\big( |{\mathcal{D}}_j^+\rho_i^+|\;|{\mathcal{D}}_j^+ \vartheta^+_i|^2
+
|{\mathcal{D}}_j^-\rho_i^+|\;|{\mathcal{D}}_j^- \vartheta^+_i|^2
\big)+
2 \kappa_0^+ h \rho_i^+
\\&\qquad\qquad+
\sum_{j=1}^2\Big(
|\rho^2_{i+he_j}-(\rho_i^+)^2|\,|{\mathcal{D}}_j^+({\mathcal{D}}_j^+\vartheta^+)_i|+
|\rho^2_{i-he_j}-(\rho_i^+)^2|\,|{\mathcal{D}}_j^-({\mathcal{D}}_j^-\vartheta^+)_i|\\&\qquad\qquad\qquad+
(\rho_i^+)^2\,|{\mathcal{L}}_j\vartheta^+_{i+he_j}
+{\mathcal{L}}_j\vartheta^+_{i-he_j}-2{\mathcal{L}}_j\vartheta^+_{i}|\Big)\\&\qquad\qquad+
\sum_{j=1}^2\Big(
|\rho^+_{i+he_j}-\rho_i^+|\,|{\mathcal{D}}_j^+\rho_i^+|\,
|{\mathcal{D}}_j^+\vartheta^+_i|
+|\rho^+_{i-he_j}-\rho_i^+|\,|{\mathcal{D}}_j^-\rho_i^+|\,|{\mathcal{D}}_j^-\vartheta^+_i|\Big),
\end{split}
\end{equation*}
thanks to \eqref{spearr0}, 
\eqref{spearr},
\eqref{EPS9}, \eqref{3.22PRE}, and~\eqref{3.22BIS}. Thus, \eqref{eq:stimaerrore eps12 in enunciato lemma} easily follows.
\end{proof}

Theorem \ref{DGDG} will now be obtained by means of a new discrete quantitative Liouville-type result. Our argument is inspired by \cite[Proof of Theorem 1.8]{MR1655510}.

\begin{proof}[Proof of Theorem \ref{DGDG}]
We let~$R>2(1+h)$, to be taken as large as we wish in what follows,
and~$\varphi\in C^\infty(\R^2,[0,1])$ with~$\varphi=1$ in~$B_1$ and~$\varphi=0$ in~$\R^2\setminus B_2$. For every~$i\in h\Z^2$, we set~$\varphi^{(R)}_i:=\varphi(i/R)$ and
\begin{equation}\label{QUAD}
\tau^{(R)}_i:=\left(\varphi^{(R)}_i\right)^2.
\end{equation}
Using~\eqref{BYPa2}, recalling~\eqref{K33},
and setting~$\widetilde\vartheta^+_i:=\vartheta^+_i-\vartheta^+_\infty$
we have that
\begin{eqnarray*}
\sum_{{1\le j\le 2}\atop{i\in h\Z^2}}
{\mathcal{D}}_j^-(\rho^2 {\mathcal{D}}_j^-\vartheta^+)_i\,(\tau^{(R)}\widetilde\vartheta^+)_i=
-\sum_{{1\le j\le 2}\atop{i\in h\Z^2}}
(\rho_i^+)^2 \,{\mathcal{D}}_j^-\vartheta^+_i\;{\mathcal{D}}_j^+(\tau^{(R)}\widetilde\vartheta^+)_i,
\end{eqnarray*}
and accordingly, in view of~\eqref{PROFGRAD-1},
\begin{eqnarray*}&&
\sum_{{1\le j\le 2}\atop{i\in h\Z^2}}
{\mathcal{D}}_j^-(\rho^2 {\mathcal{D}}_j^-\vartheta^+)_i\,(\tau^{(R)}\widetilde\vartheta^+)_i\\&=&
-\sum_{{1\le j\le 2}\atop{i\in h\Z^2}}
(\rho_i^+)^2 \,{\mathcal{D}}_j^-\vartheta^+_i\;\frac{\tau^{(R)}_{i+he_j}+\tau^{(R)}_i}{2}\;{\mathcal{D}}_j^+\vartheta^+_i
-\sum_{{1\le j\le 2}\atop{i\in h\Z^2}}
(\rho_i^+)^2 \,{\mathcal{D}}_j^-\vartheta^+_i\;\frac{\widetilde\vartheta^+_{i+he_j}+\widetilde\vartheta^+_i}{2}\;{\mathcal{D}}_j^+\tau^{(R)}_i,
\end{eqnarray*}
and similarly
\begin{eqnarray*}&&
\sum_{{1\le j\le 2}\atop{i\in h\Z^2}}
{\mathcal{D}}_j^+(\rho^2 {\mathcal{D}}_j^+\vartheta^+)_i\,(\tau^{(R)}\widetilde\vartheta^+)_i\\&=&
-\sum_{{1\le j\le 2}\atop{i\in h\Z^2}}
(\rho_i^+)^2 \,{\mathcal{D}}_j^+\vartheta^+_i\;\frac{\tau^{(R)}_{i-he_j}+\tau^{(R)}_i}{2}\;{\mathcal{D}}_j^-\vartheta^+_i
-\sum_{{1\le j\le 2}\atop{i\in h\Z^2}}
(\rho_i^+)^2 \,{\mathcal{D}}_j^+\vartheta^+_i\;\frac{\widetilde\vartheta^+_{i-he_j}+\widetilde\vartheta^+_i}{2}\;{\mathcal{D}}_j^-\tau^{(R)}_i.
\end{eqnarray*}
{F}rom these considerations, we obtain that
\begin{equation}\label{0909093}
\begin{split}&
\sum_{{1\le j\le 2}\atop{i\in h\Z^2}}\Big(
{\mathcal{D}}_j^+(\rho^2 {\mathcal{D}}_j^+\vartheta^+)_i+
{\mathcal{D}}_j^-(\rho^2 {\mathcal{D}}_j^-\vartheta^+)_i\Big)\,(\tau^{(R)}\widetilde\vartheta^+)_i\\=\,&
-\frac12\sum_{{1\le j\le 2}\atop{i\in h\Z^2}}
(\rho_i^+)^2 \;\big(\tau^{(R)}_{i+he_j}+\tau^{(R)}_{i-he_j}+2\tau^{(R)}_i\big)\;{\mathcal{D}}_j^-\vartheta^+_i\,{\mathcal{D}}_j^+\vartheta^+_i\\&\quad
-\frac12\sum_{{1\le j\le 2}\atop{i\in h\Z^2}}
(\rho_i^+)^2 \,\big(\widetilde\vartheta^+_{i-he_j}+\widetilde\vartheta^+_i\big)\;{\mathcal{D}}_j^+\vartheta^+_i\;{\mathcal{D}}_j^-\tau^{(R)}_i-\frac12\sum_{{1\le j\le 2}\atop{i\in h\Z^2}}
(\rho_i^+)^2 \,\big(\widetilde\vartheta^+_{i+he_j}+\widetilde\vartheta^+_i\big)\;{\mathcal{D}}_j^-\vartheta^+_i\;{\mathcal{D}}_j^+\tau^{(R)}_i.
\end{split}\end{equation}
Now we define
\begin{eqnarray*}&& \epsilon^{(12,+,R)}:=
\frac14\sum_{{1\le j\le 2}\atop{i\in h\Z^2}}
(\rho_i^+)^2 \;\big(\tau^{(R)}_{i+he_j}+\tau^{(R)}_{i-he_j}+2\tau^{(R)}_i\big)\;
\big({\mathcal{D}}_j^-\vartheta^+_i-{\mathcal{D}}_j^+\vartheta^+_i\big)\,{\mathcal{D}}_j^+\vartheta^+_i
\\{\mbox{and }}
&&
\epsilon^{(12,-,R)}:=
\frac14\sum_{{1\le j\le 2}\atop{i\in h\Z^2}}
(\rho_i^+)^2 \;\big(\tau^{(R)}_{i+he_j}+\tau^{(R)}_{i-he_j}+2\tau^{(R)}_i\big)\;
\big({\mathcal{D}}_j^+\vartheta^+_i-{\mathcal{D}}_j^-\vartheta^+_i\big)\,{\mathcal{D}}_j^-\vartheta^+_i
\end{eqnarray*}
and we write
\begin{eqnarray*}&&
\frac12\sum_{{1\le j\le 2}\atop{i\in h\Z^2}}
(\rho_i^+)^2 \;\big(\tau^{(R)}_{i+he_j}+\tau^{(R)}_{i-he_j}+2\tau^{(R)}_i\big)\;{\mathcal{D}}_j^-\vartheta^+_i\,{\mathcal{D}}_j^+\vartheta^+_i\\
&=&
\frac14\sum_{{1\le j\le 2}\atop{i\in h\Z^2}}
(\rho_i^+)^2 \;\big(\tau^{(R)}_{i+he_j}+\tau^{(R)}_{i-he_j}+2\tau^{(R)}_i\big)\;\Big(\big({\mathcal{D}}_j^-\vartheta^+_i-{\mathcal{D}}_j^+\vartheta^+_i\big)+{\mathcal{D}}_j^+\vartheta^+_i\Big)\,{\mathcal{D}}_j^+\vartheta^+_i\\&&\quad+
\frac14\sum_{{1\le j\le 2}\atop{i\in h\Z^2}}
(\rho_i^+)^2 \;\big(\tau^{(R)}_{i+he_j}+\tau^{(R)}_{i-he_j}+2\tau^{(R)}_i\big)\;{\mathcal{D}}_j^-\vartheta^+_i\,\Big(\big({\mathcal{D}}_j^+\vartheta^+_i-{\mathcal{D}}_j^-\vartheta^+_i\big)+{\mathcal{D}}_j^-\vartheta^+_i\Big)
\\&=&\frac14\sum_{{1\le j\le 2}\atop{i\in h\Z^2}}
(\rho_i^+)^2 \;\big(\tau^{(R)}_{i+he_j}+\tau^{(R)}_{i-he_j}+2\tau^{(R)}_i\big)\;\Big(
|{\mathcal{D}}_j^+\vartheta^+_i|^2+|{\mathcal{D}}_j^-\vartheta^+_i|^2
\Big)+\epsilon^{(12,+,R)}+\epsilon^{(12,-,R)}.
\end{eqnarray*}
Plugging this information into~\eqref{0909093}, we gather that
\begin{equation}\label{KS:0olsdPSPD2}
\begin{split}&
\sum_{{1\le j\le 2}\atop{i\in h\Z^2}}\Big(
{\mathcal{D}}_j^+(\rho^2 {\mathcal{D}}_j^+\vartheta^+)_i+
{\mathcal{D}}_j^-(\rho^2 {\mathcal{D}}_j^-\vartheta^+)_i\Big)\,(\tau^{(R)}\widetilde\vartheta^+)_i\\=\,&
-\frac14\sum_{{1\le j\le 2}\atop{i\in h\Z^2}}
(\rho_i^+)^2 \;\big(\tau^{(R)}_{i+he_j}+\tau^{(R)}_{i-he_j}+2\tau^{(R)}_i\big)\;\Big(
|{\mathcal{D}}_j^+\vartheta^+_i|^2+|{\mathcal{D}}_j^-\vartheta^+_i|^2
\Big)\\&\quad
-\frac12\sum_{{1\le j\le 2}\atop{i\in h\Z^2}}
(\rho_i^+)^2 \,\big(\widetilde\vartheta^+_{i-he_j}+\widetilde\vartheta^+_i\big)\;{\mathcal{D}}_j^+\vartheta^+_i\;{\mathcal{D}}_j^-\tau^{(R)}_i-\frac12\sum_{{1\le j\le 2}\atop{i\in h\Z^2}}
(\rho_i^+)^2 \,\big(\widetilde\vartheta^+_{i+he_j}+\widetilde\vartheta^+_i\big)\;{\mathcal{D}}_j^-\vartheta^+_i\;{\mathcal{D}}_j^+\tau^{(R)}_i-\epsilon^{(12,R)},
\end{split}\end{equation}
where
$$\epsilon^{(12,R)}:=\epsilon^{(12,+,R)}+\epsilon^{(12,-,R)}.$$
It is interesting to observe that
$$ {\mathcal{D}}_j^-\vartheta^+_i-{\mathcal{D}}_j^+\vartheta^+_i=\frac{\big(
\vartheta^+_i-\vartheta^+_{i-he_j}
\big)-\big(
\vartheta^+_{i+he_j}-\vartheta^+_i
\big)}{h}=-h{\mathcal{L}}_j\vartheta^+_i
=-h{\mathcal{D}}_j^+ ({\mathcal{D}}_j^+ \vartheta^+)_{i-he_j},$$
thanks to~\eqref{ITER-01}, which
yields that
$$| \epsilon^{(12,+,R)}|\le
\sum_{{1\le j\le 2}\atop{i\in h\Z^2}}
(\rho_i^+)^2 \;
\big|{\mathcal{D}}_j^-\vartheta^+_i-{\mathcal{D}}_j^+\vartheta^+_i\big|\,|{\mathcal{D}}_j^+\vartheta^+_i|\le
h\sum_{{1\le j\le 2}\atop{i\in h\Z^2}}
(\rho_i^+)^2 \,|{\mathcal{D}}_j^+\vartheta^+_i|\,|{\mathcal{D}}_j^+ ({\mathcal{D}}_j^+ \vartheta^+)_{i-he_j}|
$$
and similarly
$$
| \epsilon^{(12,-,R)}|\le
h\sum_{{1\le j\le 2}\atop{i\in h\Z^2}}
(\rho_i^+)^2 \,|{\mathcal{D}}_j^-\vartheta^+_i|\,|{\mathcal{D}}_j^- ({\mathcal{D}}_j^- \vartheta^+)_{i+he_j}|.
$$
This and~\eqref{K2} give that
\begin{equation}\label{EPS13}
| \epsilon^{(12,R)}|\le
h\sum_{{1\le j\le 2}\atop{i\in h\Z^2}}
(\rho_i^+)^2 \,\Big(
|{\mathcal{D}}_j^+\vartheta^+_i|\,|{\mathcal{D}}_j^+ ({\mathcal{D}}_j^+ \vartheta^+)_{i-he_j}|+
|{\mathcal{D}}_j^-\vartheta^+_i|\,|{\mathcal{D}}_j^- ({\mathcal{D}}_j^- \vartheta^+)_{i+he_j}|
\Big)\le \kappa_2^+\,h .
\end{equation}
We now combine~\eqref{KS:0olsdPSPD} and~\eqref{KS:0olsdPSPD2} to see that
\begin{equation}\label{C5rahrpp}
\begin{split}&
\frac14\sum_{{1\le j\le 2}\atop{i\in h\Z^2}}
(\rho_i^+)^2 \;\big(\tau^{(R)}_{i+he_j}+\tau^{(R)}_{i-he_j}+2\tau^{(R)}_i\big)\;\Big(
|{\mathcal{D}}_j^+\vartheta^+_i|^2+|{\mathcal{D}}_j^-\vartheta^+_i|^2
\Big)\\&\quad+
\frac12\sum_{{1\le j\le 2}\atop{i\in h\Z^2}}
(\rho_i^+)^2 \,\big(\widetilde\vartheta^+_{i-he_j}+\widetilde\vartheta^+_i\big)\;{\mathcal{D}}_j^+\vartheta^+_i\;{\mathcal{D}}_j^-\tau^{(R)}_i+\frac12\sum_{{1\le j\le 2}\atop{i\in h\Z^2}}
(\rho_i^+)^2 \,\big(\widetilde\vartheta^+_{i+he_j}+\widetilde\vartheta^+_i\big)\;{\mathcal{D}}_j^-\vartheta^+_i\;{\mathcal{D}}_j^+\tau^{(R)}_i=\epsilon^{(14,R)},
\end{split}
\end{equation}
where
$$ \epsilon^{(13,R)}:=\sum_{i\in h\Z^2} \epsilon^{\star}_i (\tau^{(R)}\widetilde\vartheta^+)_i,$$
and
$$\epsilon^{(14,R)}:=-\epsilon^{(12,R)}-\epsilon^{(13,R)}.$$
We exploit~\eqref{K33} and~\eqref{eq:stimaerrore eps12 in enunciato lemma},
and we point out that
\begin{equation}\label{EPS14}
\begin{split}
|\epsilon^{(13,R)}|\,&\le
\sum_{i\in h\Z^2}| \epsilon^{\star}_i |\,|\widetilde\vartheta^+_i|\\&\le
2e^{2\pi}h\, \sum_{{1\le j\le2}\atop{i\in h\Z^2}}
(\rho_i^+)^2\,
\Big( |{\mathcal{D}}_j^+ \vartheta^+_i|^3+|{\mathcal{D}}_j^- \vartheta^+_i|^3\Big)\,|\widetilde\vartheta^+_i|
\\&\qquad+
2e^{2\pi}h\,\sum_{{1\le j\le2}\atop{i\in h\Z^2}}\rho_i^+\,
\big( |{\mathcal{D}}_j^+\rho_i^+|\;|{\mathcal{D}}_j^+ \vartheta^+_i|^2
+
|{\mathcal{D}}_j^-\rho_i^+|\;|{\mathcal{D}}_j^- \vartheta^+_i|^2
\big)\,|\widetilde\vartheta^+_i|+
2 \kappa_0^+ h \sum_{i\in h\Z^2} \rho_i^+ \,|\widetilde\vartheta^+_i|\\&\qquad
+
h  \sum_{{1\le j\le2}\atop{i\in h\Z^2}} \Big(
|{\mathcal{D}}_j^+(\rho_i^+)^2| \,|{\mathcal{D}}_j^+({\mathcal{D}}_j^+\vartheta^+)_i|+
|{\mathcal{D}}_j^-(\rho_i^+)^2| \,|{\mathcal{D}}_j^-({\mathcal{D}}_j^-\vartheta^+)_i|+
 h \,  (\rho_i^+)^2 \,  | {\mathcal{L}_j^2}\vartheta^+_{i} |\Big) \,|\widetilde\vartheta^+_i|
 \\&\qquad +
h  \sum_{{1\le j\le2}\atop{i\in h\Z^2}} \Big(
|{\mathcal{D}}_j^+\rho_i^+|^2\,
|{\mathcal{D}}_j^+\vartheta^+_i|
+ |{\mathcal{D}}_j^-\rho_i^+|^2\,|{\mathcal{D}}_j^-\vartheta^+_i|\Big) \,|\widetilde\vartheta^+_i|
\\&=
\big(2e^{2\pi}\kappa_3^++
2e^{2\pi}\kappa_4^++
2 \kappa_5^++
\kappa_6^++\kappa_7^+\big)h.
\end{split}
\end{equation}
Now, we recall~\eqref{QUAD} and we exploit~\eqref{PROFGRAD-2}
to write that
\begin{eqnarray*}
{\mathcal{D}}_j^-\tau^{(R)}_i=
\left(\varphi^{(R)}_{i-he_j}+\varphi^{(R)}_i\right)\;{\mathcal{D}}_j^-\varphi^{(R)}_i,
\end{eqnarray*}
and thus, recalling~\eqref{PIKS},
\begin{equation}\label{9oiknddfP--0}
\left| \sum_{{1\le j\le 2}\atop{i\in h\Z^2}}
(\rho_i^+)^2 \,\big(\widetilde\vartheta^+_{i-he_j}+\widetilde\vartheta^+_i\big)\;{\mathcal{D}}_j^+\vartheta^+_i\;{\mathcal{D}}_j^-\tau^{(R)}_i\right|\le 4 \pi
\sum_{{1\le j\le 2}\atop{i\in h\Z^2}}
(\rho_i^+)^2 \,|{\mathcal{D}}_j^+\vartheta^+_i|\left(\varphi^{(R)}_{i-he_j}+\varphi^{(R)}_i\right)\,|{\mathcal{D}}_j^-\varphi^{(R)}_i|.
\end{equation}
We also observe that if~$|i|\le R/2$
then
$$|i-he_j|\le|i|+h\le \frac{R}{2}+h<R,$$
and consequently
\begin{equation}\label{9oiknddfP--1} h{\mathcal{D}}_j^-\varphi^{(R)}_i=
\varphi^{(R)}_i-\varphi^{(R)}_{i-he_j}=\varphi\left(\frac{i}{R}\right)-
\varphi\left(\frac{i-he_j}{R}\right)
=1-1=0.\end{equation}
Similarly, if~$|i|\ge 4R$
then
$$|i-he_j|\ge|i|-h\ge 4R-h>2R,$$
and, as a consequence,
\begin{equation}\label{9oiknddfP--2}h {\mathcal{D}}_j^-\varphi^{(R)}_i=
\varphi^{(R)}_i-\varphi^{(R)}_{i-he_j}=\varphi\left(\frac{i}{R}\right)-
\varphi\left(\frac{i-he_j}{R}\right)
=0-0=0.\end{equation}
By collecting the results in~\eqref{9oiknddfP--0},
\eqref{9oiknddfP--1}, and~\eqref{9oiknddfP--2},
we conclude that
\begin{eqnarray*}&&
\left| \sum_{{{1\le j\le 2}\atop{i\in h\Z^2}}}
(\rho_i^+)^2 \,\big(\widetilde\vartheta^+_{i-he_j}+\widetilde\vartheta^+_i\big)\;{\mathcal{D}}_j^+\vartheta^+_i\;{\mathcal{D}}_j^-\tau^{(R)}_i\right|\le 4 \pi
\sum_{ {{1\le j\le 2}\atop{i\in h\Z^2}}\atop{|i|\in[R/2,4R]}}
(\rho_i^+)^2 \,|{\mathcal{D}}_j^+\vartheta^+_i|\left(\varphi^{(R)}_{i-he_j}+\varphi^{(R)}_i\right)\,|{\mathcal{D}}_j^-\varphi^{(R)}_i|\\&&\qquad\le
4 \pi\,\sqrt{
\sum_{ {{1\le j\le 2}\atop{i\in h\Z^2}}\atop{|i|\in[R/2,4R]}}
(\rho_i^+)^2 \left(\varphi^{(R)}_{i-he_j}+\varphi^{(R)}_i\right)^2\,|{\mathcal{D}}_j^+\vartheta^+_i|^2}\;
\sqrt{
\sum_{ {{1\le j\le 2}\atop{i\in h\Z^2}}\atop{|i|\in[R/2,4R]}}
(\rho_i^+)^2 \,|{\mathcal{D}}_j^-\varphi_i^{(R)}|^2}.
\end{eqnarray*}
As a consequence,
since
$$ \left(\varphi^{(R)}_{i-he_j}+\varphi^{(R)}_i\right)^2\le
2 \left(\big(\varphi^{(R)}_{i-he_j}\big)^2+\big(\varphi^{(R)}_i\big)^2\right)=
2 \left(\tau^{(R)}_{i-he_j}+\tau^{(R)}_i\right),$$
we obtain that
\begin{equation}\label{SLCRHle1}
\begin{split}&
\left| \sum_{{{1\le j\le 2}\atop{i\in h\Z^2}}}
(\rho_i^+)^2 \,\big(\widetilde\vartheta^+_{i-he_j}+\widetilde\vartheta^+_i\big)\;{\mathcal{D}}_j^+\vartheta^+_i\;{\mathcal{D}}_j^-\tau^{(R)}_i\right|\\&\qquad
\le 4\sqrt{2} \, \pi\,\sqrt{
\sum_{ {{1\le j\le 2}\atop{i\in h\Z^2}}\atop{|i|\in[R/2,4R]}}
(\rho_i^+)^2 \,\big(\tau^{(R)}_{i-he_j}+\tau^{(R)}_i\big)\,|{\mathcal{D}}_j^+\vartheta^+_i|^2}\;
\sqrt{
\sum_{ {{1\le j\le 2}\atop{i\in h\Z^2}}\atop{|i|\in[R/2,4R]}}
(\rho_i^+)^2 \,|{\mathcal{D}}_j^-\varphi^{(R)}_i|^2}
\\&\qquad
\le 4\sqrt{2} \, \pi\,\sqrt{
\sum_{ {{1\le j\le 2}\atop{i\in h\Z^2}}\atop{|i|\in[R/2,4R]}}
(\rho_i^+)^2 \,\big(\tau^{(R)}_{i+he_j}+\tau^{(R)}_{i-he_j}+2\tau^{(R)}_i\big)\,\Big(
|{\mathcal{D}}_j^+\vartheta^+_i|^2+
|{\mathcal{D}}_j^-\vartheta^+_i|^2\Big)}\\&\qquad\qquad\cdot
\sqrt{
\sum_{ {{1\le j\le 2}\atop{i\in h\Z^2}}\atop{|i|\in[R/2,4R]}}
(\rho_i^+)^2 \,\Big(|{\mathcal{D}}_j^+\varphi^{(R)}_i|^2+|{\mathcal{D}}_j^-\varphi^{(R)}_i|^2\Big)}
.\end{split}
\end{equation}
Similarly,
\begin{equation}\label{SLCRHle2}
\begin{split}&
\left| \sum_{{{1\le j\le 2}\atop{i\in h\Z^2}}}
(\rho_i^+)^2 \,\big(\widetilde\vartheta^+_{i+he_j}+\widetilde\vartheta^+_i\big)\;{\mathcal{D}}_j^-\vartheta^+_i\;{\mathcal{D}}_j^+\tau^{(R)}_i\right|\\&\qquad
\le 4\sqrt{2} \, \pi\,\sqrt{
\sum_{ {{1\le j\le 2}\atop{i\in h\Z^2}}\atop{|i|\in[R/2,4R]}}
(\rho_i^+)^2 \,\big(\tau^{(R)}_{i+he_j}+\tau^{(R)}_{i-he_j}+2\tau^{(R)}_i\big)\,\Big(
|{\mathcal{D}}_j^+\vartheta^+_i|^2+
|{\mathcal{D}}_j^-\vartheta^+_i|^2\Big)}\\&\qquad\qquad\cdot
\sqrt{
\sum_{ {{1\le j\le 2}\atop{i\in h\Z^2}}\atop{|i|\in[R/2,4R]}}
(\rho_i^+)^2 \,\Big(|{\mathcal{D}}_j^+\varphi^{(R)}_i|^2+|{\mathcal{D}}_j^-\varphi^{(R)}_i|^2\Big)}
.\end{split}
\end{equation}
With this, plugging~\eqref{SLCRHle1} and~\eqref{SLCRHle2} into~\eqref{C5rahrpp},
we conclude that
\begin{equation}\label{PIv9}
\begin{split}&
\frac14\sum_{{1\le j\le 2}\atop{i\in h\Z^2}}
(\rho_i^+)^2 \;\big(\tau^{(R)}_{i+he_j}+\tau^{(R)}_{i-he_j}+2\tau^{(R)}_i\big)\;\Big(
|{\mathcal{D}}_j^+\vartheta^+_i|^2+|{\mathcal{D}}_j^-\vartheta^+_i|^2
\Big)\\&\qquad\le|\epsilon^{(14,R)}|+
 4\sqrt{2} \, \pi\,\sqrt{
\sum_{ {{1\le j\le 2}\atop{i\in h\Z^2}}\atop{|i|\in[R/2,4R]}}
(\rho_i^+)^2 \,\big(\tau^{(R)}_{i+he_j}+\tau^{(R)}_{i-he_j}+2\tau^{(R)}_i\big)\,\Big(
|{\mathcal{D}}_j^+\vartheta^+_i|^2+
|{\mathcal{D}}_j^-\vartheta^+_i|^2\Big)}\\&\qquad\qquad\cdot
\sqrt{
\sum_{ {{1\le j\le 2}\atop{i\in h\Z^2}}\atop{|i|\in[R/2,4R]}}
(\rho_i^+)^2 \,\Big(|{\mathcal{D}}_j^+\varphi^{(R)}_i|^2+|{\mathcal{D}}_j^-\varphi^{(R)}_i|^2\Big)}.
\end{split}
\end{equation}
By the Cauchy-Schwarz Inequality, for all~$a$, $b\ge0$,
\begin{eqnarray*}
4\sqrt{2}\,\pi\sqrt{a}\,\sqrt{b}\le \frac{a}{16}+ 128 \,\pi^2 b,
\end{eqnarray*}
and therefore we deduce from~\eqref{PIv9} that
\begin{equation}\label{PIv10}
\begin{split}&
\frac3{16}\sum_{{1\le j\le 2}\atop{i\in h\Z^2}}
(\rho_i^+)^2 \;\big(\tau^{(R)}_{i+he_j}+\tau^{(R)}_{i-he_j}+2\tau^{(R)}_i\big)\;\Big(
|{\mathcal{D}}_j^+\vartheta^+_i|^2+|{\mathcal{D}}_j^-\vartheta^+_i|^2
\Big)\\&\qquad\le|\epsilon^{(14,R)}|+
128 \, \pi^2\,
\sum_{ {{1\le j\le 2}\atop{i\in h\Z^2}}\atop{|i|\in[R/2,4R]}}
(\rho_i^+)^2 \,\Big(|{\mathcal{D}}_j^+\varphi^{(R)}_i|^2+|{\mathcal{D}}_j^-\varphi^{(R)}_i|^2\Big).
\end{split}
\end{equation}
We observe that
\begin{equation*} |{\mathcal{D}}_j^\pm\varphi^{(R)}_i|=\frac1h\,
\left|
\varphi\left(\frac{i\pm he_j}{R}\right)-\varphi\left(\frac{i}{R}\right)
\right|\le \frac{\|\varphi\|_{C^1(\R^2)}}{R}.\end{equation*}
{F}rom this and~\eqref{D23ecs824}, we thereby conclude that
\begin{equation}\label{IN0oqw02} \limsup_{R\nearrow+\infty}
\sum_{ {{1\le j\le 2}\atop{i\in h\Z^2}}\atop{|i|\in[R/2,4R]}}
(\rho_i^+)^2 \,\Big(|{\mathcal{D}}_j^+\varphi^{(R)}_i|^2+|{\mathcal{D}}_j^-\varphi^{(R)}_i|^2\Big)
<+\infty.\end{equation}
Now we observe that
\begin{equation}\label{EPS15}\begin{split}&
\limsup_{R\nearrow+\infty}|\epsilon^{(14,R)}|\le \limsup_{R\nearrow+\infty}
|\epsilon^{(12,R)}|+|\epsilon^{(13,R)}|\\&\qquad\qquad\leq
\left(\kappa_2^++
2e^{2\pi}\kappa_3^++
2e^{2\pi}\kappa_4^++
\kappa_5^++\kappa_6^++\kappa_7^+\right)h<+\infty,\end{split}\end{equation}
due to~\eqref{K2}, \eqref{K33}, \eqref{EPS13},
and~\eqref{EPS14}.

Hence, we thus deduce from~\eqref{PIv10}, \eqref{IN0oqw02} and~\eqref{EPS15}
that
$$ \sum_{{1\le j\le 2}\atop{i\in h\Z^2}}
(\rho_i^+)^2 \;\Big(
|{\mathcal{D}}_j^+\vartheta^+_i|^2+|{\mathcal{D}}_j^-\vartheta^+_i|^2
\Big)<+\infty.$$
In particular,
$$  \limsup_{R\nearrow+\infty}\sum_{ {{1\le j\le 2}\atop{i\in h\Z^2}}\atop{|i|\ge R/2}}
(\rho_i^+)^2 \;\Big(
|{\mathcal{D}}_j^+\vartheta^+_i|^2+|{\mathcal{D}}_j^-\vartheta^+_i|^2
\Big)
=0.$$
Using this information, \eqref{IN0oqw02} and~\eqref{EPS15}
into~\eqref{PIv9},
we thereby obtain that
$$ \frac{1}{4}\,\sum_{{1\le j\le 2}\atop{i\in h\Z^2}}
(\rho_i^+)^2 \;\Big(
|{\mathcal{D}}_j^+\vartheta^+_i|^2+|{\mathcal{D}}_j^-\vartheta^+_i|^2
\Big)\le \limsup_{R\nearrow+\infty}|\epsilon^{(14,R)}|\le \left(
\kappa_2^++
2e^{2\pi}\kappa_3^++
2e^{2\pi}\kappa_4^++
\kappa_5^++\kappa_6^++\kappa_7^+\right)h,$$
thus completing the proof of Theorem~\ref{DGDG}.
\end{proof}

Now, we prove Lemma~\ref{KK-1D}.

\begin{proof}[Proof of Lemma~\ref{KK-1D}] By~\eqref{der-2se}, we know that~$\rho_i^+\ne0$
and~$\rho_i^-\ne0$ for all~$i\in h\Z^2$. Hence, by~\eqref{noinrthoch45}, we find that
$$ {\mathcal{D}}_j^+\vartheta^+_i={\mathcal{D}}_j^+\vartheta^-_i={\mathcal{D}}_j^-\vartheta^+_i={\mathcal{D}}_j^-\vartheta^-_i=0$$
for all~$i\in h\Z^2$ and all~$j\in\{1,2\}$.

As a result, there exist~$\omega^+$ and~$\omega^-\in \C$ such that
\begin{equation}\label{7657654dsfg56} \frac{U_i^+}{|U_i^+|}=\omega^+
\qquad{\mbox{ and }}\qquad \frac{U_i^-}{|U_i^-|}=\omega^-\end{equation}
for all~$i\in h\Z^2$. Also, in light of~\eqref{der-2se}, we know that the imaginary parts of~$\omega^+$
and of~$\omega^-$ are nonzero, and therefore~\eqref{7657654dsfg56} yields that
$$ {\mathcal{D}}_1^\pm u_i= c^\pm\,{\mathcal{D}}_2^\pm u_i\qquad{\mbox{for all }}i\in h\Z^2,$$
where~$c^\pm\in\R$ is the ratio between the real and the imaginary parts of~$\omega^\pm$.
{F}rom this, we obtain~\eqref{Agfoi-doiut},
and accordingly
\begin{equation}\label{-dsuif9jjflman} u_{i\pm he_1}=c^\pm u_{i\pm he_2}+(1-c^\pm) u_i\qquad{\mbox{for all }}i\in h\Z^2.\end{equation}
Now, for all~$m\in\Z$, we define
\begin{equation}\label{-de-w}
\widetilde{u}_{hm}:=u_{hme_2},
\end{equation}
and we prove that~\eqref{k8i-know} holds true.

As a matter of fact, we focus on the proof of~\eqref{k8i-know} when~$k\le0$,
since the proof when~$k\ge0$ is similar. Thus, the proof is by induction over~$|k|$.
When~$k=0$ we have that
\begin{eqnarray*}
&&\sum_{j=0}^{|k|} {\binom {|k|}{j}} (c^{\sigma_k})^j\,(1-c^{\sigma_k})^{|k|-j}\,\widetilde{u}_{h(m+\sigma_k j)}
-u_{(hk,hm)}=\widetilde{u}_{hm}-u_{(0,hm)}=0,
\end{eqnarray*}
thanks to~\eqref{-de-w}, and this gives~\eqref{k8i-know} when~$k=0$.

Moreover, if~$k=-1$, one uses~\eqref{-de-w}
and~\eqref{-dsuif9jjflman} (here, with~$i:=hme_2$ and the
``minus sign'' choice), finding that
\begin{eqnarray*}
&&\sum_{j=0}^{|k|} {\binom {|k|}{j}} (c^{\sigma_k})^j\,(1-c^{\sigma_k})^{|k|-j}\,\widetilde{u}_{h(m+\sigma_k j)}
-u_{(hk,hm)}=
\sum_{j=0}^{1} {\binom {1}{j}} (c^{-})^j\,(1-c^{-})^{1-j}\,\widetilde{u}_{h(m- j)}
-u_{(-h,hm)}\\&&\qquad=
(1-c^{-})\,\widetilde{u}_{hm}+
c^{-}\,\widetilde{u}_{h(m- 1)}
-u_{(-h,hm)}=
(1-c^{-})\,{u}_{hme_2}+
c^{-}\,{u}_{hme_2- he_2}
-u_{hme_2-he_1}=0,
\end{eqnarray*}
which gives~\eqref{k8i-know} when~$k=-1$.

Suppose now that~\eqref{k8i-know} holds true for all integers~$\{0,-1,-2,\dots,k\}$, for some~$k\le-1$
and let us prove it for the integer~$k-1$ (that is equal to~$-|k|-1$).
To this end, we make use of~\eqref{-dsuif9jjflman} 
with~$i:=(hk,hm)$ and the ``minus sign'' choice,
and we see that
\begin{eqnarray*}&&
u_{(h(k-1),hm)}=u_{(hk,hm)-he_1}
=c^- u_{(hk,hm)- he_2}+(1-c^-) u_{(hk,hm)}=
c^{\sigma_k} u_{(hk,h(m-1))}+(1-c^{\sigma_k}) u_{(hk,hm)}.
\end{eqnarray*}
This and the recursive assumption yields that
\begin{eqnarray*}
u_{(h(k-1),hm)}&
=&
\sum_{j=0}^{|k|} {\binom {|k|}{j}} (c^{\sigma_k})^{j+1}\,(1-c^{\sigma_k})^{|k|-j}\,\widetilde{u}_{h(m-1- j)}\\&&\qquad
+
\sum_{j=0}^{|k|} {\binom {|k|}{j}} (c^{\sigma_k})^j\,(1-c^{\sigma_k})^{|k|-j+1}\,\widetilde{u}_{h(m- j)}
\\&
=&
\sum_{J=1}^{|k|+1} {\binom {|k|}{J-1}} (c^{\sigma_k})^{J}\,(1-c^{\sigma_k})^{|k|-J+1}\,\widetilde{u}_{h(m-J)}\\&&\qquad
+
\sum_{j=0}^{|k|} {\binom {|k|}{j}} (c^{\sigma_k})^j\,(1-c^{\sigma_k})^{|k|-j+1}\,\widetilde{u}_{h(m- j)}\\
&=&(c^{\sigma_k})^{|k|+1}\,\widetilde{u}_{h(m-|k|-1)}+
(1-c^{\sigma_k})^{|k|+1}\,\widetilde{u}_{h m}\\&&\qquad+
\sum_{j=1}^{|k|} \left({\binom {|k|}{j-1}}+{\binom {|k|}{j}}\right) (c^{\sigma_k})^j\,(1-c^{\sigma_k})^{|k|-j+1}\,\widetilde{u}_{h(m- j)}.
\end{eqnarray*}
Therefore, noticing that~$\sigma_{k-1}=-=\sigma_k$, and also that~$|k-1|=-k+1=|k|+1$, and
using the Pascal's triangle recurrence relation
$$ {\binom {|k|}{j-1}}+{\binom {|k|}{j}}={\binom {|k|+1}{j}},$$
we conclude that
\begin{eqnarray*}
u_{(h(k-1),hm)}&
=&(c^{\sigma_{k-1}})^{|k|+1}\,\widetilde{u}_{h(m-|k|-1)}+
(1-c^{\sigma_{k-1}})^{|k|+1}\,\widetilde{u}_{h m}\\&&\qquad+
\sum_{j=1}^{|k|} {\binom {|k|+1}{j}} (c^{\sigma_{k-1}})^j\,(1-c^{\sigma_{k-1}})^{|k|-j+1}\,\widetilde{u}_{h(m- j)}\\&=&
\sum_{j=0}^{|k|+1} {\binom {|k|+1}{j}} (c^{\sigma_{k-1}})^j\,(1-c^{\sigma_{k-1}})^{|k|-j+1}\,\widetilde{u}_{h(m- j)}\\&=&
\sum_{j=0}^{|k-1|} {\binom {|k-1|}{j}} (c^{\sigma_{k-1}})^j\,(1-c^{\sigma_{k-1}})^{|k-1|-j}\,\widetilde{u}_{h(m+\sigma_{k-1} j)}
\end{eqnarray*}
that finishes the proof of~\eqref{k8i-know}.
\end{proof}

Concering the one-dimensional properties of the solutions, we remark that there exist one-dimensional
solutions for discrete semilinear equations:
more specifically,
given a strictly monotone function~$\varphi:\R\to\R$ and
a vector~$\omega\in\R^2\setminus\{0\}$,
it is always possible to construct a function~$f:\R\to\R$ such that,
setting~$u_i:=\varphi(\omega\cdot i)$ for every~$i\in h\Z^2$,
we have that~$u$ is a solution of the discrete
semilinear equation
\begin{equation}\label{0oijhb8uj-0987654-pKCjHmCKOS}
{\mathcal{L}}u_i=f(u_i)\qquad{\mbox{ for all }}i\in h\Z^2.
\end{equation}
To check this claim, we consider a strictly monotone function~$\varphi:\R\to\R$ and we denote by~$\varphi^{-1}$ its inverse. In this way,
$$ \varphi^{-1}(\varphi(t))=t\qquad{\mbox{for every }}t\in\R.$$
Thus, for every~$x\in\R^2$ we define
$$ \psi(x):=\frac1{h^2}\sum_{j=1}^2
\Big(\varphi(\omega\cdot (x+he_j))+\varphi(\omega\cdot (x-he_j))-2\varphi(\omega\cdot x)\Big).$$
For every~$r\in\R$, let also
$$ \Psi(r):=\psi\left(\frac{h\omega r}{|\omega|^2}\right).$$
Notice that, for each~$i\in\Z^2$,
\begin{eqnarray*}
\Psi(\omega\cdot i)&=&
\psi\left(\frac{h\omega (\omega\cdot i)}{|\omega|^2}\right)
\\&=&
\frac1{h^2}\sum_{j=1}^2
\Big(\varphi(h\omega\cdot (i+e_j))+\varphi(h\omega\cdot (i+e_j))-2\varphi(h\omega\cdot i)\Big)\\&=&
\psi(hi).
\end{eqnarray*}
Hence, for each~$i\in h\Z^2$ we set
\begin{equation}\label{0DEDFDYB8im10laAmmdLLA} u_i:=\varphi(\omega\cdot i)\end{equation}
and we check that~$u$ is a solution of~\eqref{0oijhb8uj-0987654-pKCjHmCKOS}, with
$$ f(t):=\Psi(\varphi^{-1}(t))\qquad{\mbox{for all }}t\in\R.$$
Indeed, by construction,
\begin{eqnarray*}
&&\frac1{h^2}\sum_{j=1}^2
\big(u_{i+he_j}+u_{i-he_j}-2u_i\big)
\\&=&
\frac1{h^2}\sum_{j=1}^2
\Big(\varphi(\omega\cdot (i+he_j))+\varphi(\omega\cdot (i-he_j))-2\varphi(\omega\cdot i)\Big)
\\
&=&\psi(hi)\\
&=& \Psi(\omega\cdot i)\\
&=& \Psi\Big(
\varphi^{-1}\big(\varphi
(\omega\cdot i)\big)\Big)\\
&=&\Psi\big(
\varphi^{-1}(u_i)\big)\\&=&
f(u_i),
\end{eqnarray*}
that is~\eqref{0oijhb8uj-0987654-pKCjHmCKOS}.

Besides, if~$\omega:=e_2$
the function~$u$ defined in~\eqref{0DEDFDYB8im10laAmmdLLA}
satisfies~\eqref{Agfoi-doiut} with
\begin{equation}\label{apq4c35v734756573}
c^\pm:=0,\end{equation} since
\begin{eqnarray*}&&
\frac{u_{i\pm he_1}-u_i}{u_{i\pm he_2}-u_i}
=\frac{\varphi(i_2)-\varphi(i_2)}{\varphi(i_2\pm h)-\varphi(i_2)}=0.
\end{eqnarray*}
Also, $u$ satisfies~\eqref{k8i-know}
with~$\widetilde{u}_{j}:=\varphi(j)$ for all~$j\in h\Z$, since,
by~\eqref{apq4c35v734756573},
$$
u_{(hk,hm)}=\varphi(hm)=\widetilde{u}_{hm}.
$$

Similarly, if~$\omega:=e_1+e_2$
the function~$u$ defined in~\eqref{0DEDFDYB8im10laAmmdLLA}
satisfies~\eqref{Agfoi-doiut} with
\begin{equation}\label{as2qc34yvb4676tgretg}
c^\pm:=1,\end{equation} since
\begin{eqnarray*}&&
\frac{u_{i\pm he_1}-u_i}{u_{i\pm he_2}-u_i}
=\frac{\varphi(i_1+i_2\pm h)-\varphi(i_1+i_2)}{
\varphi(i_1+i_2\pm h)-\varphi(i_1+i_2)}=1.
\end{eqnarray*}
Moreover, $u$ satisfies~\eqref{k8i-know}
with~$\widetilde{u}_{j}:=\varphi(j)$ for all~$j\in h\Z$, since,
by~\eqref{as2qc34yvb4676tgretg},
$$ u_{(hk,hm)}=\varphi(h(k+m))=\widetilde{u}_{h(m+k)}=
\widetilde{u}_{h(m+\sigma_k |k|)}.
$$
\medskip

We also observe that one-dimensional solutions of
classical semilinear ordinary differential equations
of the form~$U''=F(U)$, with~$U$ bounded and with bounded
derivatives
(such as the ones arising in the stationary Sine-Gordon equation
when~$F(U):=\sin U$ and~$U(t)=4 \arctan ( e^t)$,
in the Allen-Cahn equation when~$F(U):=U^3-U$ and for instance~$U(t)=\tanh (t / \sqrt{2})$, in the pendulum equation when~$F(U)=-\sin U$ and for instance~$U(t)$ is defined implicitly
by~$\int_0^{U(t)}\frac{d\tau}{\sqrt{ 2 \cos\tau}} = t $),
naturally induce one-dimensional solution of the discrete equation in~\eqref{DGEQ}. Indeed, given~$U$ as above,
one can consider~$u_i:=U(i_2)$ for every~$i\in h\Z^2$ and then
\begin{eqnarray*}
{\mathcal{L}} u_i&=&
\frac{ U(i_2+h) + U(i_2-h) -2U(i_2)}{h^2}\\
&=& \int_{0}^1\left[
\int_{-\tau}^\tau
U''(i_2+h\sigma)\,d\sigma
\right]\,d\tau
\\&=&\int_{0}^1\left[
\int_{-\tau}^\tau
F(U(i_2+h\sigma))\,d\sigma
\right]\,d\tau\\&=&F(U(i_2))+
\int_{0}^1\left[
\int_{-\tau}^\tau
\big(F(U(i_2+h\sigma))-F(U(i_2))\big)\,d\sigma
\right]\,d\tau\\&=&f(i,u_i)
,\end{eqnarray*}
where
\begin{eqnarray*}
f(i,r)&:=&
F(r)+
\int_{0}^1\left[
\int_{-\tau}^\tau
\big(F(U(i_2+h\sigma))-F(U(i_2))\big)\,d\sigma
\right]\,d\tau
\\&=&
F(r)+h
\int_{0}^1\left[
\int_{-\tau}^\tau\left(\int_0^\sigma
F'(U(i_2+h\mu))\,U'(i_2+h\mu)\,d\mu
\right)\,d\sigma
\right]\,d\tau.
\end{eqnarray*}
In this setting, assumption~\eqref{hp:semilinearitaconspazio}
is satisfied\footnote{An alternative way to compute $\kappa_0^+$ for functions of the type $u_i := U(i_2)$ is provided in full details at the beginning of Example \ref{example 4: general semilinear}.} by taking~$\kappa_0^+$ proportional
to the~$C^2$-norm of~$F$ in the range 
of~$U$.

Furthermore, notice that~$u_i$ satisfies~\eqref{k8i-know}
and~\eqref{Agfoi-doiut} with~$c^\pm:=0$.

\section{Examples}\label{sec:examples}
The examples presented in this section show that, in our setting, an
exact symmetry result 
analogous to that in the continuous case cannot hold true.

The examples also show that the rate of convergence of the
estimate \eqref{FORM} in the formal limit as~$h \searrow 0$ is optimal, in the sense that, in these cases, right-hand side and left-hand side of~\eqref{FORM}
are of the same order of $h$.

\begin{example}\label{example 1}
{\rm
For any 
\begin{equation}\label{eq:h tra 0 e 1}
0<h<1
\end{equation}
and $i=(i_1,i_2) \in h\Z^2$, we consider
\begin{equation*}
u_i:=
\begin{cases}
h^4 \quad & \text{for } \, i = (i_1,0) , \, i_1 > 0,
\\
i_2 \quad \, & \text{otherwise in }  h\Z^2 ,
\end{cases}
\end{equation*}
and we set 
\begin{equation}\label{eq:ex1:f def serve per ex4}
f(i, u_i):= \tilde{f}(i):=  
\begin{cases}
h^2 \quad \, & \text{for } \, i=(0,0) \quad \text{and}\quad
 i = (i_1, \pm h) \, \text{ with } \, i_1>0  ,
\\
-3 \, h^2 \quad \, & \text{for } \, i = ( h,0) ,
\\
-2 \, h^2 \quad \, & \text{for } \, i = ( i_1,0) \, \text{ with } \, i_1>h   ,
\\
0 \quad \, & \text{otherwise in } \, h\Z^2 .
\end{cases}
\end{equation}
By inspection, one sees that~${\mathcal{L}}u_i=\tilde f(i)$,
and therefore~\eqref{DGEQ} is satisfied.
It can be easily checked that $f$ satisfies \eqref{hp:GENERALsemilinearitaconspazio} with $L_f^+ = 0 $ and \begin{equation}\label{eq:ex1:kappa0}
\kappa_0^+=5 .
\end{equation}
In particular, $f$ satisfies \eqref{hp:regularitysemilinearity} and \eqref{hp:semilinearitaconspazio} (with $\kappa_0^+=5$).

Also, being $0<h<1$, the function $u$ clearly satisfies~\eqref{MONO}, and hence \eqref{hp:graddiversodazero} holds true.

We now show that $u$ also
satisfies~\eqref{BOULI}, \eqref{K2}, and
\eqref{K33} with $\vartheta^+_\infty := \pi / 2$.
To this aim, we directly compute
$$
{\mathcal{D}}_1^+ u_i =
\begin{cases}
h^3 \quad \, & \text{for } \, i = ( 0 , 0 ) , 
\\
0 \quad \, & \text{otherwise in } \, h\Z^2,
\end{cases}
\qquad
\text{ and }
\qquad
{\mathcal{D}}_2^+ u_i =
\begin{cases}
1-h^3 \quad \, & \text{for } \, i = (i_1,0) \text{ with } i_1>0 ,
\\
h^3 +1 \quad \, & \text{for } \, i = (i_1,-h) \text{ with } i_1>0,
\\
1 \quad \, & \text{otherwise in } \, h\Z^2 ,
\end{cases}
$$
and hence
\begin{eqnarray*}
\vartheta^+_i&=& 
\begin{cases}
\arctan\left( \frac{1}{h^3} \right) \quad \, & \text{for } \, i = (0,0),
\\
\frac{ \pi}{2} \quad & \text{otherwise in } \, h\Z^2 ,
\end{cases}\\
\text{and }\qquad
(\rho_i^+)^2 &=& 
\begin{cases}
1+h^6 \quad \, & \text{for } \, i = (0,0),
\\
( 1 - h^3 )^2 \quad \, & \text{for } \, i = (i_1,0) \, \text{ with } \, i_1>0 ,
\\
( 1 + h^3 )^2 \quad \, & \text{for } \, i = (i_1,-h) \, \text{ with } \, i_1>0 ,
\\
1 \quad & \text{otherwise in } \, h\Z^2 .
\end{cases}
\end{eqnarray*}
Thus, we have that
\begin{eqnarray*}
\kappa_1^+ :=  \sup_{{1\le j\le 2}\atop{i\in h\Z^2}} | {\mathcal{D}}_j^+ u_i | = 1 + h^3 < 2,
\end{eqnarray*}
being~$h < 1$, and this establishes~\eqref{BOULI}.
We also notice that
\begin{equation*}
\rho_i^+ \le (1+h^3)<2.
\end{equation*} 

We then compute
$$
{\mathcal{D}}_1^+ \vartheta^+_i =
\begin{cases}
\frac{1}{h} \left( \frac{ \pi}{2} - \arctan\left( \frac{1}{h^3} \right) \right) \quad \, & \text{for } \, i = ( 0 , 0 ) , 
\\
\frac{1}{h} \left( \arctan\left( \frac{1}{h^3} \right) -\frac{\pi}{2} 
\right) \quad \, & \text{for } \, i = ( -h , 0),
\\
0 \quad \, & \text{otherwise in } \, h\Z^2,
\end{cases}
$$
$$
{\mathcal{D}}_2^+ \vartheta^+_i =
\begin{cases}
\frac{1}{h} \left( \frac{ \pi}{2} - \arctan\left( \frac{1}{h^3} \right) \right) \quad \, & \text{for } \, i = ( 0 , 0 ) , 
\\
\frac{1}{h} \left( \arctan\left( \frac{1}{h^3} \right) -\frac{\pi}{2} \right) \quad \, & \text{for } \, i = ( 0, -h)
,\\
0 \quad \, & \text{otherwise in } \, h\Z^2 ,
\end{cases}
$$
$$
{\mathcal{D}}_1^+ \left( {\mathcal{D}}_1^+ \vartheta^+ \right)_{i-h e_j} =
\begin{cases}
\frac{2}{h^2} \left( \frac{ \pi}{2} - \arctan\left( \frac{1}{h^3} \right) \right) \quad \, & \text{for } \, i = ( 0 , 0 ) , 
\\
\frac{1}{h^2} \left( \arctan\left( \frac{1}{h^3} \right) -\frac{\pi}{2} \right) \quad \, & \text{for } \, i = ( \pm h , 0)
,\\
0 \quad \, & \text{otherwise in } \, h\Z^2,
\end{cases}
$$
and
$$
{\mathcal{D}}_2^+ \left( {\mathcal{D}}_2^+ \vartheta^+ \right)_{i-h e_j} =
\begin{cases}
\frac{2}{h^2} \left( \frac{ \pi}{2} - \arctan\left( \frac{1}{h^3} \right) \right) \quad \, & \text{for } \, i = (0, 0) ,
\\
\frac{1}{h^2} \left( \arctan\left( \frac{1}{h^3} \right) -\frac{\pi}{2} \right) \quad \, & \text{for } \, i = (0, \pm h) ,
\\
0 \quad \, & \text{otherwise in } \, h\Z^2 .
\end{cases}
$$
By noting that for any function $v: h\Z^2 \to \R$ it holds that
\begin{equation}\label{eq:ex1:relations che servono anche in ex4}
{\mathcal{D}}_j^- v_i=  {\mathcal{D}}_j^+ v_{i-h e_j} \quad \text{ and } \quad {\mathcal{D}}_j^- ({\mathcal{D}}_j^- v)_{i+he_j} = {\mathcal{D}}_j^+ ({\mathcal{D}}_j^+ v)_{i-he_j} \, ,
\quad \text{ where } \, i \in h \Z^2 , \, j=1,2,
\end{equation}
we can now directly compute
\begin{equation}\begin{split}\label{poqiuytrsxscvhiyutwer63786}
\kappa_2^+ =\,& \sum_{{1\le j\le 2}\atop{i\in h\Z^2}}
(\rho_i^+)^2 \,
\Big(
|{\mathcal{D}}_j^+\vartheta^+_i|\,|{\mathcal{D}}_j^+ ({\mathcal{D}}_j^+ \vartheta^+)_{i-he_j}|+
|{\mathcal{D}}_j^-\vartheta^+_i|\,|{\mathcal{D}}_j^- ({\mathcal{D}}_j^- \vartheta^+)_{i+he_j}|
\Big)\\
=\,&  (12 + 9 h^6 - 2 h^3) \frac{ \left( \frac{\pi}{2} - \arctan\left( \frac{1}{h^3} \right) \right)^2  }{h^3} .
\end{split}\end{equation}
Now we recall the inequality
\begin{equation}\label{eq:arctanestimatePRE}
\left| \mathrm{sgn}(t) \frac{\pi}{2} - \arctan(t) \right| \le \frac{1}{|t|}  \, \text{ for every } t \in \R , 
\end{equation}
Accordingly, using~\eqref{eq:arctanestimatePRE} with~$t=1/ h^3$,
\begin{equation}\label{eq:arctanestimate}
\left| \mathrm{sgn}\left(\frac1{h^3}\right) 
\frac{\pi}{2} - \arctan\left(\frac1{h^3}\right)  \right| \le h^3.
\end{equation}
{F}rom this and~\eqref{poqiuytrsxscvhiyutwer63786},
and recalling~\eqref{eq:h tra 0 e 1},
we obtain that
\begin{equation}\label{eq:ex1:k2}
\kappa_2^+ \le 21 \, h^3 ,
\end{equation}
that immediately gives~\eqref{K2} and also keeps track of the order of $h$.

In order to verify \eqref{K33}, we also compute
$$
\mathcal{L}_1^2 \vartheta_i^+ =
\begin{cases}
\frac{6}{h^4} \left( \arctan\left( \frac{1}{h^3} \right) - \frac{ \pi}{2}  \right) \quad \, & \text{for } \, i = (0, 0) ,
\\
\frac{4}{h^4} \left( \frac{\pi}{2} - \arctan\left( \frac{1}{h^3} \right)  \right) \quad \, & \text{for } \, i = (\pm h , 0 ) ,
\\
\frac{1}{h^4} \left( \arctan\left( \frac{1}{h^3} \right) - \frac{ \pi}{2}  \right) \quad \, & \text{for } \, i = ( \pm 2  h , 0 ) ,
\\
0  \quad \, & \text{otherwise in } \, h\Z^2 ,
\end{cases}
$$
and
$$
\mathcal{L}_2^2 \vartheta_i^+ =
\begin{cases}
\frac{6}{h^4} \left( \arctan\left( \frac{1}{h^3} \right) - \frac{ \pi}{2}  \right) \quad \, & \text{for } \, i = (0, 0) ,
\\
\frac{4}{h^4} \left( \frac{\pi}{2} - \arctan\left( \frac{1}{h^3} \right)  \right) \quad \, & \text{for } \, i = ( 0 , \pm h  ) ,
\\
\frac{1}{h^4} \left( \arctan\left( \frac{1}{h^3} \right) - \frac{ \pi}{2}  \right) \quad \, & \text{for } \, i = ( 0, \pm 2 h ) ,
\\
0  \quad \, & \text{otherwise in } \, h\Z^2 .
\end{cases}
$$
Recalling that~$\vartheta_\infty^+= \pi/{2}$, we have that
\begin{equation*}
|\vartheta_i^+ - \vartheta_\infty^+|=
\begin{cases}
  \frac{ \pi}{2} - \arctan\left( \frac{1}{h^3} \right)  \quad \, & \text{for } \, i = ( 0 , 0 ) ,
\\
0  \quad \, & \text{otherwise in } \, h\Z^2 .
\end{cases}
\end{equation*}
Thus, the only nonzero term in the summations
defining $\kappa_3^+$, $\kappa_4^+$, $\kappa_5^+$, $\kappa_6^+$,
$\kappa_7^+$ in \eqref{K33} are those for~$i=(0,0)$. To explicitly obtain $\kappa_4^+$, $\kappa_6^+$, $\kappa_7^+$ we will also need to compute
\begin{eqnarray*}
&&| {\mathcal{D}}_1^+\rho_{(0,0)}^+ | = \frac{ 
| (1-h^3)- \sqrt{1+h^6} | }{h}, \qquad 
| {\mathcal{D}}_2^+\rho_{(0,0)}^+ | =
| {\mathcal{D}}_1^- \rho_{(0,0)}^+ | = 
| {\mathcal{D}}_2^- \rho_{(0,0)}^+ | =
 \frac{ \sqrt{1+h^6} -1 }{h} ,
\\
&&
|{\mathcal{D}}_1^{+} ( \rho_{(0,0)}^+  )^2 | =
2 h^2 ,\qquad 
|{\mathcal{D}}_2^{+} ( \rho_{(0,0)}^+  )^2 | =
|{\mathcal{D}}_1^{-} ( \rho_{(0,0)}^+  )^2 | =
|{\mathcal{D}}_2^{-} ( \rho_{(0,0)}^+  )^2 | =
h^5 .
\end{eqnarray*}

Thus, we find that
\begin{equation*}
\kappa_3^+ = 4 (1 + h^6 )\frac{\left( \frac{\pi}{2} - \arctan\left( \frac{1}{h^3} \right)  \right)^4}{h^3} ,
\end{equation*}
which, in light of \eqref{eq:h tra 0 e 1} and \eqref{eq:arctanestimate},
gives that
\begin{equation}\label{eq:ex1:k3}
\kappa_3^+ \le 8 h^9 .
\end{equation}
Furthermore, we have that
\begin{equation}\label{alertenvxsatru-008080}
\kappa_4^+ = \sqrt{ 1 + h^6 } \, \frac{ \left( \frac{\pi}{2} -
 \arctan\left( \frac{1}{h^3} \right)  \right)^3}{h^2} 
\left[ h^2 + 4 \, \frac{\sqrt{1+h^6} - 1}{h}  \right].
\end{equation}
We also recall the inequality
\begin{equation*}
| 1 - \sqrt{1-t} | \le |t|  \, \text{ for every }  -1 \le t \le 1,
\end{equation*}
which, taking~$t = - h^6$, gives that
\begin{equation}\label{eq:squarerootestimate}
| \sqrt{1 + h^6} - 1 | \le | h |^6  .
\end{equation}
{F}rom this, \eqref{eq:h tra 0 e 1}, \eqref{eq:arctanestimate} and~\eqref{alertenvxsatru-008080},
we obtain that
\begin{equation}\label{eq:ex1:k4}
\kappa_4^+ \le 5 \sqrt{2} \, h^9 .
\end{equation}
By recalling \eqref{eq:ex1:kappa0}, we also immediately find that
\begin{equation*}
\kappa_5^+ = 5 \, \sqrt{ 2 }  \left( \frac{\pi}{2} - \arctan\left( \frac{1}{h^3} \right)  \right)  ,
\end{equation*}
and hence, by \eqref{eq:h tra 0 e 1} and \eqref{eq:arctanestimate},
\begin{equation}\label{eq:ex1:k5}
\kappa_5^+ \le 5 \sqrt{2} \, h^3 .
\end{equation}
Furthermore, we compute
\begin{equation*}
\kappa_6^+ =  \left[ 2 h^2 +3 h^5 + \frac{12(1+h^6)}{h^3} \right]  \left( \frac{\pi}{2} - \arctan\left( \frac{1}{h^3} \right)  \right)^2 = \left[ 12+ 2 h^5 +12 h^6 +3 h^8 \right] \frac{\left( \frac{\pi}{2} - \arctan\left( \frac{1}{h^3} \right)  \right)^2}{h^3} ,
\end{equation*}
and hence, by \eqref{eq:h tra 0 e 1} and \eqref{eq:arctanestimate},
\begin{equation}\label{eq:ex1:k6}
\kappa_6^+ \le 29 h^3 .
\end{equation}
Finally, we have that
\begin{equation*}
\kappa_7^+ =  \left[ h^4 + 2 h( \sqrt{1+h^6}-1 ) + 2 \, \left( \frac{\sqrt{1+h^6} - 1}{h} \right)^2 + 2 \, \left( \frac{ 1 - \sqrt{1 - h^6}}{h} \right)^2 \right]  \frac{ \left( \frac{\pi}{2} - \arctan\left( \frac{1}{h^3} \right)  \right)^2}{h} ,
\end{equation*}
and hence, by \eqref{eq:h tra 0 e 1}, 
\eqref{eq:arctanestimate} and~\eqref{eq:squarerootestimate},
\begin{equation}\label{eq:ex1:k7}
\kappa_7^+ \le \left[ h^4 + 2h^7 +2h^{10} +2 h^{10} \right]  h^5 \le 7 h^9 .
\end{equation}

In light of \eqref{eq:h tra 0 e 1}, inequalities \eqref{eq:ex1:k3}, \eqref{eq:ex1:k4}, \eqref{eq:ex1:k5}, \eqref{eq:ex1:k6}, \eqref{eq:ex1:k7} clearly give \eqref{K33}.

All in all, we have that~$u$ and $f$ satisfy all the assumptions of
Theorem \ref{DGDG} and hence \eqref{FORM} holds true.
Nevertheless~$u$ is not one-dimensional, and 
\begin{equation}\label{eq:esempi div da zero}
\sum\limits_{{1\le j\le 2}\atop{i\in h\Z^2}}
(\rho_i^+)^2 \;\Big(
|{\mathcal{D}}_j^+\vartheta^+_i|^2+|{\mathcal{D}}_j^-\vartheta^+_i|^2
\Big) \neq 0.
\end{equation}
We stress that the quantity in the left-hand side of~\eqref{eq:esempi div da zero}
is precisely the one appearing in~\eqref{FORM}, hence the fact that it
is nonzero says that Theorem~\ref{DGDG} cannot be improved
in general by obtaining that such a quantity vanishes.
We also notice that the quantity in the left-hand side of~\eqref{eq:esempi div da zero} can be explicitly computed. Here, we just notice that
\begin{eqnarray*}
 \sum\limits_{{1\le j\le 2}\atop{i\in h\Z^2}}
(\rho_i^+)^2 \;\Big(
|{\mathcal{D}}_j^+\vartheta^+_i|^2+|{\mathcal{D}}_j^-\vartheta^+_i|^2
\Big) 
& \ge &
 (\rho_{(-h,0)}^+)^2 \; \sum\limits_{j=1}^2 \Big(
|{\mathcal{D}}_{j}^+\vartheta^+_{(-h,0)}|^2+|{\mathcal{D}}_j^-\vartheta^+_{(-h,0)}|^2
\Big)
\\
& = & | {\mathcal{D}}_1^+\vartheta^+_{(-h,0)}|^2 
 =  \frac{ \left( \frac{\pi}{2} - \arctan\left( \frac{1}{h^3} \right) \right)^2  }{h^2}
.\end{eqnarray*}
Moreover, the following inequality holds true:
\begin{equation}\label{eq:ex1:inequality serve per ex4 arctan da sotto}
\left| \mathrm{sgn}(t) \frac{\pi}{2} - \arctan(t) \right| 
\ge \frac{4}{\pi \, |t|}  \quad \text{ for every } t \ge 1  ,
\end{equation}
which gives, taking~$t = 1/h^3$,
\begin{equation}\label{eq:nuovastimaarctandasotto}
\left| \mathrm{sgn}\left(\frac1{h^3}\right) \frac{\pi}{2} - \arctan
\left(\frac1{h^3}\right) \right| 
\ge \frac{4 \left| h \right|^3}{\pi}.
\end{equation}
{F}rom this and recalling~\eqref{eq:h tra 0 e 1},
we thus get that
\begin{equation}\label{eq:ex1:stima da sotto}
\sum\limits_{{1\le j\le 2}\atop{i\in h\Z^2}}
(\rho_i^+)^2 \;\Big(
|{\mathcal{D}}_j^+\vartheta^+_i|^2+|{\mathcal{D}}_j^-\vartheta^+_i|^2
\Big) 
 \ge
 \frac{4^2}{\pi^2} \, h^4.
\end{equation}

On the other hand, \eqref{eq:h tra 0 e 1}, \eqref{eq:ex1:k2}, \eqref{eq:ex1:k3}, \eqref{eq:ex1:k4}, \eqref{eq:ex1:k5}, \eqref{eq:ex1:k6}, \eqref{eq:ex1:k7}, inform us that the right-hand side of \eqref{FORM} satisfies
\begin{equation}\label{eq:ex:stimadasopra}
C \, h \le 4 \left[ 57 + 10 \sqrt{2} + 2 \, e^{2 \pi} \left( 8 + 5 \sqrt{2} \right) \right] h^4 ,
\end{equation}
where $C$ is the quantity defined in \eqref{eq:constant C}. 
Thus, by putting together \eqref{FORM}, \eqref{eq:ex1:stima da sotto},
and \eqref{eq:ex:stimadasopra}, it is clear that, in the formal limit as~$
h \searrow 0$, left-hand side and right-hand side of \eqref{FORM}
are both of the order of $h^4$. In this sense, \eqref{FORM} is optimal.
}
\end{example}

We provide other three examples confirming the optimality of \eqref{FORM}.

In particular, in the next two examples left-hand side and right-hand side of \eqref{FORM} are both of the order of $h^2$.
The next two examples also show that an exact symmetry result cannot hold true in the discrete case, even if we restrict our analysis to the case of source terms of the form $f(i,u_i)=\hat{f}(u_i)$ (i.e., only depending on $u$),
satisfying \eqref{hp:GENERALsemilinearitaconspazio}.

Since the functions involved in the next examples are restrictions
of smooth functions defined in the whole of~$\R^2$,
the following fact will be useful.
If $v:h\Z^2 \to \R$ is the restriction on $h\Z^2$ of a smooth function $\overline{v}: \R^2 \to \R$, then, for any $i \in h \Z^2$ and $j=1,2$, we have that
\begin{equation}\label{eq:Lagrangeperestensione}
\mathcal{D}_j^\pm v_i = \partial_{x_j} \overline{v}(i \pm \xi e_j)
\end{equation}
for some~$\xi \in (0,h)$,
\begin{equation}\label{eq:TaylorLagrangelaplacianodirezionale}
\mathcal{D}_j^\pm(\mathcal{D}_j^\pm v)_{i- h e_j} =
 \mathcal{L}_j v_i = \frac{ \partial_{x_j x_j} \overline{v}(i
 \pm \xi_1 e_j) + \partial_{x_j x_j} \overline{v}(i \pm \xi_2 e_j) }{2}
,
\end{equation}
for some~$\xi_1 \in (0, h)$ and~$\xi_2\in (-h, 0)$,
\begin{equation}\label{eq:TaylorLagrangesecondederivate}
\mathcal{D}_j^\pm(\mathcal{D}_j^\pm v)_i = 2 \
\partial_{x_j x_j} \overline{v}(i \pm \xi_1 e_j) -
 \partial_{x_j x_j} \overline{v}(i \pm \xi_2 e_j) ,
\end{equation}
for some~$\xi_1 \in (0, 2h)$ and~$\xi_2 \in (0,h)$, and
\begin{equation}\label{eq:TaylorLagrangeLaplacianoallaseconda}
\mathcal{L}_j^2 v_i =
\frac{ 4 \, \partial_{x_j x_j x_j x_j} \overline{v}(i + \xi_1 e_j) - \partial_{x_j x_j x_j x_j} \overline{v}(i + \xi_2 e_j)  - \partial_{x_j x_j x_j x_j} \overline{v}(i + \xi_3 e_j) + 4 \,
 \partial_{x_j x_j x_j x_j} \overline{v}(i + \xi_4 e_j) }{6}  ,
\end{equation}
for some~$\xi_1 \in (0, 2h)$, $\xi_2 \in (0,h)$, $\xi_3 \in (-h,0)$,
and~$\xi_4 \in (-2h,0)$.

Identity \eqref{eq:Lagrangeperestensione} directly follows from Lagrange theorem. Identities \eqref{eq:TaylorLagrangelaplacianodirezionale}, \eqref{eq:TaylorLagrangesecondederivate} and \eqref{eq:TaylorLagrangeLaplacianoallaseconda} can be easily obtained by using Taylor expansions with Lagrange reminder terms.
From \eqref{eq:Lagrangeperestensione}, \eqref{eq:TaylorLagrangelaplacianodirezionale}, \eqref{eq:TaylorLagrangesecondederivate}, and \eqref{eq:TaylorLagrangeLaplacianoallaseconda}, we easily deduce that
\begin{equation}\label{eq:TAYLORBOUNDSUP FIRSTDERIVATIVE}
| \mathcal{D}_j^\pm v_i | \le \sup_{x \in \R^2} | \partial_{x_j} \overline{v}(x) |  \, , \quad \text{ for any } \, i \in h\Z^2, \, j=1,2,
\end{equation}
\begin{equation}\label{eq:TAYLORBOUNDSUP DIRECTIONALLAPLACIAN}
\left| \mathcal{D}_j^\pm(\mathcal{D}_j^\pm v)_{i- h e_j} \right| = \left| \mathcal{L}_j v_i \right| \le \sup_{x \in \R^2} | \partial_{x_j x_j} \overline{v}(x) | \, , \quad \text{ for any } \, i \in h\Z^2, \, j=1,2,
\end{equation}
\begin{equation}\label{eq:TAYLORBOUNDSUP SECONDDERIVATIVES}
| \mathcal{D}_j^\pm(\mathcal{D}_j^\pm v)_i | \le 3 \, \sup_{x \in \R^2} |\partial_{x_j x_j} \overline{v}(x) | \, , \quad \text{ for any }
 \, i \in h\Z^2, \, j=1,2,
\end{equation}
and
\begin{equation}\label{eq:TAYLORBOUNDSUP DIRECTIONALLAPLACIANSQUARED}
\left| \mathcal{L}_j^2 v_i \right| \le
\frac{ 5 }{3} \, \sup_{x \in \R^2} |\partial_{x_j x_j x_j x_j} \overline{v}(x)| \, , \quad \text{ for any } \, i \in h\Z^2, \, j=1,2.
\end{equation}

\begin{example}\label{example ARCTAN}
{\rm
For any 
\begin{equation}\label{eq:ex con arctan:hsmallness}
0< h \le 1 ,
\end{equation}
and $i=(i_1,i_2) \in h\Z^2$, we consider $u: h\Z^2 \to \R$ defined by
\begin{equation}\label{arqyrhtjuojih587462q4e5rfwt}
u_i := i_2 + \frac{h}{\pi} \, e^{-i_2^2} \arctan(i_1)
\end{equation}
and we set
\begin{eqnarray*}
\tilde{f}(i) &:=& \frac{h}{ \pi } \Bigg\lbrace 
e^{-i_2^2} \left[ \frac{\arctan(i_1+h)+ \arctan(i_1-h) 
-2 \arctan(i_1) }{h^2} \right] \\&&\qquad+ \arctan(i_1) \left[ 
\frac{ e^{-(i_2+h)^2} + e^{-(i_2-h)^2} -
2 e^{- i_2^2} }{h^2} \right] \Bigg\rbrace .
\end{eqnarray*}
In this setting,
\begin{equation}\label{eq:ex arctan:equazione}
{\mathcal{L}}u_i=\tilde f(i) ,
\end{equation}
and hence \eqref{DGEQ} holds true with 
$$
f(i, u_i) := \tilde{f} (i).
$$

At the end of this example, we will also show that there exists a function $\hat{f} :\R \to \R$ such that
\begin{equation}\label{eq:exarctan:FINALEALLAFINE}
\tilde{f}(i)= \hat{f}(u_i)  \quad \text{ for any } i \in h\Z^2 ,
\end{equation}
and hence that $u$ is solution of the equation
\begin{equation*}
{\mathcal{L}} u_i = \hat{f}(u_i) , 
\end{equation*}
where the source term $\hat{f}$ only depends on $u$.

\smallskip

We now show that $f$ satisfies \eqref{hp:GENERALsemilinearitaconspazio} with $L_f^+ = 0$.
For this, we notice that,
for any $i \in h\Z^2$, by Lagrange Theorem, there exist $\xi_1$, $\xi_2 \in (0,h)$ such that
\begin{equation}\label{eq:exarctan:calcolokappa0 1}
\sum_{j=1}^2 \frac{ \left| f(i+h e_j, u_{i+h e_j})-f(i, u_{i}) \right| }{h}
= \sum_{j=1}^2 \frac{ \left| \tilde{f}(i+h e_j)- \tilde{f}(i, ) \right| }{h} = \sum_{j=1}^2 \left| \partial_{x_j}\overline{f}(i+\xi_j e_j) \right| ,
\end{equation}
where $\overline{f}$ denotes the function obtained by extending the definition of $\tilde{f}$
to the whole of~$\R^2$, that is,
\begin{equation*}
\begin{split}
\overline{f}(x) := & \frac{h}{ \pi } \left\{ e^{-x_2^2} \left[ \frac{\arctan(x_1+h)+ \arctan(x_1-h) -2 \arctan(x_1) }{h^2} \right] 
\right.
\\
& \left. 
+ \arctan(x_1) \left[ \frac{ e^{-(x_2+h)^2} + e^{-(x_2-h)^2} -2 e^{- x_2^2} }{h^2} \right] \right\} ,
\end{split}
\end{equation*}
for any $x=(x_1,x_2) \in \R^2$.
Then we directly compute
\begin{equation}\label{eq:exarctan:partial x_1 fsemilin}
\partial_{x_1} \overline{f} (x) =  \frac{h}{ \pi } \left\{ e^{-x_2^2} \left[ \frac{ \frac{1}{ 1+(x_1+h)^2}+ \frac{1}{1+(x_1-h)^2} -2 \frac{1}{1+ x_1^2} }{h^2} \right] 
+ \frac{1}{1+ x_1^2} \left[ \frac{ e^{-(x_2+h)^2} + e^{-(x_2-h)^2}
 -2 e^{- x_2^2} }{h^2} \right] \right\}
\end{equation}
and
\begin{equation}\label{eq:exarctan:partial x_2 f semilin}
\begin{split}
\partial_{x_2} \overline{f}(x) = & \frac{h}{ \pi } \left\{ -2 x_2 e^{-x_2^2} \left[ \frac{\arctan(x_1+h)+ \arctan(x_1-h) -2 \arctan(x_1) }{h^2} \right]
\right.
\\
& \left.  
+ \arctan(x_1) \left[ \frac{ -2 (x_2 +h) e^{-(x_2+h)^2} - 2 (x_2 - h ) e^{-(x_2-h)^2} -2 (-2 x_2) e^{- x_2^2} }{h^2} \right] \right\} .
\end{split}
\end{equation}
Now we recall that for a smooth real function $\overline{v}: \R^2
\to \R$, a Taylor expansion with second order Lagrange reminder term
gives that, fixed $j=1,2$, there exist $\eta_1 \in  (0,h)$ and $\eta_2 \in (-h , 0)$ such that
\begin{equation*}
\frac{\overline{v}(x+h e_j) + \overline{v}(x - h e_j) -2 \overline{v}(x)}{h^2}= \frac{1}{2} \left\lbrace \partial_{x_j x_j} \overline{v} (x+\eta_1 e_j) + \partial_{x_j x_j} \overline{v} (x+\eta_2 e_j) \right\rbrace ,
\end{equation*}
and hence
$$
\left| \frac{\overline{v}(x+h e_j) + \overline{v}(x - h e_j) -2 \overline{v}(x)}{h^2} \right| \le |\partial_{x_j x_j} \overline{v} (x+\eta e_j)| \quad \text{ for some } \, \eta\in (-h,h),
$$
from which in particular we have
\begin{equation}\label{eq:exarctan:Taylorpervreale}
\left| \frac{\overline{v}(x+h e_j) + \overline{v}(x - h e_j) -2 \overline{v}(x)}{h^2} \right| \le \sup_{x \in \R^2} |\partial_{x_j x_j} \overline{v} (x)|  , \quad j=1,2 .
\end{equation}
By setting $\overline{v}(x):= \frac{1}{1+x_1^2}$, we find that
$$
\partial_{x_1 x_1} \overline{v}(x) =2 \, \frac{3x_1^2 -1}{(1+x_1^2)^3} ,
$$
and noting that
$$
\left| \frac{3 t^2 -1}{(1+ t^2)^3} \right| \le 1 \quad \text{ for any } t \in \R ,
$$
by \eqref{eq:exarctan:Taylorpervreale} we get that
\begin{equation}\label{eq:exarctan:perfsemilin 1}
\left| \frac{ \frac{1}{ 1+(x_1+h)^2}+ \frac{1}{1+(x_1-h)^2} -2 \frac{1}{1+ x_1^2} }{h^2} \right| \le 2 .
\end{equation}
Similarly, if we set $\overline{v}(x):= e^{- x_2^2}$, we compute
$$
\partial_{x_2 x_2} \overline{v}(x) = 2 e^{-x_2^2} (2 x_2^2 - 1) ,
$$
and noting that
\begin{equation}\label{eq:exarctan:RICICLATA}
\left| e^{-t^2} (2 t^2 - 1) \right| \le 1 \quad \text{ for any } t \in \R ,
\end{equation}
formula~\eqref{eq:exarctan:Taylorpervreale} informs us that
\begin{equation}\label{eq:exarctan:perfsemilin 2}
\left| \frac{ e^{-(x_2+h)^2} + e^{-(x_2-h)^2} -2 e^{- x_2^2} }{h^2} \right| \le 2.
\end{equation}
By putting together \eqref{eq:exarctan:partial x_1 fsemilin},
\eqref{eq:exarctan:perfsemilin 1} and~\eqref{eq:exarctan:perfsemilin 2}, we thus obtain that
\begin{equation}\label{eq:exarctan:partialx_1 boundsemilin}
\left| \partial_{x_1} \overline{f}(x) \right| \le \frac{4}{\pi} \, h , \quad \text{ for any } \, x \in \R^2 .
\end{equation}

In order to obtain a similar estimate for $\partial_{x_2} \overline{f}$, we now set $\overline{v}(x):= \arctan(x_1)$
and we compute
$$
\partial_{x_1 x_1} \overline{v}(x) = - \frac{2 x_1}{(1+x_1^2)^2} ,
$$
and noting that
$$
\left| \frac{t}{(1+ t^2)^2} \right| \le \frac{ \sqrt{27} }{16} \left( < \frac{1}{2} \right) \quad \text{ for any } t \in \R ,
$$
by \eqref{eq:exarctan:Taylorpervreale} we get that
\begin{equation}\label{eq:exarctan:perfsemilin 3}
\left|   \frac{\arctan(x_1+h)+ \arctan(x_1-h) -2 \arctan(x_1) }{h^2} \right| < 1 .
\end{equation}
Similarly, if we set 
$\overline{v}(x):= -2 x_2 e^{-x_2^2}$, we compute
$$
\partial_{x_2 x_2} \overline{v}(x) = 4 e^{-x_2^2} x_2 (3 - 2x_2^2) ,
$$
and noting that
$$
\left| e^{-t^2} t (3 - 2t^2) \right| <1 \quad \text{ for any } t \in \R ,
$$
by \eqref{eq:exarctan:Taylorpervreale} we get that
\begin{equation}\label{eq:exarctan:perfsemilin 4}
\left| \frac{ -2 (x_2 +h) e^{-(x_2+h)^2} - 2 (x_2 - h ) e^{-(x_2-h)^2} -2 (-2 x_2) e^{- x_2^2} }{h^2}  \right| < 4 .
\end{equation}
By putting together \eqref{eq:exarctan:partial x_2 f semilin},
\eqref{eq:exarctan:perfsemilin 3} and \eqref{eq:exarctan:perfsemilin 4},
and using
the inequality
\begin{equation}\label{eq:exarctan:disuguaglianzaper t e^(-t^2)}
\left| t e^{-t^2} \right| \le \frac{1}{\sqrt{2\, e}} \left( < \frac{1}{2} \right) \quad \text{ for any } \, t \in \R ,
\end{equation}
and the bound
\begin{equation}\label{eq:exarctan:boundsuarctan}
|\arctan(t)| \le \frac\pi{2} \quad \text{ for any } \, t \in \R ,
\end{equation}
we obtain that
\begin{equation*}
\left| \partial_{x_2} \overline{f}(x) \right| \le \left( \frac{1}{\pi} + 2 \right)  h , \quad \text{ for any } \, x \in \R^2 .
\end{equation*}
{F}rom this,
\eqref{eq:exarctan:calcolokappa0 1}
and~\eqref{eq:exarctan:partialx_1 boundsemilin}, we thus obtain
$$
\sum_{j=1}^2 \frac{ \left| f(i+h e_j, u_{i+h e_j})-f(i, u_{i}) \right| }{h} < \left( \frac{5}{\pi} + 2 \right)h ,
$$
that is, \eqref{hp:GENERALsemilinearitaconspazio} holds true with $L_f^+ = 0$ and 
\begin{equation}\label{EQ:EXARCTAN KAPPA_0}
\kappa_0^+ = \frac{5}{\pi} + 2 .
\end{equation} 

We notice that, being $h>0$, $u_i$ is increasing in $i_2$,
that is, \eqref{MONO} holds true. Indeed, 
by recalling \eqref{eq:exarctan:boundsuarctan}
and the fact that~$|e^{-(i_2+h)^2} - e^{-i_2^2}|
 \le \max \left\lbrace e^{-i_2^2} , e^{-(i_2+h)^2} \right\rbrace \le 1$,
we see that
\begin{equation}\label{eq:exarctan:USOQUESTAPERSCREMARE}
\left| \frac{\arctan(i_1)}{\pi} \left( e^{-(i_2+h)^2} - e^{-i_2^2} \right) \right| \le \frac{1}{2}.
\end{equation}
{F}rom this and~\eqref{arqyrhtjuojih587462q4e5rfwt}, one finds that
$$
u_{i+h e_2} - u_i= h \left[ 
1 + \frac{\arctan(i_1)}{\pi} \left( e^{-(i_2+h)^2} 
- e^{-i_2^2} \right) \right] \ge \frac{h}{2} > 0,
$$ 
which proves~\eqref{MONO}.
In particular, \eqref{hp:graddiversodazero} surely holds true.

By direct inspection, we also check that \eqref{BOULI}, \eqref{K2}, \eqref{K33} are all satisfied, and hence Theorem \ref{DGDG} applies.
To this end, we start by computing
$$
{\mathcal{D}}_1^+ u_i= \frac{h}{\pi} e^{-i_2^2} \left[ \frac{\arctan(i_1+h) - \arctan(i_1)}{h} \right]
\qquad{\mbox{and}}\qquad
{\mathcal{D}}_2^+ u_i= 1 + \frac{h}{\pi} \arctan(i_1)  \left[ \frac{ e^{-(i_2+h)^2} - e^{-i_2^2} }{h} \right] .
$$
By Lagrange Theorem, for any $t \in \R$ we have that
\begin{equation*}
\frac{\arctan(t+h) - \arctan(t)}{h}= \frac{1}{1+(t + \xi)^2} , \quad \text{ for some } \, \xi \in (0,h) ,
\end{equation*}
and hence
\begin{equation*}
{\mathcal{D}}_1^+ u_i =
 \frac{h}{\pi} \left[ \frac{e^{-i_2^2}}{1+ (i_1+\xi)^2} \right].
\end{equation*}
{F}rom this, we get
\begin{equation}\label{eq:exarctan:Boundderivata in i_1 con h}
0 < {\mathcal{D}}_1^+ u_i \le \frac{h}{\pi} \quad \text{ for any } \, i \in h\Z^2 ,
\end{equation}
and therefore, by \eqref{eq:ex con arctan:hsmallness}, 
\begin{equation}\label{eq:exarctan:bounds derivata in i_1}
0 < {\mathcal{D}}_1^+ u_i \le \frac{1}{\pi} \quad \text{ for any } \, i \in h\Z^2 .
\end{equation}
By \eqref{eq:exarctan:USOQUESTAPERSCREMARE} we also have that
\begin{equation}\label{eq:exarctan:bounds derivata in i_2}
\frac{1}{2} < {\mathcal{D}}_2^+ u_i < \frac{3}{2} \quad \text{ for any } \, i \in h\Z^2 .
\end{equation}
By putting together \eqref{eq:exarctan:bounds derivata in i_1}
and \eqref{eq:exarctan:bounds derivata in i_2} we
obtain that \eqref{BOULI} holds true with $\kappa_1^+ = \frac{3}{2}$.

Being $\mathcal{D}_1^+ $ and $\mathcal{D}_2^+$ positive, for any $i \in h\Z^2$ we have that
\begin{equation}\label{eq:exarctan:defvarthetainfty}
\vartheta_i^+= \arctan \left[ \frac{\mathcal{D}_2^+ u_i}{\mathcal{D}_1^+ u_i } \right] = 
\arctan \left[ \frac{1 + \frac{h}{\pi} \arctan(i_1)  \left( \frac{ e^{-(i_2+h)^2} - e^{-i_2^2}  }{h} \right) }{ \frac{h}{\pi} e^{-i_2^2} \left( \frac{\arctan(i_1+h) - \arctan(i_1)}{h} \right) } \right] .
\end{equation}
Let us now prove \eqref{K2}. Since $\mathcal{D}_1^+u$, $\mathcal{D}_2^+ u$, and $\vartheta^+$
can be seen as restrictions to $h\Z^2$ of
smooth functions of $\R^2$, it is convenient to define, for $x=(x_1,x_2)\in \R^2$:
\begin{equation}\label{eq:EXARCTAN: DEF OVERLINE U_1 U_2}
\begin{split}\overline{u}_1 (x) &:= \frac{h}{\pi} e^{-x_2^2}
 \left[ \frac{\arctan(x_1+h) - \arctan(x_1)}{h} \right] ,
\\
\overline{u}_2 (x)& := 1 + \frac{h}{\pi} \arctan(x_1)  \left[ \frac{ e^{-(x_2+h)^2} - e^{-x_2^2} }{h} \right] ,
\\{\mbox{and }}\quad
\overline{\vartheta}(x)& := \arctan \left( \frac{\overline{u}_2 (x)}{\overline{u}_1 (x)} \right) .
\end{split}\end{equation}
With these definitions, the restrictions of ${\overline{u}_1}, {\overline{u}_2}$, and $\overline{\vartheta}$ to $h\Z^2$ coincide with $\mathcal{D}_1^+u$, $\mathcal{D}_2^+ u$, and $\vartheta^+$.

It is easy to check that the estimates obtained in \eqref{eq:exarctan:Boundderivata in i_1 con h} and \eqref{eq:exarctan:bounds derivata in i_2} for $\mathcal{D}_1^+$ and $\mathcal{D}_2^+$ still hold true for $\overline{u}_1$ and $\overline{u}_2$, that is
\begin{equation}\label{eq:exarctan:OVERLINEU1}
0 < \overline{u}_1 (x) \le \frac{h}{\pi} \quad \text{ for any } \, x \in \R^2 ,
\end{equation}
and
\begin{equation}\label{eq:exarctan:OVERLINEU2}
\frac{1}{2} < \overline{u}_2 (x) < \frac{3}{2} \quad \text{ for any } \, x \in \R^2 .
\end{equation}

Let us now find a useful estimate for $\left| \mathcal{D}_j^\pm \vartheta_i^+ \right|$, for $i \in h\Z^2$ and $j = 1, 2$.
By Lagrange Theorem, for any given $i \in h \Z^2$, there exists $\xi \in (0,h)$ such that
\begin{equation}\label{eq:exarctan:LAGRANGEDJ}
\left| \mathcal{D}_j^\pm \vartheta_i^+ \right| = \left| \partial_{x_j}\overline{\vartheta}(i + \xi e_j) \right| .
\end{equation}
Let us now compute
\begin{equation}\label{EQ:LB EXARCTAN SERVE 1}
\partial_{x_j} \overline{\vartheta} = \frac{ ( \partial_{x_j}\overline{u}_2) \, \overline{u}_1 - \overline{u}_2 \, ( \partial_{x_j} \overline{u}_1 ) }{\overline{u}_1^2 + \overline{u}_2^2} .
\end{equation}
Since by \eqref{eq:exarctan:OVERLINEU1} and \eqref{eq:exarctan:OVERLINEU2}, it holds that
\begin{equation}\label{eq:exarctan:U1^2+U2^2}
\overline{u}_1^2+\overline{u}_2^2 \ge \frac{1}{4} ,
\end{equation}
we find that
\begin{equation}\label{eq:exarctan:STIMAPARTIAL1VARTHETA}
\left| \partial_{x_j} \overline{\vartheta} \right| \le 4 \left| ( \partial_{x_j}\overline{u}_2) \, \overline{u}_1 - \overline{u}_2 \, ( \partial_{x_j} \overline{u}_1 ) \right| .
\end{equation}
By Lagrange Theorem, for any given $x \in \R^2$ there exists  $\xi \in (0,h)$ such that
\begin{equation*}
\overline{u}_1 (x)= \frac{h}{\pi} \frac{e^{-x_2^2}}{1+(x_1+\xi)^2} ,
\end{equation*}
and hence, by using that, for any $\xi \in (0,h) \subseteq (0,1)$
\begin{equation}\label{eq:exarctan:stima derivata arctan +xi}
\frac{1}{1+(t+\xi)^2} \le \frac{3 + \sqrt{5}}{2} \, \frac{1}{1+t^2} \left( <  \frac{3}{1+t^2} \right) \quad \text{ for any } \, t \in \R ,
\end{equation}
we find that
\begin{equation}\label{eq:exarctan:STIMAU1}
\left| \overline{u}_1 (x) \right| \le \frac{3 h}{\pi} \frac{e^{-x_2^2}}{1+ x_1^2} \quad \text{ for any } x \in \R^2 .
\end{equation}
By means of straightforward computations and in light of Lagrange Theorem, from \eqref{eq:EXARCTAN: DEF OVERLINE U_1 U_2} we find that
\begin{equation}\label{eq:exarctan:PA1U1}
\partial_{x_1} \overline{u}_1 (x)= \frac{h}{\pi} e^{-x_2^2} \left[ \frac{ \frac{1}{ 1+(x_1+h)^2 } - \frac{1}{  1 + x_1^2  } }{h} \right] = \frac{h}{\pi} e^{-x_2^2} \left[  - \frac{2(x_1+ \xi)}{ \left( 1+(x_1+\xi)^2 \right)^2 } \right] \quad \text{ for some } \, \xi \in (0,h)
,\end{equation}
\begin{equation}\label{eq:exarctan:PA2U1}
\partial_{x_2} \overline{u}_1 (x)= \frac{h}{\pi} (-2 x_2)e^{-x_2^2} \left[  \frac{1}{1+(x_1+ \xi )^2}  \right]
\quad \text{ for some } \, \xi \in (0,h),
\end{equation}
\begin{equation}\label{eq:exarctan:PA1U2}
\partial_{x_1} \overline{u}_2 (x) =  \frac{h}{\pi} \frac{1}{1 + x_1^2}  \left[ \frac{ e^{-(x_2+h)^2} - e^{-x_2^2} }{h} \right] = \frac{h}{\pi} \frac{1}{1 + x_1^2}  \left[ - 2( x_2+ \xi )  e^{-(x_2+ \xi )^2} \right]  \quad \text{ for some } \, \xi \in (0,h)
,\end{equation}
and
\begin{equation}\label{eq:exarctan:PA2U2}
\begin{split}
\partial_{x_2} \overline{u}_2 (x) 
= &
\frac{h}{\pi} \arctan (x)  \left[ \frac{ -2 (x_2+ h ) e^{-(x_2+ h )^2} - ( -2 x_2) e^{-x_2^2} }{h} \right] 
\\
 = &
\frac{h}{\pi} \arctan (x)  \left[  -2  e^{-(x_2+ \xi )^2} + 4 (x_2 +\xi)^2 e^{-(x_2 + \xi )^2}  \right] 
\quad \text{ for some } \, \xi \in (0,h) .
\end{split}
\end{equation}

By using \eqref{eq:exarctan:PA1U1} and that, for any $\xi \in (0,h) \subseteq (0,1)$ it holds that
$$
\frac{ \left| t+\xi \right| }{(1+(t+\xi)^2)^2}
\le \frac{2}{1+t^2} \quad \text{ for any } t \in \R ,
$$
we get
\begin{equation}\label{eq:exarctan:NuovaPA1U1}
\left| \partial_{x_1} \overline{u}_1 (x) \right| \le \frac{4 h}{\pi} \frac{e^{-x_2^2}}{1+x_1^2}  \quad \text{ for any } \, x \in \R^2.
\end{equation}
Moreover, by using \eqref{eq:exarctan:stima derivata arctan +xi} and~\eqref{eq:exarctan:PA2U1}, and that
$$
\left| t \right| e^{-t^2} \le e^{-\frac{t^2}{2}} \quad \text{ for any } \, t \in \R ,
$$
we get that
\begin{equation}\label{eq:exarctan:NuovaPA2U1}
|\partial_{x_2} \overline{u}_1 (x)|= \frac{6h}{\pi} \frac{e^{-\frac{x_2^2}{2}}}{1+x_1^2} \quad \text{ for any } \, x \in \R^2.
\end{equation}
Also, by~\eqref{eq:exarctan:disuguaglianzaper t e^(-t^2)}
and~\eqref{eq:exarctan:PA1U2}, we have
\begin{equation}\label{eq:exarctan:NuovaPA1U2}
\left| \partial_{x_1} \overline{u}_2 (x) \right| \le \frac{2h}{\pi}
\end{equation}
and, by \eqref{eq:exarctan:RICICLATA}, \eqref{eq:exarctan:boundsuarctan},
and~\eqref{eq:exarctan:PA2U2},
we get
\begin{equation}\label{eq:exarctan:NuovaPA2U2}
\left| \partial_{x_2} \overline{u}_2 (x) \right| \le h
\end{equation}
By putting together \eqref{eq:ex con arctan:hsmallness},
\eqref{eq:exarctan:OVERLINEU2},
\eqref{eq:exarctan:STIMAPARTIAL1VARTHETA},
\eqref{eq:exarctan:STIMAU1},
\eqref{eq:exarctan:NuovaPA1U1}, \eqref{eq:exarctan:NuovaPA2U1},
\eqref{eq:exarctan:NuovaPA1U2}, \eqref{eq:exarctan:NuovaPA2U2},
and the trivial inequality
$$
e^{-t^2} \le e^{-\frac{t^2}{2}} \quad \text{ for any } t \in \R
$$
we find that, there exists a universal finite positive constant
$c$ (independent of~$h$ and~$x$) such that
\begin{equation}\label{eq:exarctan PREPARTI}
\left| \partial_{x_j}\overline{\vartheta}(x) \right| \le c  \, \frac{e^{-\frac{x_2^2}{2}}}{1+x_1^2} \, h \quad \text{ for any } x \in \R^2 , \, \, j = 1,2 .
\end{equation}

By recalling \eqref{eq:exarctan:stima derivata arctan +xi} and using that, for any $\xi \in (0,h) \subseteq (0,1)$
%
%
$$
e^{-\frac{(t+\xi)^2}{2}} \le \sqrt{e} \, e^{-\frac{t^2}{4}} \quad \text{ for any } t \in \R ,
$$
from \eqref{eq:exarctan:LAGRANGEDJ} and \eqref{eq:exarctan PREPARTI} we thus obtain that 
\begin{equation}\label{eq:exarctan:DJ STIMA PER CONVERGENZA}
\left| \mathcal{D}_j^\pm \vartheta_i^+ \right| \le c  \, \frac{e^{-\frac{i_2^2}{4}}}{1+i_1^2} \, h \quad \text{ for any } i \in h\Z^2, \, \, j=1,2,
\end{equation}
where $c$ is a universal finite positive constant (independent of $h$ and $i$).

We now claim that, there exists a universal finite positive constant $c$
(independent of $h$ and $x$) such that, for any $x \in \R^2$,
\begin{equation}\label{EQ:EXARCTAN:PROVA1}
|\overline{u}_1 (x)| = \overline{u}_1 \le c , \quad
|\overline{u}_2 (x)| = \overline{u}_2 \le c   ,
\end{equation}
\begin{equation}\label{EQ:EXARCTAN:PROVA2}
|\partial_{x_j} \overline{u}_k (x)| \le c \, h  ,  \quad {\mbox{ for any }}
j,k \in \left\lbrace 1,2 \right\rbrace,
\end{equation}
$$
|\partial_{x_j x_j} \overline{u}_k (x)| \le c \, h , \quad
|\partial_{x_j x_j x_j} \overline{u}_k (x)| \le c \, h , \quad
|\partial_{x_j x_j x_j x_j} \overline{u}_k (x)| \le c \, h  , \quad {\mbox{ for any }}
j,k \in \left\lbrace 1,2 \right\rbrace ,
$$
and hence, in light of \eqref{eq:ex con arctan:hsmallness},
\begin{equation}\label{eq:exarctan:BOUNDEDNESS U_j}
|\overline{u}_k (x)| \le c , \quad
|\partial_{x_j} \overline{u}_k (x)| \le c , \quad
|\partial_{x_j x_j} \overline{u}_k (x)| \le c  , \quad
|\partial_{x_j x_j x_j} \overline{u}_k (x)| \le c  , \quad
|\partial_{x_j x_j x_j x_j} \overline{u}_k (x)| \le c ,
\end{equation}
for any $x \in \R^2$ and $j,k \in \left\lbrace 1,2 \right\rbrace$.
The estimates \eqref{EQ:EXARCTAN:PROVA1} have been proved in \eqref{eq:exarctan:OVERLINEU1} and \eqref{eq:exarctan:OVERLINEU2}, while those in \eqref{EQ:EXARCTAN:PROVA2} clearly follow from
\eqref{eq:exarctan:NuovaPA1U1}, \eqref{eq:exarctan:NuovaPA2U1}, 
\eqref{eq:exarctan:NuovaPA1U2}, and~\eqref{eq:exarctan:NuovaPA2U2}. The estimates for the higher order derivatives follow by similar
straightforward computations.  


In light of \eqref{eq:exarctan:BOUNDEDNESS U_j}
and using \eqref{eq:exarctan:U1^2+U2^2}, straightforward computations
lead to find a universal positive constant (independent of $x$ and $h$) such that
\begin{equation}\label{eq:exarctan:STIMA DJJ THETA}
| \partial_{x_j x_j} \overline{\vartheta}(x)| \le c  , \quad
\text{ for any } \, x \in \R^2 {\mbox{ and }} j=1,2 ,
\end{equation}
and
\begin{equation}\label{eq:exarctan:STIMA DJJJ THETA}
| \partial_{x_j x_j x_j x_j} \overline{\vartheta}(x) | \le c  , 
\quad \text{ for any } \, x \in \R^2 {\mbox{ and }} j=1,2 .
\end{equation}
By putting together \eqref{eq:TAYLORBOUNDSUP DIRECTIONALLAPLACIAN} (with $v := \vartheta^+$ and $\overline{v} := \overline{\vartheta}$) and \eqref{eq:exarctan:STIMA DJJ THETA}, we thus find
\begin{equation}\label{eq:exarctan:DJJ Directional STIMA PER Convergenza}
\left| \mathcal{D}_j^\pm(\mathcal{D}_j^\pm \vartheta^+ )_{i- h e_j}
 \right| = \left| \mathcal{L}_j \vartheta_i^+ \right| \le c \, , 
\quad \text{ for any } \, i \in h\Z^2 {\mbox{ and }} j=1,2 .
\end{equation}

By using \eqref{K2}, \eqref{eq:exarctan:DJ STIMA PER CONVERGENZA}, \eqref{eq:exarctan:DJJ Directional STIMA PER Convergenza}, and the fact that by \eqref{eq:exarctan:bounds derivata in i_1} and \eqref{eq:exarctan:bounds derivata in i_2} it holds that
\begin{equation}\label{eq:exarctan:BOUND SU RHO^2}
(\rho_i^+)^2 \le 
\frac{4+ 9 \pi^2}{4 \pi^2},
\end{equation}
we thus find that
\begin{equation}\label{EQ:EXARCTAN: KAPPA 2}
\kappa_2^+:=
\sum_{{1\le j\le 2}\atop{i\in h\Z^2}}
(\rho_i^+)^2 \,\Big(
|{\mathcal{D}}_j^+\vartheta^+_i|\,|{\mathcal{D}}_j^+ ({\mathcal{D}}_j^+ \vartheta^+)_{i-he_j}|+
|{\mathcal{D}}_j^-\vartheta^+_i|\,|{\mathcal{D}}_j^- ({\mathcal{D}}_j^- \vartheta^+)_{i+he_j}|
\Big) \le c \, \left( \sum_{{1\le j\le 2}\atop{i\in h\Z^2}} \frac{e^{- \frac{i_2^2}{4}}}{1+i_1^2} \right) \, h  ,
\end{equation}
where the letter $c$ denotes a universal positive constant (independent
of $i$ and $h$).
Since the series in the brackets clearly converges and \eqref{eq:ex con arctan:hsmallness} holds true, \eqref{K2} is verified.

In order to prove \eqref{K33}, we set $\vartheta^+_\infty := {\pi}/{2}$. With this choice, by putting
together~\eqref{eq:arctanestimate},
\eqref{eq:exarctan:bounds derivata in i_2} and~\eqref{eq:exarctan:defvarthetainfty}, we get
\begin{equation*}
\left| \vartheta^+_\infty - \vartheta^+_i \right| \le
\frac{\mathcal{D}_1^+ u_i }{\mathcal{D}_2^+ u_i}
\le 
2 \, \mathcal{D}_1^+ u_i .
\end{equation*}
{F}rom this, by recalling \eqref{eq:exarctan:STIMAU1} and that $\mathcal{D}_1^+ u_i = \overline{u}_1(i)$ for any $i \in h\Z^2$,
%
%
we obtain
\begin{equation}\label{eq:exarctan:bound vartheta - vartheta_infty}
\left| \vartheta^+_\infty - \vartheta^+_i \right| \le
\frac{6 h}{\pi} \frac{e^{-i_2^2}}{1+ i_1^2} \, .
\end{equation}

To verify \eqref{K33}, it remains
to check that all the other terms appearing in \eqref{K33} remain bounded.
To this aim, we define
$$
\overline{\rho}(x) := \sqrt{ \overline{u}_1^2 (x) + \overline{u}_2^2 (x) } ,
$$
and we claim that 
\begin{equation}\label{EQ:EXARCTAN:DJ RHO}
\left| \partial_{x_j} \overline{\rho} (x) \right| \le c , \quad
\text{ for any } x \in \R^2
\end{equation}
and
\begin{equation}\label{EQ:EXARCTAN:DJ RHO^2}
\left| \partial_{x_j} \overline{\rho}^2 (x) \right| \le c , \quad \text{ for any } x \in \R^2
,\end{equation}
where $c$ is a positive finite universal constant (independent of $x$ and $h$).
Indeed, we compute that
$$
\partial_{x_j} \overline{\rho}^2 = 2 \left[ \overline{u}_1 \, ( \partial_{x_j} \overline{u}_1) + \overline{u}_2 \, ( \partial_{x_j} \overline{u}_2) \right] ,
$$
and therefore \eqref{EQ:EXARCTAN:DJ RHO^2}
follows by \eqref{eq:exarctan:BOUNDEDNESS U_j}.
We also compute
$$
\partial_{x_j} \overline{\rho} = \frac{\overline{u}_1 \, ( \partial_{x_j} \overline{u}_1) + \overline{u}_2 \, ( \partial_{x_j} \overline{u}_2) }{\sqrt{\overline{u}_1^2  + \overline{u}_2^2 }} ,
$$
from which, by recalling \eqref{eq:exarctan:U1^2+U2^2}, we find
$$
| \partial_{x_j} \overline{\rho} | \le 2 \left| \overline{u}_1 \, ( \partial_{x_j} \overline{u}_1) + \overline{u}_2 \, ( \partial_{x_j} \overline{u}_2) \right| ,
$$
and hence
\eqref{EQ:EXARCTAN:DJ RHO} follows by \eqref{eq:exarctan:BOUNDEDNESS U_j}.

By using \eqref{eq:TAYLORBOUNDSUP FIRSTDERIVATIVE}
with $v:= \rho$, from \eqref{EQ:EXARCTAN:DJ RHO} we get
\begin{equation}\label{EQ:EXARCTAN:DISCRETO DJ RHO}
\left| \mathcal{D}_j^\pm \rho_i \right| \le c , \quad \text{ for any } \, i
 \in h\Z^2 {\mbox{ and }} j=1,2 ,
\end{equation}
and from \eqref{EQ:EXARCTAN:DJ RHO^2} we get
\begin{equation}\label{EQ:EXARCTAN:DISCRETO DJ RHO^2}
\left| \mathcal{D}_j^\pm (\rho_i^2) \right| \le c , 
\quad \text{ for any } \, i \in h\Z^2 {\mbox{ and }} j=1,2 .
\end{equation}

By putting together \eqref{eq:TAYLORBOUNDSUP SECONDDERIVATIVES}
(used here with $v := \vartheta^+$ and $\overline{v}:=\overline{\vartheta}$) and \eqref{eq:exarctan:STIMA DJJ THETA},
we find
\begin{equation}\label{EQ:EXARCTAN:DJJ THETA DISCRETO BOUNDEDNESS}
\left| \mathcal{D}_j^\pm(\mathcal{D}_j^\pm \vartheta^+ )_{i}
 \right| \le c \, , \quad \text{ for any } \, i \in h\Z^2 {\mbox{ and }} j=1,2 .
\end{equation}
Furthermore, by~\eqref{eq:TAYLORBOUNDSUP DIRECTIONALLAPLACIANSQUARED} (with $v := \vartheta^+$ and $\overline{v} := \overline{\vartheta}$) and \eqref{eq:exarctan:STIMA DJJJ THETA}, we also find
\begin{equation}\label{EQ:EXARCTAN:DJJJJ THETA DISCRETO BOUNDEDNESS}
\left| \mathcal{L}_j^2 \vartheta_{i}^+ \right| \le c \, , 
\quad \text{ for any } \, i \in h\Z^2 {\mbox{ and }} j=1,2 .
\end{equation}
Finally, we notice that the estimate
in~\eqref{eq:exarctan:DJ STIMA PER CONVERGENZA} gives
that~$\left| \mathcal{D}_j^{\pm} \vartheta_{i} \right| \le c \,h$,
and hence, by \eqref{eq:ex con arctan:hsmallness},
\begin{equation}\label{EXARCTAN:PPPPPPPPPP}
\left| \mathcal{D}_j^{\pm} \vartheta_{i} \right| \le c \, , 
\quad \text{ for any } \, i \in h\Z^2 {\mbox{ and }} j=1,2 .
\end{equation}

By putting together
\eqref{EQ:EXARCTAN KAPPA_0},
\eqref{eq:exarctan:BOUND SU RHO^2},
\eqref{eq:exarctan:bound vartheta - vartheta_infty},
\eqref{EQ:EXARCTAN:DISCRETO DJ RHO},
\eqref{EQ:EXARCTAN:DISCRETO DJ RHO^2},
\eqref{EQ:EXARCTAN:DJJ THETA DISCRETO BOUNDEDNESS},
\eqref{EQ:EXARCTAN:DJJJJ THETA DISCRETO BOUNDEDNESS},
and~\eqref{EXARCTAN:PPPPPPPPPP}, and
recalling the definitions of~$\kappa^+_m$
in~\eqref{K33} we conclude that
\begin{equation}\label{EQ:EXARCTAN:KAPPA 3,4,5,6,7}
\kappa_m^+ \le c \left( \sum_{{1\le j\le 2}\atop{i\in h\Z^2}} 
\frac{e^{-i_2^2}}{1+ i_1^2} \right) \, h , \quad {\mbox{ for any }} m=3,4,5,6,7 ,
\end{equation}
where the letter $c$ denotes a universal positive
constant (independent of $i$ and $h$). Since the series
in the brackets clearly converges and \eqref{eq:ex con arctan:hsmallness} holds true, we obtain that~\eqref{K33} is satisfied.

All in all, we have that~$u_i$ and $f$ satisfy all the assumptions of Theorem \ref{DGDG}.
Nevertheless, $u_i$ is not one-dimensional and \eqref{eq:esempi div da zero} holds true.
More precisely, we can prove that there exists a universal positive
constant~$c_0$ (independent of $i$ and $h$) such that
\begin{equation}\label{eq:EXARCTAN:LOWERBOUNDTHM}
\sum_{{1\le j\le 2}\atop{i\in h\Z^2}}
(\rho_i^+)^2 \;\Big(
|{\mathcal{D}}_j^+\vartheta^+_i|^2+|{\mathcal{D}}_j^-\vartheta^+_i|^2
\Big) \ge c_0 \, h^2 .
\end{equation}
To prove \eqref{eq:EXARCTAN:LOWERBOUNDTHM}, we take
\begin{equation}\label{EQ:LB SET I_1HAT: EXARCTAN}
\hat{i}_1 \in h\Z \cap \left[-4 , -3 \right]
\end{equation}
and we notice that, by \eqref{eq:exarctan:bounds derivata in i_2},
$$
(\rho_{i}^+)^2 \ge \frac{1}{4} \quad \text{ for any } \, i \in h\Z^2.
$$
Accordingly,
\begin{equation}\label{EQ:LB EXARCTAN STEP 1}
\sum_{{1\le j\le 2}\atop{i\in h\Z^2}}
(\rho_i^+)^2 \;\Big(
|{\mathcal{D}}_j^+\vartheta^+_i|^2+|{\mathcal{D}}_j^-\vartheta^+_i|^2
\Big) \ge  (\rho_{( \hat{i}_1 ,0)}^+)^2 \;\Big(
|{\mathcal{D}}_1^+\vartheta^+_{( \hat{i}_1 ,0) }|^2
\Big) \ge 4 \, |{\mathcal{D}}_1^+\vartheta^+_{( \hat{i}_1 ,0) }|^2 .
\end{equation}
Also, by using Lagrange Theorem, we have that
\begin{equation}\label{EQ:LB EXARCTAN STEP 2}
{\mathcal{D}}_1^+\vartheta^+_{( \hat{i}_1 ,0) } =
 \partial_{x_1} \overline{\vartheta} ( \hat{i}_1 + \eta ,0) 
\quad \text{ for some } \, \eta \in (0,h).
\end{equation}
By putting together \eqref{EQ:LB EXARCTAN SERVE 1}, the second equality in
\eqref{eq:exarctan:PA1U1},
and the first equality in
\eqref{eq:exarctan:PA1U2} (with $x = ( \hat{i}_1 + \eta ,0)$), we thus compute
$$
\partial_{x_1} \overline{\vartheta} ( \hat{i}_1 + \eta,0)= - \frac{      
\left\lbrace \frac{h}{\pi} \frac{1}{1 + ( \hat{i}_1 +  \eta)^2}  \left[ \frac{ 1 - e^{-h^2}  }{h} \right] \right\rbrace
\, \overline{u}_1 ( \hat{i}_1 +  \eta,0)
+
\overline{u}_2 ( \hat{i}_1 +  \eta,0) \,
\left\lbrace \frac{2 \, h}{\pi} \left[   \frac{ - ( \hat{i}_1 +  \eta + \xi)}{ \left( 1+( \hat{i}_1 +  \eta +\xi)^2 \right)^2 } \right]  \right\rbrace }{\overline{u}_1^2 ( \hat{i}_1 + \eta,0) +\overline{u}_2^2 ( \hat{i}_1 + \eta,0) } ,
$$
where $\xi \in (0,h)$ is that appearing in \eqref{eq:exarctan:PA1U1}
and $\eta \in (0,h)$ is that appearing in \eqref{EQ:LB EXARCTAN STEP 2}.
By noting that, in light of \eqref{eq:ex con arctan:hsmallness} and \eqref{EQ:LB SET I_1HAT: EXARCTAN}, the terms in the braces are
non-negative, and using the lower bounds
in \eqref{eq:exarctan:OVERLINEU1} and \eqref{eq:exarctan:OVERLINEU2},
we obtain that
\begin{equation}\label{EQ:LB EXARCTAN STEP FORSE}
\begin{split}
\left| \partial_{x_1} \overline{\vartheta} ( \hat{i}_1 +  \eta,0) \right|  = &
\frac{      
\left\lbrace \frac{h}{\pi} \frac{1}{1 + ( \hat{i}_1 + \eta)^2}  \left[ \frac{ 1 - e^{-h^2}  }{h} \right] \right\rbrace
\, \overline{u}_1 ( \hat{i}_1 +  \eta,0)
+
\overline{u}_2 ( \hat{i}_1 +  \eta,0) \,
\left\lbrace \frac{2 \, h}{\pi} \left[   \frac{ - ( \hat{i}_1 +  \eta + \xi)}{ \left( 1+( \hat{i}_1 +  \eta +\xi)^2 \right)^2 } \right]  \right\rbrace }{\overline{u}_1^2 ( \hat{i}_1 + \eta,0) +\overline{u}_2^2 ( \hat{i}_1 + \eta,0) }
\\
 & \ge 
\frac{
\frac{h}{ \pi}\cdot
\frac{ - ( \hat{i}_1 +  \eta + \xi)}{ \left( 1+( \hat{i}_1 +  \eta +\xi)^2
 \right)^2 }}{\overline{u}_1^2 ( \hat{i}_1 + \eta,0)  +\overline{u}_2^2 ( \hat{i}_1 + \eta,0) } 
\\
& \ge 
\frac{4 \pi \, h}{ 4 + 9 \pi^2 } \cdot
 \frac{ - ( \hat{i}_1 +  \eta + \xi)}{ \left( 1+( \hat{i}_1 + 
 \eta +\xi)^2 \right)^2 },
\end{split}
\end{equation}
where we have also used
in the last inequality the fact that
$$
\overline{u}_1^2 (i) + \overline{u}_2^2 (i) \le 
\frac{4+9 \pi^2}{4 \pi^2}  \quad \text{ for any } \, i \in h\Z^2,
$$
which follows from the upper bounds in \eqref{eq:exarctan:OVERLINEU1}
and \eqref{eq:exarctan:OVERLINEU2}.

By using that $\xi, \eta \in (0,h)$ and recalling \eqref{eq:ex con arctan:hsmallness} and \eqref{EQ:LB SET I_1HAT: EXARCTAN} we have that
$$
1 \le - (\hat{i}_1 +  \eta +\xi ) \le 4 ,
$$
and hence \eqref{EQ:LB EXARCTAN STEP FORSE} gives
$$
\left| \partial_{x_1} \overline{\vartheta} ( \hat{i}_1 +  \eta,0) \right| \ge 
\frac{4 \pi }{ 289 (4 + 9 \pi^2) } \, h ,
$$
from which, by recalling \eqref{EQ:LB EXARCTAN STEP 2}, we find
$$
\left| {\mathcal{D}}_1^+\vartheta^+_{( \hat{i}_1 ,0) } \right| \ge \frac{4 \pi }{ 289 (4 + 9 \pi^2) } \, h.
$$
The last inequality and \eqref{EQ:LB EXARCTAN STEP 1} clearly give \eqref{eq:EXARCTAN:LOWERBOUNDTHM} with $c_0= 4 \left[ \frac{4 \pi }{ 289 (4 + 9 \pi^2) } \right]^2$.

On the other hand, in light of \eqref{EQ:EXARCTAN: KAPPA 2} and \eqref{EQ:EXARCTAN:KAPPA 3,4,5,6,7}, there exists a universal positive
constant $c_1$ (independent of $h$ and $i$) such that
\begin{equation}\label{eq:EXARCTAN:UPPERBOUNDTHM} 
C:= 4 \left( \kappa_2^++ 2 e^{2\pi}\kappa_3^++ 2 e^{2\pi}\kappa_4^++2 \kappa_5^+ + \kappa_6^+ + \kappa_7^+ \right)
\le c_1 \, h .
\end{equation}
By putting together \eqref{FORM},
\eqref{eq:EXARCTAN:LOWERBOUNDTHM},
and \eqref{eq:EXARCTAN:UPPERBOUNDTHM}
we thus find
$$
c_0 \, h^2 \le \sum_{{1\le j\le 2}\atop{i\in h\Z^2}}
(\rho_i^+)^2 \;\Big(
|{\mathcal{D}}_j^+\vartheta^+_i|^2+|{\mathcal{D}}_j^-\vartheta^+_i|^2
\Big)\le C h \le c_1 \, h^2 ,
$$
where $c_0$ and $c_1$ are two positive
universal constants (independent of $h$).
Thus,
left-hand side and right-hand side of \eqref{FORM} are both of the order of $h^2$. In this sense, \eqref{FORM} is optimal.

\smallskip 

We conclude this example by proving \eqref{eq:exarctan:FINALEALLAFINE}.
%
%
To this aim, notice that, since \eqref{eq:ex con arctan:hsmallness} is in force, for any $c\in \R$ the level curve
$$
L_c:= \left\lbrace x_2 + \frac{h}{\pi} \, e^{-x_2^2} \arctan(x_1)=c , \quad{\mbox{for}}\quad (x_1,x_2) \in \R^2 \right\rbrace
$$
passes through $(0,c)$ and is contained in $ \R  \times \left( c - \frac{h}{2}, c + \frac{h}{2} \right)  $; see Figure~\ref{FIG1}
for a sketch of these level curves.

\begin{figure}[h!]
\includegraphics[scale=0.8]{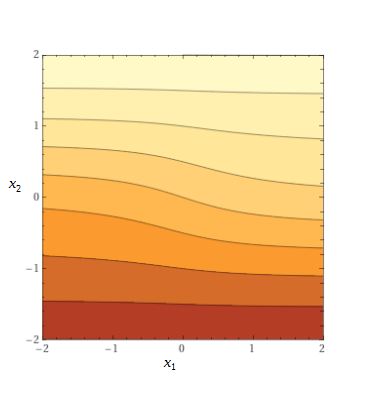}
\caption{Level curves~$L_c$ for $h=1$.}\label{FIG1}
\end{figure}

Since the ``height'' (in the $x_2$-direction) of $L_c$ is less
than $h$ and $L_c$ is decreasing in
the\footnote{By the implicit function theorem,
$L_c$ is (locally) the graph of a function $g(x_1)$
of the $x_1$ variable, which is decreasing. Indeed,
by setting $\overline{u}(x):= x_2 +
 \frac{h}{\pi} \, e^{-x_2^2} \arctan(x_1)$, we have
$$
g'= - \frac{ \partial_{x_1} \overline{u} }{ \partial_{x_2} \overline{u} } <0 ,
$$
being $\partial_{x_1} \overline{u} = 
\frac{h}{\pi} \, \frac{e^{-x_2^2}}{ 1 + x_1^2 } >0 $
by \eqref{eq:ex con arctan:hsmallness}, and
$\partial_{x_2} \overline{u} = 1 - \frac{h}{\pi} 2 x_2^2
 \, e^{-x_2^2} \arctan(x_1) > \frac{1}{2}>0 $
by~\eqref{eq:ex con arctan:hsmallness},
\eqref{eq:exarctan:disuguaglianzaper t e^(-t^2)},
and \eqref{eq:exarctan:boundsuarctan}.}
$x_1$-direction, we have that
for any $c \in \R$, $L_c$ intersects (at most) only one horizontal
line of the family $x_2= h z$, $z \in \Z$,
and this intersection is given by a single point.
From this, we deduce that the function $u : h \Z^2 \to \R$ is injective. 
Thus, by denoting with $\mathrm{Im}(u)$ the image of $u$, there exists $u^{-1}: \mathrm{Im}(u) \to h \Z^2$ such that 
\begin{equation}\label{eq:leftinverse}
u^{-1}(u_i)=i , \quad \text{for all } \, i \in h \Z^2.
\end{equation}

Now we define $\hat{f} : \R \to \R$ as follows
$$
\hat{f}(x)=
\begin{cases}
\tilde{f} \circ u^{-1}(x) , & \quad \text{ for } \, x \in \mathrm{Im}(u),
\\
0 , & \quad \text{ for } \, x \in \R \setminus \mathrm{Im}(u) .
\end{cases}
$$
For all $i \in h \Z^2$ we have
$$
\tilde{f}(i)= \tilde{f} \circ u^{-1} (u_i) = \hat{f} (u_i) ,
$$
and hence \eqref{eq:exarctan:FINALEALLAFINE} holds true. By \eqref{eq:ex arctan:equazione}, $u$ satisfies
\eqref{DGEQ} with
\begin{equation*}
f(i,u_i):=\tilde{f}(i)= \hat{f}(u_i) .
\end{equation*}
In particular, $u$ is solution of
$$
{\mathcal{L}} u_i = \hat{f}(u_i) ,
$$
where the source term $\hat{f}$ only depends on $u$.
}
\end{example}

\begin{example}\label{example EXPONENTIAL}
{\rm
For any
\begin{equation}\label{EQ:EXEXPONENTIAL H SMALLNESS}
0 < h \le 1 .
\end{equation}
we consider
$$
u_i: = i_2 + \frac{h}{2} \, e^{-|i|^2} 
$$
and we define
$$
\tilde{f}(i) := \mathcal{L}u_i =\frac{e^{-|i|^2}}h\,\sum_{j=1}^2\big(
e^{-2hi_j-h^2}+e^{2hi_j-h^2}-2
\big)
$$
and
\begin{equation}\label{eq:exexponential EQUATION TILDE}
f(i, u_i) := \tilde{f}(i) .
\end{equation}
In this way, 
we have that~\eqref{DGEQ} holds true. Also, $f$
satisfies~\eqref{hp:semilinearitaconspazio}. Indeed,
given~$i\in h\Z^2$,
we have that the map~$\R\ni r\mapsto f(i,r)$ is constant,
thus~$f'(i,\cdot)=0$. Additionally,
\begin{eqnarray*}
\big|e^{-2hi_j}+e^{2hi_j}-2\big|&=&
\left| \sum_{k=0}^{+\infty} \frac{(-2hi_j)^k}{k!}
+\sum_{k=0}^{+\infty} \frac{(2hi_j)^k}{k!}
-2\right|\\&=&
\left| \sum_{k=2}^{+\infty} \frac{(-2hi_j)^k}{k!}
+\sum_{k=2}^{+\infty} \frac{(2hi_j)^k}{k!}\right|\\&\le&
2\sum_{k=2}^{+\infty} \frac{(2h|i|)^k}{k!}\\&=&
8  h^2|i|^2\sum_{j=0}^{+\infty} \frac{(2h|i|)^j}{(j+2)!}\\&\le&
8  h^2|i|^2\sum_{j=0}^{+\infty} \frac{(2h|i|)^j}{j!}\\&=&
8  h^2|i|^2e^{2h|i|}\\&\le&8  h^2|i|^2e^{2|i|}
\end{eqnarray*}
and, as a result,
\begin{eqnarray*} &&\big|
e^{-2hi_j-h^2}+e^{2hi_j-h^2}-2
\big|\le
\big|
e^{-2hi_j-h^2}+e^{2hi_j-h^2}-2e^{-h^2}
\big|+2(1-e^{-h^2})\\
&&\qquad\le
\big|
e^{-2hi_j}+e^{2hi_j}-2
\big|+2\left(1-\sum_{k=0}^{+\infty}\frac{(-h^2)^k}{k!}\right)\\&&\qquad\le
8  h^2|i|^2e^{2h|i|}
-2\sum_{k=1}^{+\infty}\frac{(-h^2)^k}{k!}
\\&&\qquad=
8  h^2|i|^2e^{2h|i|}
+2h^2\sum_{j=0}^{+\infty}\frac{(-h^2)^j}{(j+1)!}\\&&\qquad\le
8  h^2|i|^2e^{2h|i|}
+2h^2\sum_{j=0}^{+\infty}\frac{h^{2j}}{j!}\\&&\qquad\le
8  h^2|i|^2e^{2h|i|}
+2e^{h^2}h^2\\&&\qquad\le
8  h^2\big(|i|^2e^{2h|i|}
+1\big).
\end{eqnarray*}
Consequently,
\begin{eqnarray*}
|\nabla\tilde f(i)|&\le&
\frac{2|i|\,e^{-|i|^2}}h\,\sum_{j=1}^2\big|
e^{-2hi_j-h^2}+e^{2hi_j-h^2}-2
\big|+
2e^{-|i|^2}\,\sum_{j=1}^2\big|
e^{-2hi_j-h^2}-e^{2hi_j-h^2}
\big|\\&\le&
32\,h\, |i|\,e^{-|i|^2}\big(|i|^2e^{2h|i|}
+1\big)
+
2e^{-|i|^2}\,\sum_{j=1}^2\left|\int_{-2hi_j}^{2hi_j}e^t\,dt
\right|\\&\le&
32\,h\, |i|\,e^{-|i|^2+2h|i|}\big(|i|^2
+1\big)
+
16\,h\,|i|\,e^{-|i|^2+2h|i|}
\\&\le&
32\,h\, |i|\,e^{-|i|^2+2|i|}\big(|i|^2
+2\big)\\&\le&
32 \,S \,h,
\end{eqnarray*}
where
$$ S:=\sup_{t\in\R}\Big(
t\,e^{-t^2+2t}\big(t^2
+2\big)\Big).$$
Therefore, for all~$ i \in h\Z^2 $,
\begin{eqnarray*}
&& \sum_{j=1}^{2} \frac{ \big| f(i+he_j, u_{i+he_j})-f(i, u_i)-
 f' (i, u_i)(u_{i+he_j}-u_i)\big| }{h}
=
\sum_{j=1}^{2} \frac{ \big| \tilde f(i+he_j)-\tilde f(i)\big| }{h}
\le \kappa_0^+
 \,  h,
\end{eqnarray*}
with~$\kappa_0^+:=64 \,S $, which yields~\eqref{hp:semilinearitaconspazio}.

In addition,
for any $h>0$, $u$ is not one-dimensional,  
\eqref{eq:esempi div da zero} holds true, and, if $h$ is small enough, then~\eqref{hp:regularitysemilinearity}, \eqref{hp:semilinearitaconspazio}, \eqref{BOULI}, \eqref{K2},
\eqref{K33}, 
and~\eqref{MONO} (and hence \eqref{hp:graddiversodazero}), 
are all satisfied. Thus, Theorem \ref{DGDG} applies, but~$u$ is not one-dimensional.
The computations to verify all the assumptions are similar to those of Example \ref{example ARCTAN}.

In the formal limit as $h \searrow 0$, an asymptotic analysis similar to that of Example \ref{example ARCTAN} can be performed. As in Example \ref{example ARCTAN}, left-hand side and right-hand side of \eqref{FORM} are both of the order of $h^2$. Thus, also this example confirms the optimality of \eqref{FORM}.

We conclude by showing that, also the function presented in this example can be seen as solution of an equation of the type
$$
{\mathcal{L}} u_i = \hat{f}(u_i) ,
$$
where the source term $\hat{f}$ only depends on $u$.

For any $c\in \R$, the ``height'' (in the $x_2$-direction) of the level curve
$$
L_c:= \left\lbrace x_2 + \frac{h}{2} \, e^{-(x_1^2+ x_2^2)}=c , \quad{\mbox{for}}\quad (x_1,x_2) \in \R^2 \right\rbrace
$$
is less than $h/2$, whenever \eqref{EQ:EXEXPONENTIAL H SMALLNESS} is in force. 
Also,
$L_c$ is symmetric with respect to the $x_2$-axis --- i.e., $(x_1,x_2) \in L_c$ if and only if $(-x_1,x_2) \in L_c $ ---, and it is increasing (resp. decreasing) in the\footnote{By the
implicit function theorem, $L_c$ is (locally) the graph
of a function $g(x_1)$ of the $x_1$ variable,
which is increasing (resp. decreasing) in the $x_1$-direction for
$x_1>0$ (resp. $x_1 <0$). Indeed, by setting $\overline{u}(x):= x_2 + \frac{h}{2} \, e^{-x_1^2 - x_2^2} $, we have
$$
g'= - \frac{ \partial_{x_1} \overline{u} }{ \partial_{x_2} \overline{u} }
\, \, 
\begin{cases} 
> 0 \quad \text{ if } \, x_1>0
\\
< 0 \quad \text{ if } \, x_1<0
,
\end{cases}
$$
being 
$$
\partial_{x_1} \overline{u} = - h \, x_1 \, e^{-x_1^2 -x_2^2} \, \,
\begin{cases} 
< 0 \quad \text{ if } \, x_1>0
\\
> 0 \quad \text{ if } \, x_1<0 
\end{cases}
$$ 
by \eqref{eq:ex con arctan:hsmallness}, and $\partial_{x_2}
 \overline{u} = 1 - h ( x_2 \, e^{-x_2^2} ) e^{-x_1^2} > 
\frac{1}{2}>0 $ by  \eqref{eq:ex con arctan:hsmallness} 
and \eqref{eq:exarctan:disuguaglianzaper t e^(-t^2)}. }
$x_1$-direction for $x_1>0$
(resp. $x_1 <0$). Thus, we have that for any $c \in \R$, $L_c$ intersects (at most) only one horizontal line of the family $x_2= h z$, $z \in \Z$, and this intersection is given by a single point 
$$
(\overline{x}_1, \overline{x}_2) \in \R^+ \times \R:= \left\lbrace (x_1,x_2) \in \R^2 \, : \, x_1 \ge 0 \right\rbrace 
$$
in the right half-space, and its symmetric 
$$
(-\overline{x}_1, \overline{x}_2) \in \R^- \times \R:= \left\lbrace (x_1,x_2) \in \R^2 \, : \, x_1 \le 0 \right\rbrace 
$$
in the left half-space. 
See Figure~\ref{FIG2}
for a sketch of the level curves $L_c$.

\begin{figure}[h!]
\includegraphics[scale=0.8]{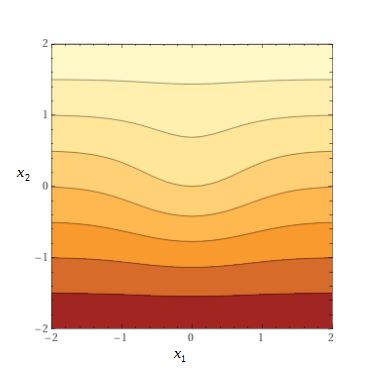}
\caption{Level curves~$L_c$ for $h=1$.}\label{FIG2}
\end{figure}

Hence, the function $u : h \Z^2 \to \R$ is injective on $h\Z^+ \times h\Z$ (or $h\Z^- \times h\Z$), where $h\Z^+ \times h\Z$ (resp. $h\Z^- \times h\Z$) is the set of points~$i=(i_1,i_2) \in h\Z^2$ such that $i_1 \ge 0$ (resp. $i_1 \le 0$).

Thus, by denoting with $\mathrm{Im}(u)$ the image of $u$,
%
%
there exists $u^{-1}: \mathrm{Im}(u) \to h \Z^+ \times h \Z$ such that 
\begin{equation*}
u^{-1}(u_i)=i , \quad \text{ for all }  i \in h \Z^+ \times h\Z.
\end{equation*}

Now we define $\hat{f} : \R \to \R$ as follows
$$
\hat{f}(x)=
\begin{cases}
\tilde{f} \circ u^{-1}(x)  & \quad \text{ for } \, x \in \mathrm{Im}(u),
\\
0  & \quad \text{ for } \, x \in \R \setminus \mathrm{Im}(u) ,
\end{cases}
$$
and we notice that
$$
\tilde{f}(i)= \tilde{f} \circ u^{-1} (u_i) = \hat{f} (u_i) , \quad \text{ for all } \, i \in h \Z^2 .
$$
The previous identity clearly holds true, by definition of $\hat{f}$, whenever $i \in h\Z^+ \times h\Z$. However, it remains true even if $i \in h\Z^- \times h\Z $. Indeed, thanks to the symmetry properties
$$
u_{(i_1,i_2)} = u_{(-i_1,i_2)} \quad \text{ for any } i=(i_1,i_2) \in h\Z^2 ,
$$
$$\tilde{f}(i_1, i_2) = \tilde{f}(-i_1,i_2) \quad \text{ for any } i=(i_1,i_2) \in h\Z^2 ,
$$
if $i=(i_1,i_2) \in h\Z^- \times h\Z$ we can still write
$$
\tilde{f}(i)= \tilde{f}(i_1, i_2) = \tilde{f}(-i_1,i_2) = \tilde{f} \circ u^{-1} (u_{(-i_1,i_2)}) = \hat{f} (u_{( - i_1,i_2)})  = \hat{f} (u_{(i_1,i_2)}) = \hat{f} (u_i).
$$

Thus, for any $i \in h\Z^2$ \eqref{DGEQ} holds true with
\begin{equation*}
f(i,u_i):=\tilde{f}(i)= \hat{f}(u_i) .
\end{equation*}
In particular, by \eqref{eq:exexponential EQUATION TILDE}, $u$ is solution of
$$
{\mathcal{L}} u_i = \hat{f}(u_i) ,
$$
where the source term $\hat{f}$ only depends on $u$.
}
\end{example}

The previous three examples have been obtained by perturbing the one-dimensional linear function $i_2$. We stress that more general examples could be obtained by perturbing different one-dimensional functions.
For instance, the following example is obtained by perturbing a (strictly monotone) one-dimensional solution $\tilde{v}: \R \to \R$ of a general semilinear autonomous equation $\tilde{v}'' = g( \tilde{v})$, with the perturbation already used in Example \ref{example 1}.  
As a concrete reference case, one may think to, e.g., $\tilde{v}(t) := 4 \arctan(e^{t})$ which satisfies the stationary sine-Gordon equation
\begin{equation}\label{eq:ex4:sineGordoneq}
\tilde{v}'' = \sin \left( \tilde{v} \right) .
\end{equation}

Of course, Example \ref{example 1} is a particular case of Example \ref{example 4: general semilinear} (in which $\tilde{v}(t):=t$ and $g\equiv 0$).

\begin{example}\label{example 4: general semilinear}
{\rm
Consider $\tilde{v}: \R \to \R$ satisfying the semilinear equation
\begin{equation}\label{eq:ex4:g semilinear}
\tilde{v}'' = g( \tilde{v} ) ,
\end{equation}
where $g: \R \to \R$.
For $x=(x_1, x_2) \in \R^2$ we set
\begin{equation}\label{eq:ex4:dev ol v con v tilde}
\overline{v}(x) := \tilde{v} (x_2)  ,
\end{equation}
and for $i=(i_1 , i_2)$ we define
\begin{equation}\label{AAASSSS}
v_i := \overline{v}(i),\end{equation}
which is the restriction of $\overline{v}$ to $h\Z^2$.
With these definitions we clearly have
\begin{eqnarray*}
&&\Delta \overline{v} (x)= \tilde{v}''(x_2) , \quad{\mbox{ for all }} x \in \R^2,
\\
{\mbox{and }}&&
{\mathcal{L}} v_i = {\mathcal{L}}_2 v_i , \quad{\mbox{ for all }}
i \in \R^2,
\end{eqnarray*}
where we have used the notation in~\eqref{0ok8uh7gfr8} for~${\mathcal{L}}_2$.

We now show that, if $\Vert \tilde{v}' \Vert_{L^\infty (\R)} 
<\infty $, $\Vert \tilde{v}^{( iv)} 
\Vert_{L^\infty( \R)} < \infty$  and
$\Vert g'' \Vert_{L^\infty ( \R )}<\infty$,
then $v$, as defined in~\eqref{AAASSSS},
satisfies assumption \eqref{hp:GENERALsemilinearitaconspazio}
with $f(i,u_i)$ and $L_f^+ (i, u_i)$ replaced by ${\mathcal{L}} v_i$
and $g'(v_i)$. Here, $\tilde{v}^{( iv)} $ denotes the fourth derivative
of~$\tilde{v}$.

Indeed, a fourth order Taylor expansion with Lagrange
remainder terms shows that
%
%
\begin{equation}\label{eq:ex4:step1}
\begin{split}&
\frac{\left({\mathcal{L}} v\right)_{i+h e_2} + \left({\mathcal{L}} v\right)_{i} - g'(v_i)(v_{i+h e_2} - v_i)}{h}
\\&\qquad=\, \frac{h}{4!} \left[ \sum_{k=1}^4 \tilde{v}^{( iv)}(i+\xi_k e_2)   \right] 
+ \frac{ \tilde{v}''(i+h e_2) - \tilde{v}''(i) - g'(v_i)(v_{i+h e_2} - v_i) }{h}
\end{split}
\end{equation}
for some $\xi_1$, $ \xi_4 \in (0,h)$, $\xi_2 \in ( 0 , -h )$, $\xi_3 \in ( h, 2h)$.
By using \eqref{eq:ex4:g semilinear} we compute
\begin{equation}\label{eq:ex4:step2}
\begin{split}
\frac{ \tilde{v}''(i+h e_2) - \tilde{v}''(i) - g'(v_i)(v_{i+h e_2} - v_i) }{h}
= & \,  \frac{ g(i+h e_2) - g(i) - g'(v_i)(v_{i+h e_2} - v_i) }{h}
\\
= & \, \frac{ g''( v_i + \eta ) }{ 2 } \, \left[ \frac{v_{ i + h e_2} - v_i }{h} \right]^2 \, h
\\
= & \, \frac{ g'' ( v_i + \eta ) }{ 2 } \, \left[ v'( i + \xi_5 ) \right]^2 \, h ,
\end{split}
\end{equation}
for some $\eta \in (0, v_{i+h e_2} - v_i)$, $\xi_5 \in (0,h)$.
By using that $v$ only depends on $i_2$ and putting together \eqref{eq:ex4:step1}
and~\eqref{eq:ex4:step2} we get that
\begin{equation}\label{eq:ex4:condizione k0 su v}
\begin{split}
\sum_{j=1}^2 \frac{ \left( {\mathcal{L}} v \right)_{i+h e_j} + \left( {\mathcal{L}} v \right)_{i} - g'(v_i)( v_{i+h e_j} - v_i )}{h}
= & \,  \frac{ \left( {\mathcal{L}} v \right)_{i+h e_2} + \left( {\mathcal{L}} v \right)_{i} - g'(v_i)( v_{i+h e_2} - v_i )}{h} 
\\
\le & \, 
h \left\lbrace \frac{ \Vert \tilde{v}^{( iv)} \Vert_{L^\infty( \R)} }{6} + \frac{ \Vert g'' \Vert_{L^\infty ( 
\R )}}{2} \Vert \tilde{v}' \Vert_{L^\infty ( \R )}  \right\rbrace ,
\end{split}
\end{equation}
that is, $v$ satisfies \eqref{hp:GENERALsemilinearitaconspazio} --- where $f(i,u_i)$ and $L_f^+ (i, u_i)$ are replaced by ${\mathcal{L}} v_i$ and $g'(v_i)$ --- with 
$$
k_0^+ = \frac{ \Vert \tilde{v}^{( iv)} \Vert_{L^\infty( \R)} }{6} + \frac{ \Vert g'' \Vert_{L^\infty (
\R )}}{2} \Vert \tilde{v}' \Vert_{L^\infty ( \R )}.
$$

Notice that in concrete cases in which $\tilde{v}$ and $g$ are explicitly given\footnote{As a concrete example, one may consider the function
$\tilde{v}(t) := 4 \arctan(e^{t})$ which satisfies the stationary sine-Gordon equation \eqref{eq:ex4:sineGordoneq}.
In this case, $g(t) := \sin \left( t \right)$, $t \in \R$,
$\overline{v}(x) := 4 \arctan(e^{x_2}) $, $x=(x_1, x_2) \in \R^2$,
and $v_i := 4 \arctan(e^{i_2})$ is the restriction of $\overline{v}$ to $h\Z^2$.
}, one could explicitly compute $\Vert \tilde{v}' \Vert_{L^\infty ( \R )}$, $\Vert \tilde{v}^{(iv)} \Vert_{L^\infty( \R)}$ and $\Vert g'' \Vert_{L^\infty ( 
\R )}$.

We now denote by $w$ the perturbation already used in Example \ref{example 1}, that is,
\begin{equation}\label{ex4:eq:w def}
w_i:=
\begin{cases}
h^4 \quad & \text{for } \, i = (i_1,0) , \, i_1 > 0,
\\
0 \quad \, & \text{otherwise in }  h\Z^2 .
\end{cases}
\end{equation}
For any 
\begin{equation}\label{eq:ex4:h tra 0 e 1}
0<h<1
\end{equation}
and $i=(i_1,i_2) \in h\Z^2$, we then consider
\begin{equation}\label{eq:ex4:def u}
u_i:= v_i + w_i
\end{equation}
and we set 
\begin{equation}\label{eq:ex4:def f per u}
f(i, u_i):= \tilde{f}(i):=  {\mathcal{L}} v_i + {\mathcal{L}}w_i ,
\end{equation}
so that~${\mathcal{L}}u_i=\tilde f(i)$,
and therefore~\eqref{DGEQ} is satisfied.

We now show that, if $\Vert \tilde{v}' \Vert_{L^\infty (\R)} 
<\infty $, $\Vert \tilde{v}^{( iv)} 
\Vert_{L^\infty( \R)} < \infty$, $\Vert g' \Vert_{L^\infty ( \R )}<\infty$  and
$\Vert g'' \Vert_{L^\infty ( \R )}<\infty$,
then $f$ satisfies \eqref{hp:GENERALsemilinearitaconspazio} with $L_f^+(i, u_i) := g'(v_i) $.

To check this, we start by computing 
\begin{equation}\label{eq:ex4:step3}
\begin{split}
 & \sum_{j=1}^2 \frac{ \left| \tilde{f}(i+h e_j)- \tilde{f}(i) - g'(v_i)(u_{i+ h e_j} - u_i) \right| }{h} 
\\ 
& =  \, 
\sum_{j=1}^2 \frac{ \left| \left( {\mathcal{L}} v\right)_{i+h e_j} + \left( {\mathcal{L}} w\right)_{i+h e_j} - \left( {\mathcal{L}} v\right)_{i} - \left( {\mathcal{L}} w\right)_{i}  - g'(v_i)(v_{i+ h e_j} + w_{i+ h e_j} - v_i - w_i )  \right| }{h}
\\
& \le  \, 
\left\lbrace \sum_{j=1}^2  \frac{ \left| \left( {\mathcal{L}} v\right)_{i+h e_j}  - \left( {\mathcal{L}} v\right)_{i}  - g'(v_i)(v_{i+ h e_j} - v_i )  \right| }{h} \right\rbrace
 +
\left\lbrace  \sum_{j=1}^2  \frac{ \left| \left( {\mathcal{L}} w\right)_{i+h e_j} -  \left( {\mathcal{L}} w\right)_{i}  - g'(v_i)(w_{i+ h e_j} - w_i )  \right| }{h} \right\rbrace
.
\end{split}
\end{equation}
Here, in the last inequality we used the triangle inequality.

The term in the first braces in \eqref{eq:ex4:step3} can be estimated by means of \eqref{eq:ex4:condizione k0 su v}.
We now estimate the term in the second braces as follows:
\begin{equation}\label{eq:ex4:step4}
\begin{split}
\sum_{j=1}^2 & \frac{ \left| \left( {\mathcal{L}} w \right)_{i+h e_j} -  \left( {\mathcal{L}} w\right)_{i}  - g'(v_i)( w_{i+ h e_j} - w_i )  \right| }{h} 
\\
\le &  
\sum_{j=1}^2 \frac{ \left| \left( {\mathcal{L}} w \right)_{i+h e_j} -  \left( {\mathcal{L}} w\right)_{i}  \right| }{h} + \Vert g' \Vert_{L^\infty ( 
\R )}  \frac{ \left| w_{i+ h e_j} - w_i  \right| }{h}
\\
\le & \,
\sum_{j=1}^2 \frac{ \left| \left( {\mathcal{L}} w\right)_{i+h e_j} -  \left( {\mathcal{L}} w\right)_{i}  \right| }{h} + \Vert g' \Vert_{L^\infty ( 
\R )} \,  h^3 
\\
\le & \,
\sum_{j=1}^2 \frac{ \left| \left( {\mathcal{L}} w\right)_{i+h e_j} -  \left( {\mathcal{L}} w\right)_{i}  \right| }{h} + \Vert g' \Vert_{L^\infty (
\R )}  \, h
\end{split}
.
\end{equation}
Here, the first inequality follows by the triangle inequality,
the second inequality can be
deduced by \eqref{ex4:eq:w def}, and the third inequality follows by \eqref{eq:ex4:h tra 0 e 1}.

By noting that ${\mathcal{L}}w_i$ coincides with the function $\tilde{f}(i)$ defined in \eqref{eq:ex1:f def serve per ex4} (in Example \ref{example 1}), the same computations that gave \eqref{eq:ex1:kappa0} now inform us that
\begin{equation}\label{eq:ex4:step5}
\sum_{j=1}^2 \frac{ \left| \left( {\mathcal{L}} w\right)_{i+h e_j} -  \left( {\mathcal{L}} w\right)_{i}  \right| }{h} \le 5 \, h .
\end{equation}

By putting together \eqref{eq:ex4:step3},
\eqref{eq:ex4:condizione k0 su v}, \eqref{eq:ex4:step4}
and \eqref{eq:ex4:step5}, we thus obtain that $f$ (as defined in \eqref{eq:ex4:def f per u}) satisfies \eqref{hp:semilinearitaconspazio} with $L_f^+ (i, u_i) = g'(v_i)$ and 
\begin{equation}\label{eq:ex4:k0+ definition}
\kappa_0^+= \left\lbrace \frac{ \Vert \tilde{v}^{( iv)} \Vert_{L^\infty( \R)} }{6} + \frac{ \Vert g'' \Vert_{L^\infty (
\R )}}{2} \Vert \tilde{v}' \Vert_{L^\infty ( \R )}  \right\rbrace + \left\lbrace   5 + \Vert g' \Vert_{L^\infty ( 
\R )} \right\rbrace .
\end{equation}

From now on we assume that $\tilde{v}$ is strictly increasing (this
is indeed the case of the sine-Gordon equation).
Under this assumption it is clear that, if $h$ is small enough, then $u$ satisfies~\eqref{MONO}, and hence \eqref{hp:graddiversodazero} holds true.
In fact, by recalling the definition of $u$ in \eqref{eq:ex4:def u},
we have that $u$ satisfies~\eqref{MONO} if and only if
\begin{equation*}
v_{(h,h)} - v_{(h,0)} > h^4
\end{equation*}
and hence --- by recalling that $v_{(h,h)}= v_{(0,h)}$ and $v_{(h,0)} = v_{(0,0)}$ (since $v$ depends on $i_2$ only) --- if and only if
\begin{equation*}
\mathcal{D}_2^+ v_{(0,0)} > h^3.
\end{equation*}
For this reason and recalling \eqref{eq:ex4:h tra 0 e 1}, from now on we assume
\begin{equation}\label{eq:ex4: h tra 0 e min tra 1 e D2}
0<h< \min \left\lbrace 1 , \left( \mathcal{D}_2^+ v_{(0,0)} \right)^{1/3} \right\rbrace
\end{equation}
We stress that $\mathcal{D}_2^+ v_{(0,0)}$ is always a number strictly greater than $0$ in light of the assumption that $\tilde{v}$ is strictly increasing.
Moreover, ${\mathcal{D}}_2^+ v_{(0,0)}$ could be explicitly computed in concrete examples in which $\tilde{v}$ is explicitly
given\footnote{
For instance, in the concrete case of the stationary sine Gordon equation \eqref{eq:ex4:sineGordoneq}, 
by recalling \eqref{eq:TAYLORBOUNDSUP FIRSTDERIVATIVE} (with $\overline{v}(x) := 4 \arctan(e^{x_2}) $ and $v_i := 4 \arctan(e^{i_2})$), we easily find that
\begin{equation*}
{\mathcal{D}}_2^+ v_{(0,0)}=  4 \frac{e^\xi}{1+e^{2 \xi}} , 
\end{equation*}
for some $\xi \in ( 0,h )$. Thus, since
$
{\mathcal{D}}_2^+ v_{(0,0)} \ge 4 \frac{e^h}{1+ e^{2 h}  } \ge \frac{2}{e^h} ,
$
by using \eqref{eq:ex4:h tra 0 e 1} we easily obtain that
\begin{equation*}
{\mathcal{D}}_2^+ v_{(0,0)} \ge 2/e .
\end{equation*}
Hence, 
in this case \eqref{eq:ex4: h tra 0 e min tra 1 e D2} would become simply
$
0<h< \left( 2/e \right)^{1/3} .
$
%
%
}.

We now show that $u$ also
satisfies~\eqref{BOULI}, \eqref{K2}
and
\eqref{K33} with $\vartheta^+_\infty := \pi / 2$.
To this aim, we directly compute
$$
{\mathcal{D}}_1^+ u_i =
\begin{cases}
h^3 \quad \, & \text{for } \, i = ( 0 , 0 ) , 
\\
0 \quad \, & \text{otherwise in } \, h\Z^2,
\end{cases}
\qquad
\text{ and }
\qquad
{\mathcal{D}}_2^+ u_i =
\begin{cases}
{\mathcal{D}}_2^+ v_i - h^3 \quad \, & \text{for } \, i = (i_1,0) \text{ with } i_1>0 ,
\\
{\mathcal{D}}_2^+ v_i + h^3  \quad \, & \text{for } \, i = (i_1,-h) \text{ with } i_1>0,
\\
1 \quad \, & \text{otherwise in } \, h\Z^2 ,
\end{cases}
$$
and hence
\begin{eqnarray*}
\vartheta^+_i&=& 
\begin{cases}
\arctan\left( \frac{ {\mathcal{D}}_2^+ v_{i} }{h^3} \right) \quad \, & \text{for } \, i = (0,0),
\\
\frac{ \pi}{2} \quad & \text{otherwise in } \, h\Z^2 ,
\end{cases}\\
\text{and }\qquad
(\rho_i^+)^2 &=& 
\begin{cases}
\left( {\mathcal{D}}_2^+ v_i \right)^2 + h^6 \quad \, & \text{for } \, i = (0,0),
\\
( {\mathcal{D}}_2^+ v_i - h^3 )^2 \quad \, & \text{for } \, i = (i_1,0) \, \text{ with } \, i_1>0 ,
\\
( {\mathcal{D}}_2^+ v_i + h^3 )^2 \quad \, & \text{for } \, i = (i_1,-h) \, \text{ with } \, i_1>0 ,
\\
\left( {\mathcal{D}}_2^+ v_i \right)^2 \quad & \text{otherwise in } \, h\Z^2 .
\end{cases}
\end{eqnarray*}
By recalling \eqref{eq:TAYLORBOUNDSUP FIRSTDERIVATIVE}
and \eqref{eq:ex4:dev ol v con v tilde} we get that
\begin{equation*}
|{\mathcal{D}}_2^+ v_i| \le \Vert \tilde{v}' \Vert_{L^\infty (\R)} ,
\end{equation*}
and hence, by recalling that $0<h<1$, we find that
\begin{equation}\label{eq:ex4:bound per rho+}
\rho_i^+ \le \Vert \tilde{v}' \Vert_{L^\infty (\R)} + h^3 < \Vert \tilde{v}' \Vert_{L^\infty (\R)} + 1 =:S .
\end{equation} 
We also notice that
\begin{eqnarray*}
\kappa_1^+ :=  \sup_{{1\le j\le 2}\atop{i\in h\Z^2}} |
 {\mathcal{D}}_j^+ u_i | \le \Vert \tilde{v}' \Vert_{
L^\infty (\R)} + h^3 < \Vert \tilde{v}' \Vert_{L^\infty (\R)} + 1 =S ,
\end{eqnarray*}
being~$0< h < 1$, and this establishes~\eqref{BOULI}.

We then compute
$$
{\mathcal{D}}_1^+ \vartheta^+_i =
\begin{cases}
\frac{1}{h} \left( \frac{ \pi}{2} - \arctan\left( \frac{ {\mathcal{D}}_2^+ v_{(0,0)} }{h^3} \right) \right) \quad \, & \text{for } \, i = ( 0 , 0 ) , 
\\
\frac{1}{h} \left( \arctan\left( \frac{ {\mathcal{D}}_2^+ v_{(0,0)} }{h^3} \right) -\frac{\pi}{2} 
\right) \quad \, & \text{for } \, i = ( -h , 0),
\\
0 \quad \, & \text{otherwise in } \, h\Z^2,
\end{cases}
$$
$$
{\mathcal{D}}_2^+ \vartheta^+_i =
\begin{cases}
\frac{1}{h} \left( \frac{ \pi}{2} - \arctan\left( \frac{ {\mathcal{D}}_2^+ v_{(0,0)} }{h^3} \right) \right) \quad \, & \text{for } \, i = ( 0 , 0 ) , 
\\
\frac{1}{h} \left( \arctan\left( \frac{ {\mathcal{D}}_2^+ v_{(0,0)}  }{h^3} \right) -\frac{\pi}{2} \right) \quad \, & \text{for } \, i = ( 0, -h)
,\\
0 \quad \, & \text{otherwise in } \, h\Z^2 ,
\end{cases}
$$
$$
{\mathcal{D}}_1^+ \left( {\mathcal{D}}_1^+ \vartheta^+ \right)_{i-h e_j} =
\begin{cases}
\frac{2}{h^2} \left( \frac{ \pi}{2} - \arctan\left( \frac{ {\mathcal{D}}_2^+ v_{(0,0)} }{h^3} \right) \right) \quad \, & \text{for } \, i = ( 0 , 0 ) , 
\\
\frac{1}{h^2} \left( \arctan\left( \frac{ {\mathcal{D}}_2^+ v_{(0,0)} }{h^3} \right) -\frac{\pi}{2} \right) \quad \, & \text{for } \, i = ( \pm h , 0)
,\\
0 \quad \, & \text{otherwise in } \, h\Z^2,
\end{cases}
$$
and
$$
{\mathcal{D}}_2^+ \left( {\mathcal{D}}_2^+ \vartheta^+ \right)_{i-h e_j} =
\begin{cases}
\frac{2}{h^2} \left( \frac{ \pi}{2} - \arctan\left( \frac{ {\mathcal{D}}_2^+ v_{(0,0)} }{h^3} \right) \right) \quad \, & \text{for } \, i = (0, 0) ,
\\
\frac{1}{h^2} \left( \arctan\left( \frac{ {\mathcal{D}}_2^+ v_{(0,0)} }{h^3} \right) -\frac{\pi}{2} \right) \quad \, & \text{for } \, i = (0, \pm h) ,
\\
0 \quad \, & \text{otherwise in } \, h\Z^2 .
\end{cases}
$$
By using \eqref{eq:ex4:bound per rho+} and
recalling that the relations in \eqref{eq:ex1:relations che servono anche in ex4} hold true for $\vartheta^+$,
%
%
we can now directly compute
\begin{equation}
\begin{split}\label{eq:ex4:step per k2}
\kappa_2^+ =\,& \sum_{{1\le j\le 2}\atop{i\in h\Z^2}}
\left( \rho_i^+ \right)^2
\Big(
|{\mathcal{D}}_j^+\vartheta^+_i|\,|{\mathcal{D}}_j^+ ({\mathcal{D}}_j^+ \vartheta^+)_{i-he_j}|+
|{\mathcal{D}}_j^-\vartheta^+_i|\,|{\mathcal{D}}_j^- ({\mathcal{D}}_j^- \vartheta^+)_{i+he_j}|
\Big)\\
\le \, &  12 \, S^2 \, \frac{ \left( \frac{\pi}{2} - \arctan\left( \frac{ {\mathcal{D}}_2^+ v_{(0,0)} }{h^3} \right) \right)^2  }{h^3} ,
\end{split}
\end{equation}
where $S$
is defined in \eqref{eq:ex4:bound per rho+}.
By using~\eqref{eq:arctanestimatePRE} with~$t= {\mathcal{D}}_2^+ v_{(0,0)} / h^3$,
\begin{equation}\label{eq:ex4:arctanestimate}
\left| \mathrm{sgn}\left(\frac{ {\mathcal{D}}_2^+ v_{(0,0)} }{h^3}\right) 
\frac{\pi}{2} - \arctan\left(\frac{ {\mathcal{D}}_2^+ v_{(0,0)} }{h^3}\right)  \right| \le \frac{h^3}{ {\mathcal{D}}_2^+ v_{(0,0)} }.
\end{equation}
{F}rom this and~\eqref{eq:ex4:step per k2},
we obtain that
\begin{equation}\label{eq:ex4:k2}
\kappa_2^+ \le 12 \left( \frac{ S }{ {\mathcal{D}}_2^+ v_{(0,0)} } \right)^2 \, h^3 ,
\end{equation}
that gives~\eqref{K2} and also keeps track of the order of $h$.

In order to verify \eqref{K33}, we notice that, being~$\vartheta_\infty^+= \pi/{2}$, we have that
\begin{equation*}
|\vartheta_i^+ - \vartheta_\infty^+|=
\begin{cases}
  \frac{ \pi}{2} - \arctan\left( \frac{ {\mathcal{D}}_2^+ v_{(0,0)} }{h^3} \right)  \quad \, & \text{for } \, i = ( 0 , 0 ) ,
\\
0  \quad \, & \text{otherwise in } \, h\Z^2 .
\end{cases}
\end{equation*}

Thus, the only nonzero term in the summations
defining $\kappa_3^+$, $\kappa_4^+$, $\kappa_5^+$, $\kappa_6^+$,
$\kappa_7^+$ in \eqref{K33} are those for~$i=(0,0)$.

We start by computing that
\begin{equation*}
\kappa_3^+ = 4 \left[  \left( {\mathcal{D}}_2^+ v_{(0,0)} \right)^2 + h^6  \right] \frac{\left( \frac{\pi}{2} - \arctan\left( \frac{  {\mathcal{D}}_2^+ v_{(0,0)}  }{h^3} \right)  \right)^4}{h^3} ,
\end{equation*}
which, in light of \eqref{eq:ex4:arctanestimate},
gives that
$$
\kappa_3^+ \le \frac{ 4 }{ \left( {\mathcal{D}}_2^+ v_{(0,0)} \right)^2 } \left[ 1 + \left( \frac{h^3}{ {\mathcal{D}}_2^+ v_{(0,0)} } \right)^2 \right] \, h^9
$$
and hence, by recalling \eqref{eq:ex4: h tra 0 e min tra 1 e D2}, that
\begin{equation}\label{eq:ex4:k3}
\kappa_3^+ \le \frac{ 8 }{ \left( {\mathcal{D}}_2^+ v_{(0,0)} \right)^2 } \, h^9 .
\end{equation}

In order to estimate $\kappa_4^+$, $\kappa_6^+$, $\kappa_7^+$ we just need to compute 
\begin{equation*}
\mathcal{L}_1^2 \vartheta_{(0,0)}^+ = \mathcal{L}_2^2 \vartheta_{(0,0)}^+ = \frac{6}{h^4} \left( \arctan\left( \frac{ {\mathcal{D}}_2^+ v_{(0,0)} }{h^3} \right) - \frac{ \pi}{2}  \right) ,
\end{equation*}
and to notice that, by \eqref{eq:ex4:bound per rho+}, we have that
\begin{equation}\label{eq:ex4:ROUGH BOUNDS rho}
|{\mathcal{D}}_j^{\pm}\rho_{(0,0)}^+ | \le \frac{2 \,  S}{h} ,
 \quad {\mbox{ for }} j=1,2 ,
\end{equation}
and
\begin{equation}\label{eq:ex4:ROUGH BOUNDS rho^2}
|{\mathcal{D}}_j^{\pm} \left( \rho_{(0,0)}^+  \right)^2 | 
\le \frac{2 \, S^2}{h} , \quad {\mbox{ for }} j=1,2 .
\end{equation}
We stress that more accurate computations
(similar to those performed in Example \ref{example 1})
could be performed in order to obtain the exact values of 
$|{\mathcal{D}}_j^{\pm}\rho_{(0,0)}^+ |$
and~$\Big|
{\mathcal{D}}_j^{\pm} \left( \rho_{(0,0)}^+  \right)^2 \Big|$.
However, the
bounds in \eqref{eq:ex4:ROUGH BOUNDS rho} and
\eqref{eq:ex4:ROUGH BOUNDS rho^2} are
sufficient in order to verify that $u$ satisfies \eqref{K33} and also to check the optimality of
Theorem~\ref{DGDG}.

By putting together that 
\begin{equation*}
%
%
\rho^+_{(0,0)} = \sqrt{ \left( {\mathcal{D}}_2^+ v_{(0,0)} \right)^2 +h^6 } = {\mathcal{D}}_2^+ v_{(0,0)} \, \sqrt{ 1 + \left( \frac{ h^3}{ {\mathcal{D}}_2^+ v_{(0,0)} } \right)^2 }
\end{equation*}
and \eqref{eq:ex4: h tra 0 e min tra 1 e D2}, we obtain that
\begin{equation}\label{eq:ex4:STIMA rho+ in (0,0)}
\rho^+_{(0,0)} \le \sqrt{2} \ {\mathcal{D}}_2^+ v_{(0,0)} .
\end{equation}
By using~\eqref{eq:ex4:ROUGH BOUNDS rho}
and~\eqref{eq:ex4:STIMA rho+ in (0,0)}, we
now compute that
\begin{equation*}
\kappa_4^+ \le 
\left( \sqrt{ 2 } \ {\mathcal{D}}_2^+ v_{(0,0)} \right)
\, \left( \frac{2 \, S}{h} \right) \,
4 \, \frac{ \left( \frac{\pi}{2} - \arctan\left( \frac{ {\mathcal{D}}_2^+ v_{(0,0)} }{h^3} \right)  \right)^3}{h^2} .
\end{equation*}
{F}rom this and \eqref{eq:ex4:arctanestimate},
we deduce that
%
%
%
\begin{equation}\label{eq:ex4:k4}
\kappa_4^+ \le 8\sqrt{2} \, \frac{S}{ \left( {\mathcal{D}}_2^+ v_{(0,0)} \right)^2 } \,
h^6 .
\end{equation}

By using \eqref{eq:ex4:STIMA rho+ in (0,0)}, we also find that
\begin{equation*}
\kappa_5^+ \le \kappa_0^+ \, \sqrt{ 2 } \ {\mathcal{D}}_2^+ v_{(0,0)} \,  \left( \frac{\pi}{2} - \arctan\left( \frac{1}{h^3} \right)  \right)  ,
\end{equation*}
where $\kappa_0^+$ is that obtained in \eqref{eq:ex4:k0+ definition}.
Thus, by \eqref{eq:ex4:arctanestimate}, we see that
\begin{equation}\label{eq:ex4:k5}
\kappa_5^+ \le \kappa_0^+ \sqrt{2} \ h^3 .
\end{equation}

By using the inequality $\left( \rho^+_{(0,0)} \right)^2
 \le 2 \left( {\mathcal{D}}_2^+ v_{(0,0)} \right)^2 $ ---
which follows by \eqref{eq:ex4:STIMA rho+ in (0,0)} --- and \eqref{eq:ex4:ROUGH BOUNDS rho^2}, we then compute
\begin{equation*}
\kappa_6^+ \le  \left[ \left( \frac{2 \, S^2 }{h} \right) \, 4 \, \frac{\left( \frac{\pi}{2} - \arctan\left( \frac{ {\mathcal{D}}_2^+ v_{(0,0)} }{h^3} \right)  \right)}{h} + 24 \left( {\mathcal{D}}_2^+ v_{(0,0)} \right)^2 \, \frac{\left( \frac{\pi}{2} - \arctan\left( \frac{ {\mathcal{D}}_2^+ v_{(0,0)} }{h^3} \right)  \right)}{h^3}  \right]  \left( \frac{\pi}{2} - \arctan\left( \frac{ {\mathcal{D}}_2^+ v_{(0,0)} }{h^3} \right)  \right),
\end{equation*}
and hence, by \eqref{eq:ex4:arctanestimate},
\begin{equation}\label{eq:ex4:k6}
\kappa_6^+ \le \left[ \frac{ 8 \, S^2 }{ \left( {\mathcal{D}}_2^+ v_{(0,0)} \right)^2 } \, h + 24 \right] \, h^3 .
\end{equation}

Finally, by using \eqref{eq:ex4:ROUGH BOUNDS rho} we find that
\begin{equation*}
\kappa_7^+ \le
\left( \frac{2 \, S }{ h } \right)^2 \left[ 4 \, \frac{ \left( \frac{\pi}{2} - \arctan\left( \frac{ {\mathcal{D}}_2^+ v_{(0,0)} }{h^3} \right)  \right)}{h} \right] \, \left( \frac{\pi}{2} - \arctan\left( \frac{ {\mathcal{D}}_2^+ v_{(0,0)} }{h^3} \right)  \right) ,
\end{equation*}
and hence, by \eqref{eq:ex4:arctanestimate}, we obtain that
\begin{equation}\label{eq:ex4:k7}
\kappa_7^+ \le \left[ \frac{16 \, S^2 }{ \left( {\mathcal{D}}_2^+ v_{(0,0)} \right)^2 } \right] \, h^3 .
\end{equation}

In light of \eqref{eq:ex4:h tra 0 e 1}, inequalities
\eqref{eq:ex4:k3}, \eqref{eq:ex4:k4}, \eqref{eq:ex4:k5},
\eqref{eq:ex4:k6} and~\eqref{eq:ex4:k7} give \eqref{K33}.

All in all, we have that~$u$ and $f$ satisfy all the assumptions of
Theorem \ref{DGDG} and hence \eqref{FORM} holds true.
Nevertheless~$u$ is not one-dimensional, and \eqref{eq:esempi div da zero} holds true.

We notice that the quantity in the left-hand side of~\eqref{eq:esempi div da zero} can be explicitly computed. Here, as usual, in order to check the optimality of Theorem \ref{DGDG}, we just notice that
\begin{eqnarray*}
 \sum\limits_{{1\le j\le 2}\atop{i\in h\Z^2}}
(\rho_i^+)^2 \;\Big(
|{\mathcal{D}}_j^+\vartheta^+_i|^2+|{\mathcal{D}}_j^-\vartheta^+_i|^2
\Big) 
& \ge &
 (\rho_{(-h,0)}^+)^2 \; \sum\limits_{j=1}^2 \Big(
|{\mathcal{D}}_{j}^+\vartheta^+_{(-h,0)}|^2+|{\mathcal{D}}_j^-\vartheta^+_{(-h,0)}|^2
\Big)
\\
& = &  | {\mathcal{D}}_2^+ v_{ (-h,0) } |^2 | {\mathcal{D}}_1^+\vartheta^+_{(-h,0)}|^2 
 =  | {\mathcal{D}}_2^+ v_{ (0,0) } |^2 \frac{ \left( \frac{\pi}{2} - \arctan\left( \frac{  {\mathcal{D}}_2^+ v_{ (0,0) }  }{h^3} \right) \right)^2  }{h^2}
.
\end{eqnarray*}
Here, in the last equality we used the explicit value of ${\mathcal{D}}_1^+\vartheta^+_{(-h,0)}$ computed before, and the equality ${\mathcal{D}}_2^+ v_{ (0,0) } = {\mathcal{D}}_2^+ v_{ (-h,0) } $ which holds true since $v$ is a function depending on the $i_2$-variable only.

%
%

By recalling~\eqref{eq:ex4: h tra 0 e min tra 1 e D2}, we can take~$t = {\mathcal{D}}_2^+ v_{ (0,0) }/h^3$ in \eqref{eq:ex1:inequality serve per ex4 arctan da sotto}, obtaining that
\begin{equation*}
\left| \mathrm{sgn}\left(\frac{ {\mathcal{D}}_2^+ v_{ (0,0) } }{h^3}\right) \frac{\pi}{2} - \arctan
\left(\frac{ {\mathcal{D}}_2^+ v_{ (0,0) } }{h^3}\right) \right| 
\ge \frac{4 }{\pi} \, \frac{ \left| h \right|^3 }{ {\mathcal{D}}_2^+ v_{ (0,0) } }.
\end{equation*}
{F}rom this we thus get that
\begin{equation}\label{eq:ex4:stima da sotto}
\sum\limits_{{1\le j\le 2}\atop{i\in h\Z^2}}
(\rho_i^+)^2 \;\Big(
|{\mathcal{D}}_j^+\vartheta^+_i|^2+|{\mathcal{D}}_j^-\vartheta^+_i|^2
\Big) 
 \ge
 \frac{4^2}{\pi^2} \, h^4.
\end{equation}

On the other hand, \eqref{eq:ex4:k2}, \eqref{eq:ex4:k3},
\eqref{eq:ex4:k4}, \eqref{eq:ex4:k5}, \eqref{eq:ex4:k6},
\eqref{eq:ex4:k7} and~\eqref{eq:ex4:h tra 0 e 1}
give that the right-hand side of \eqref{FORM} satisfies
\begin{equation}\label{eq:ex4:stimadasopra}
C \, h \le c_1 h^4 ,
\end{equation}
where $C$ is the quantity defined in \eqref{eq:constant C}, and 
$$
c_1:= 4 \left\lbrace \frac{1}{ \left( {\mathcal{D}}_2^+ v_{(0,0)} \right)^2 } \left[ 36 \, S^2 + 16 \, e^{2 \pi} \left( 1 + \sqrt{2} \,  S  \right) \right] + 2 \sqrt{2} \, \kappa_0^+ +24  \right\rbrace ,
$$
where $\kappa_0^+$ and $S$ are those defined in \eqref{eq:ex4:k0+ definition} and \eqref{eq:ex4:bound per rho+}.
 
Thus, by putting together \eqref{FORM}, \eqref{eq:ex4:stima da sotto}
and \eqref{eq:ex4:stimadasopra}, it is clear that, in the
formal limit as~$
h \searrow 0$, the left-hand side and the right-hand side of \eqref{FORM}
are both of the order of $h^4$. Thus, also this general example confirms the optimality of \eqref{FORM}.

We stress that $c_1$ is a positive constant only depending on (a lower bound on) ${\mathcal{D}}_2^+ v_{(0,0)}$ and (upper bounds on) $\Vert \tilde{v}' \Vert_{L^\infty (\R)}$, $\Vert \tilde{v}^{(iv)} \Vert_{L^\infty (\R)}$, $\Vert g' \Vert_{L^\infty ( 
\R )}$, and $\Vert g'' \Vert_{L^\infty ( \R)}$. 
We recall that, in concrete examples --- such as, e.g.,
in the case of the sine-Gordon equation \eqref{eq:ex4:sineGordoneq} ---, these bounds can be explicitly obtained and $c_1$ is just a universal constant.
}
\end{example}

\begin{appendix}

\section{The identity in~\eqref{k8i-know-CONTIN}
as a formal limit of the one in~\eqref{k8i-know}}\label{APPEDK894}

In this appendix, we discuss, merely at a formal
level, how the continuous identity in~\eqref{k8i-know-CONTIN}
may be understood as a suitable limit of the discrete identity
in~\eqref{k8i-know} (and we believe that this observation
is interesting, since it relates the identity in~\eqref{k8i-know-CONTIN}, which
is classical and well understood,
with the one in~\eqref{k8i-know}, which is, as far as we can tell,
completely new in the literature).

Establishing a rigorous framework to relate~\eqref{k8i-know}
and~\eqref{k8i-know-CONTIN} goes beyond the goals of this paper
and would rather fit into the more comprehensive and ambitious goal
of scrupulously connect discrete and continuous models,
hence our arguments will rely on formal, yet solid, approximations.
To deduce~\eqref{k8i-know-CONTIN} from~\eqref{k8i-know}
we fix~$(x_1,x_2)\in\R^2$. Without loss of generality, we suppose that~$x_1>0$.
Given~$h>0$, we take~$k\in\N$ such as~$hk$ is as close as possible to~$x_1$,
for instance by taking~$k$ such that~$hk\le x_1 < h(k+1)$.
Similarly, we take~$m\in\Z$ such that~$hm\le x_2<h(m+1)$.
As a matter of fact, to make the computation as transparent as possible,
we simply suppose that~$x_1=1$ and $x_2=0$, and also that~$\frac1h\in\N$,
in which case we have that~$k=\frac1h$ and~$m=0$.
In this way, we can write~\eqref{k8i-know} in the simpler form
\begin{equation}\label{k8i-know-SIM}
u_{(1,0)}=\sum_{j=0}^{1/h} {\binom {1/h}{j}}
c^j\,(1-c)^{\frac1h-j}\,\widetilde{u}_{hj},
\end{equation}
where~$c$ is short for~$c^+$. Similarly, we can state~\eqref{k8i-know-CONTIN}
in its simpler version given by
\begin{equation}\label{k8i-know-CONTIN-SIM}
u(1,0)=\widetilde{u}(c).
\end{equation}
Since the role of the point~$(x_1,x_2)=(1,0)$ is somehow
arbitrary, we focus on the relation between~\eqref{k8i-know-SIM}
and~\eqref{k8i-know-CONTIN-SIM}. 
Furthermore,
we suppose for simplicity that~$c\in(0,1)$
and use the Binomial Theorem to observe that
$$ 1=( c+(1-c))^k=\sum_{j=0}^k{\binom {k}{j}}
c^j\,(1-c)^{k-j}.$$
On this account, we can write~\eqref{k8i-know-SIM}
in the form
\begin{equation}\label{k8i-know-SIM-2}
\sum_{j=0}^{1/h} {\binom {1/h}{j}}
c^j\,(1-c)^{\frac1h-j}\,\big(\widetilde{u}_{hj}-\widetilde{u}(c)\big)=u_{(1,0)}-
\widetilde{u}(c).
\end{equation}
With a slight abuse of notation, we also identify the
functions in the discrete setting with the corresponding ones in the continuum
without further notice.
With this, up to replacing~$u_{(i_1,i_2)}$ and~$\widetilde{u}_{i}$
with~$v (i_1,i_2):=
u_{(i_1,i_2)}-\widetilde{u}(c)$
and~$\widetilde{v} (i) := \widetilde{u}_{i}-
\widetilde{u}(c)$
respectively, we can replace~\eqref{k8i-know-SIM-2}
by
\begin{equation}\label{k8i-know-SIM-3}
\sum_{j=0}^{1/h} {\binom {1/h}{j}}
c^j\,(1-c)^{\frac1h-j}\,\widetilde{v}(hj) = v(1,0) ,
\end{equation}
with the additional assumption that
\begin{equation}\label{KS:odw398tuhrfnstekll05}
\widetilde{v}(c)=0.
\end{equation}
The same setting~$v(x_1,x_2):=u(x_1,x_2)-\widetilde{u}(c)$
reduces~\eqref{k8i-know-CONTIN-SIM} to
\begin{equation}\label{k8i-know-CONTIN-SIM-3}
v(1,0)=0,
\end{equation}
therefore we focus on discussing how~\eqref{k8i-know-SIM-3}
formally implies~\eqref{k8i-know-CONTIN-SIM-3} under condition~\eqref{KS:odw398tuhrfnstekll05}.

To this end, it is convenient to reduce
to ``large indexes''~$j$ in~\eqref{k8i-know-SIM-3}, in view of the following observation.
Let~$M>0$. 
Since
$$
e^j = \sum_{i=0}^{\infty} \frac{j^i}{ i!} \ge \frac{j^j}{ j!} ,
$$
we have that
\begin{equation*}
{k \choose j}\leq \frac{k^j}{ j! } \le
\left({\frac {ek}{j}}\right)^{j}. 
\end{equation*}
As a result,
if~$1\le j\le M$ we have that
$$ {k \choose j}\leq 
\left({\frac {ek}{j}}\right)^{j}\le (ek)^M.$$
Consequently, assuming~$\widetilde{v}$ bounded,
\begin{equation}\label{CBciel}\begin{split}
\sum_{j=0}^{M} {\binom {1/h}{j}}
c^j\,(1-c)^{\frac1h-j}\,|\widetilde{v}(hj)|\le\,& (1-c)^{\frac1h}\,\widetilde v(0)+
\frac{e^M}{h^M} (1-c)^{\frac1h}
\sum_{j=1}^{M} \left(\frac{c}{1-c}\right)^j\,|\widetilde{v}(hj)|\\
\le\,& \frac{C\,e^M}{h^M} (1-c)^{\frac1h}
\end{split}
\end{equation}
for some~$C>0$ independent of~$h$, and the latter quantity
in~\eqref{CBciel} is infinitesimal as~$h\searrow0$.
For this reason, we can focus in~\eqref{k8i-know-SIM-3}
on the ``large indexes''~$j\ge M$.
Similarly, given the symmetry properties of the binomial
coefficients, we can focus on the case in which~$\frac1h-j$ is large.

As a result, it is appropriate to use
Stirling's 
Formula 
$$
{1/h \choose j}\sim {\sqrt {1/h \over 2\pi j\,\left(\frac1h-j\right)}}\; 
{(1/h)^{\frac1h} \over j^{j}\,\left(\frac1h-j\right)^{\frac1h-j}}=
{\sqrt {1\over 2\pi j(1-hj)}}\; 
{1 \over(h j)^{j}\,(1-hj)^{\frac1h-j}}
$$
and bound the absolute value of the left hand side of~\eqref{k8i-know-SIM-3}
by
\begin{equation}\label{ruot83irj3grg8berfvjndngir}
\sum_{j=M}^{\frac1h-M} 
{\sqrt {1\over j(1-hj)}}\;
\left(\frac{c}{hj}\right)^j\,\left(\frac{1-c}{1-hj}\right)^{\frac1h-j}\,|\widetilde{v}(hj)|,
\end{equation}
up to multiplicative constants that we neglect for the sake of simplicity.

Now, we will formally replace some quantities in~\eqref{ruot83irj3grg8berfvjndngir}
with their asymptotic counterparts as~$h\searrow0$.
For instance, observing that
$$ \frac1h\int_{hj}^{h(j+1)}|\widetilde{v}(t)|\,dt-|\widetilde{v}(hj)|\longrightarrow0
\qquad{\mbox{ as }}h\searrow0,$$
we formally replace~\eqref{ruot83irj3grg8berfvjndngir}
by
\begin{equation}\label{ruot83irj3grg8berfvjndngir2}
\frac1h\sum_{j=0}^{\frac1h-1} \int_{hj}^{h(j+1)}
{\sqrt {1\over j(1-hj)}}\;
\left(\frac{c}{hj}\right)^j\,\left(\frac{1-c}{1-hj}\right)^{\frac1h-j}\,|\widetilde{v}(t)|\,dt.
\end{equation}
In addition,
if~$t\in[hj,h(j+1))$ we have that
$$ t-hj\in [0,h]\longrightarrow0
\qquad{\mbox{ as }}h\searrow0,$$
whence we formally replace~\eqref{ruot83irj3grg8berfvjndngir2}
with the expression
\begin{equation}\label{ruot83irj3grg8berfvjndngir3}\begin{split}&
\frac1{\sqrt{h}}\sum_{j=0}^{\frac1h-1} \int_{hj}^{h(j+1)}
\frac1{{\sqrt{t(1-t)}}}\;
\left(\frac{c}{t}\right)^\frac{t}{h}\,\left(\frac{1-c}{1-t}\right)^{\frac{1-t}h}\,|\widetilde{v}(t)|\,dt
\\=\;&
\frac1{\sqrt{h}} \int_{0}^{1}
\frac1{{\sqrt{t(1-t)}}}\;
\left(\frac{c}{t}\right)^\frac{t}{h}\,\left(\frac{1-c}{1-t}\right)^{\frac{1-t}h}\,|\widetilde{v}(t)|\,dt
\\=\;&
\frac1{\sqrt{h}} \int_{0}^{1}
\frac{\big(\phi(t)\big)^{\frac1h}}{{\sqrt{t(1-t)}}}\;
|\widetilde{v}(t)|\,dt
,\end{split}
\end{equation}
where
$$\phi(t):=
\left(\frac{c}{t}\right)^t\,\left(\frac{1-c}{1-t}\right)^{{1-t}}.$$
We remark that~$\phi(c)=1$,
$$\phi'(t)=\frac{1-c}{1-t}
\,\left(\frac{1-t}{1-c}\right)^t \left(\frac{c}t\right)^t
\log\frac{c(1-t)}{t(1-c)}
$$
and therefore the only critical point of~$\phi$ is~$t=c$.

We also notice that
$$\phi''(t)=
\left(\frac{1-c}{1-t}\right)^{1 - t} \left(\frac{c}{t}\right)^t \left(
\log^2 \frac{1-c}{1-t}+ \log^2\frac{c}{t} - 2 \log\frac{c}t \log\frac{1-c}{1-t} -\frac1{1-t} - \frac1t\right)
$$
and as a consequence
$$\phi''(c)=-\frac1c-\frac1{1-c}=-\frac{1}{c(1-c)}<0.$$
This gives that~$t=c$ is a maximum for~$\phi$ and there exists
small but strictly positive~$\delta_0$ and~$\delta_1$
such that
\begin{eqnarray*}&&
\phi(t)\le 1-\frac{(t-c)^2}{2c(1-c)}\qquad{\mbox{ for all }}t\in
(c-\delta_0,c+\delta_0)\subset(0,1),\\
&&\phi(t)\le 1-\delta_1
\qquad{\mbox{ for all }}t\in(0,1)\setminus (c-\delta_0,c+\delta_0).
\end{eqnarray*}
Hence, noticing that
\begin{equation}\label{OJSHNDK i9023rutgope}
\frac1{\sqrt{h}} \int_{(0,1)\setminus (c-\delta_0,c+\delta_0)}
\frac{\big(\phi(t)\big)^{\frac1h}}{{\sqrt{t(1-t)}}}\;
|\widetilde{v}(t)|\,dt\le
\frac{C\,(1-\delta_1)^{\frac1h}}{\sqrt{h}} \int_{(0,1)}
\frac{dt}{{\sqrt{t(1-t)}}}\le\frac{C\,(1-\delta_1)^{\frac1h}}{\sqrt{h}},
\end{equation}
up to renaming the quantity~$C>0$ independently on~$h$,
and remarking that the latter term in~\eqref{OJSHNDK i9023rutgope}
is infinitesimal as~$h\searrow0$,
we can bound~\eqref{ruot83irj3grg8berfvjndngir3} by
\begin{equation}\label{JOSNK-LS8ND}
\frac1{\sqrt{h}} \int_{c-\delta_0}^{c+\delta_0}
\frac{\big(\phi(t)\big)^{\frac1h}}{{\sqrt{t(1-t)}}}\;
|\widetilde{v}(t)|\,dt,
\end{equation}
up to adding an infinitesimal term.

In turn, we use the transformation~$\tau:=\frac{t-c}{\sqrt{h}}$
to see that the quantity in~\eqref{JOSNK-LS8ND}
is controlled by
\begin{eqnarray*}&&
\frac{C}{\sqrt{h}} \int_{c-\delta_0}^{c+\delta_0}\left(
1-\frac{(t-c)^2}{2c(1-c)}\right)^{\frac1h}
\frac{|\widetilde{v}(t)|\,dt}{{\sqrt{t(1-t)}}}
\le
\frac{C}{\sqrt{h}} \int_{c-\delta_0}^{c+\delta_0}\left(
1-\frac{(t-c)^2}{2c(1-c)}\right)^{\frac1h}\,|\widetilde{v}(t)|
\,dt
\\&&\qquad=\frac{C}{\sqrt{h}} \int_{c-\delta_0}^{c+\delta_0}
\exp
\left(\frac1h\,\log\left(
1-\frac{(t-c)^2}{2c(1-c)}\right)\right)
\,|\widetilde{v}(t)|\,dt\le
\frac{C}{\sqrt{h}} \int_{c-\delta_0}^{c+\delta_0}
\exp
\left(-\frac{(t-c)^2}{4c(1-c)h}\right)
\,|\widetilde{v}(t)|
\,dt\\&&\qquad=
C\int_{-\delta_0/\sqrt{h}}^{\delta_0/\sqrt{h}}
\exp
\left(-\frac{\tau^2}{4c(1-c)}\right)\,
|\widetilde{v}(c+\sqrt{h}\tau)|
\,d\tau
\end{eqnarray*}
up to renaming~$C$ at each step of the calculation.
Thus, as~$h\searrow0$, we formally obtain the bound
$$ C\int_{-\infty}^{+\infty}
\exp
\left(-\frac{\tau^2}{4c(1-c)}\right)\,
|\widetilde{v}(c)|
\,d\tau,$$
which is equal to zero, thanks to~\eqref{KS:odw398tuhrfnstekll05}.

These considerations imply that
the formal limit as~$h\searrow0$ of~\eqref{k8i-know-SIM-3}
is~$0=v{(1,0)}$,
which is~\eqref{k8i-know-CONTIN-SIM-3}.

\end{appendix}

\section*{Acknowledgments}

The authors are members of INdAM and AustMS and
are supported by the Australian Research Council
Discovery Project DP170104880 NEW ``Nonlocal Equations at Work''.
The first author is supported by
the Australian Research Council DECRA DE180100957
``PDEs, free boundaries and applications''. 
The second and third authors are supported by
the Australian Laureate Fellowship
FL190100081
``Minimal surfaces, free boundaries and partial differential equations''.

\end{document}